\documentclass[11pt,twoside,reqno]{amsart}
\usepackage[T1]{fontenc}
\usepackage{graphicx}
\usepackage{amsmath,amssymb,amsthm,mathrsfs}
\usepackage[a4paper,margin=1in]{geometry}
\usepackage{microtype,enumitem,booktabs,array,xurl}
\usepackage{fancyhdr}
\fancypagestyle{plain}{%
  \fancyhf{}%
  \fancyfoot[C]{\thepage}%
}
\usepackage[colorlinks=true,linkcolor=blue,citecolor=blue,urlcolor=blue]{hyperref}
\hypersetup{pdftitle={Finite-time singularity formation for periodic IPM},pdfsubject={Research draft; proposed analytical proof on the two-dimensional torus}}
\numberwithin{equation}{section}
\newtheorem{theorem}{Theorem}[section]
\newtheorem{proposition}[theorem]{Proposition}
\newtheorem{lemma}[theorem]{Lemma}
\newtheorem{corollary}[theorem]{Corollary}
\theoremstyle{remark}

\DeclareMathOperator{\diver}{div}
\DeclareMathOperator{\supp}{supp}
\DeclareMathOperator{\tr}{tr}
\DeclareMathOperator{\cof}{cof}
\newcommand{\R}{\mathbb R}
\newcommand{\T}{\mathbb T}
\newcommand{\cT}{\mathcal T}
\newcommand{\cP}{\mathbb P}
\newcommand{\id}{\mathrm{id}}
\newcommand{\norm}[1]{\left\lVert #1\right\rVert}
\newcommand{\av}[1]{\langle #1\rangle_\omega}
\newcommand{\eps}{\varepsilon}
\newcommand{\G}{\mathcal G}
\newcommand{\A}{\mathcal A}
\allowdisplaybreaks[2]
\title[Finite-time blowup for 2D unforced IPM]{Finite-time blowup for 2D unforced IPM}

\author[Mimi Dai]{Mimi Dai}
\address{Department of Mathematics, Statistics and Computer Science,
University of Illinois at Chicago, Chicago, IL 60607, USA}
\email{mdai@uic.edu}

\thanks{The author is partially supported by the NSF grant DMS--2308208 and Simons Foundation.}

\begin{document}
\begin{abstract}
We construct an odd smooth initial density for the unforced incompressible
porous media (IPM) equation on $\T^2$ whose solution develops a finite-time
blowup. The components $\partial_{x_2}\rho(t,0)$ and
$\partial_{x_1}u_1(t,0)$ tend to $+\infty$ as the maximal smooth existence
time is approached, while the density and velocity remain bounded in
$L^2$. Starting from a stationary solution,
we iterate the angular amplification of localized oscillations with rapidly growing frequencies, 
adapting the mechanism of C\'ordoba and Mart\'inez-Zoroa~\cite{CMZ} for forced IPM.
Each perturbation is constructed from time zero and creates the geometry
configuration for the next iteration stage. A one-sided upper bound for the pressure
Hessian is applied to control backward preparation and initial perturbations, following an
idea from the Euler construction of OpenAI~\cite{OpenAIEuler}.
\end{abstract}
\maketitle
\begingroup
\setlength{\parskip}{0pt}
\tableofcontents
\endgroup

\section{Introduction}\label{sec:introduction}

We consider the incompressible porous media (IPM) equation on
$\T^2=(\R/2\pi\mathbb Z)^2$,
\begin{equation}\label{eq:ipm}
\rho_t+u\cdot\nabla\rho = 0,\qquad
u = -\nabla\pi-\rho e_2,\qquad \diver u = 0,
\end{equation}
where $\rho$ denotes the density, $u$ the velocity field, and $\pi$ the periodic
pressure, normalized to have zero mean. Darcy's law determines the
velocity $u=\cT\rho$ for an operator $\cT$ associated with Fourier symbol
\begin{equation}\label{eq:periodicsymbol}
m(k) = \frac{k_2k}{|k|^2}-e_2\quad(k\in\mathbb Z^2\setminus\{0\}),
\qquad m(0) = -e_2.
\end{equation}
We note $m$ is homogeneous of degree zero, satisfying $m(\xi)\cdot\xi=0$ and $|m(\xi)|\leq1$.

Local existence of smooth solution, uniqueness, and continuation criteria for \eqref{eq:ipm} were established
by C\'ordoba, Gancedo, and Orive \cite{CGO}; and more broad background on the 2D IPM equation can be found in references therein.
This paper investigates finite-time blowup scenarios from smooth periodic initial density. The work is heavily inspired by C\'ordoba and Mart\'inez-Zoroa's 
construction of finite-time blowup for forced 2D IPM in \cite{CMZ}. 
The construction begins with a stationary solution and adds localized
oscillations with higher and higher frequencies, whose initial increments are small but not necessarily zero,
while their subsequent amplification occurs in a bounded time interval.

Our main result is:
\begin{theorem}\label{thm:blowup}
There exists an odd density $\rho_{\rm in}\in C^\infty(\T^2)$ such that the maximal lifespan of the smooth solution of \eqref{eq:ipm} denoted by 
$T_*(\rho_{\rm in})$ satisfies
\[
0<T_*(\rho_{\rm in})<1.
\]
In particular, we have
\[
\partial_{x_2}\rho(t,0)\to+\infty,\qquad
\partial_{x_1}u_1(t,0)\to+\infty
\quad\text{as }t\to T_*(\rho_{\rm in}).
\]
Moreover, 
\[\|\rho(t)\|_{L^2}=\|\rho_{\rm in}\|_{L^2}, \qquad  \|u(t)\|_{L^2}\leq C, \qquad \forall t\in[0, T_*(\rho_{\rm in})]\]
for a constant $C>0$.
\end{theorem}

\subsection{Relevant prior work}
Finite-time singularities for unforced IPM have been established in different energy or regularity settings. C\'ordoba, Gancedo, and
Orive \cite{CGO} constructed singular solutions in a special infinite-energy
class. Collot, Prange, and Tan \cite{CPT} proved stability of a self-similar
singularity within an infinite-energy reduction on an unbounded strip,
including solutions periodic in one coordinate. Dembski \cite{Dembski}
constructed singular solutions from Lipschitz initial density in a wedge,
with density vanishing on the physical boundary and smooth away from the
corner. These results concern either infinite energy or a boundary domain
with lower initial regularity.

For the unforced equation on $\T^2$, Kiselev and Yao
\cite{KY} constructed smooth initial data with unbounded
Sobolev growth under the assumption that the evolution remains globally
smooth, and established instability of stratified steady states. Their conclusions do not force breakdown at a finite time.
At critical regularity, Bianchini, C\'ordoba, and Mart\'inez-Zoroa
\cite{BCM} established nonexistence and strong ill-posedness in
$H^2(\R^2)$ near the stable profile $-x_2$. Their construction exhibits
instantaneous loss of $H^2$ regularity from small $H^2$ perturbations.

The closest smooth-data singularity constructions involve an external
source. C\'ordoba and Mart\'inez-Zoroa \cite{CMZ} obtained finite-time blowup
on $\R^2$ from smooth compactly supported initial density, with a compactly
supported source in $L^\infty_tC^\infty_x$.  Very recently, Alp\"oge, Buckmaster, and Coiculescu
\cite{ABC} subsequently obtained blowup on $\T^2$ from smooth odd
initial density with a source in $C^\infty([0,1]\times\T^2)$. 

In our earlier work for the forced electron megnetohydrodynamics (MHD) \cite{DaiEMHD}, we applied the construction strategy 
of \cite{CMZ} to show the existence of a smooth solution that blows up in finite time. 
In \cite{DaiEMHD}, a magnetic profile generating a nearly hyperbolic
current near the origin is augmented by increasingly oscillatory
perturbations. At each stage, the perturbation is prescribed at the terminal
time and recovered by solving an auxiliary linear transport--stretching
equation backward along the characteristic flow generated by the negative of a modified current
\cite[Section~2.4, Lemma~2.8]{DaiEMHD}. Each new perturbation generates the
current profile for the following stage. The present construction for the 2D IPM shares
this iterative selection of terminal geometry and earlier perturbation
histories. In \cite{DaiEMHD}, temporal cutoffs activate successive
perturbations, and their errors enter the external force. Here the entire
residual must be canceled by an additional perturbation while the initial
increments remain small.

OpenAI \cite{OpenAIEuler} recently announced a proof of finite-time blowup for the 3D Euler equation without external forcing,
where they mentioned that the construction was influenced by C\'ordoba and Mart\'inez-Zoroa's work \cite{CMZ-Euler}
for the 3D Euler equation with external forcing. We postpone to discussing the construction in Subsection~\ref{sec:eulercomparison}.

The current paper is inspired by these aforementioned works in different aspects. In particular, our construction inherits and adapts the angular amplification mechanism of
C\'ordoba and Mart\'inez-Zoroa \cite{CMZ}. A nearly vertical background
gradient amplifies a localized oscillation and changes its phase direction,
preparing the geometry for the next stage. However, additional ingredients are required to 
realize this mechanism for unforced solutions.
Section~\ref{sec:heuristics} describes these mechanisms. In particular,
Subsection~\ref{sec:forcedcomparison} explains how preparing small initial
increments and controlling the exact correction allow us to remove the
external source while preserving amplification.

Section~\ref{sec:local} records the local theory and transport identities.
Section~\ref{sec:heuristics} introduces the perturbations and their time
intervals, followed by comparisons with constructions for the forced IPM
and with the Euler construction by OpenAI. The principal
amplification and its persistence are established in
Section~\ref{sec:amplification}. Sections~\ref{sec:insertion}
and~\ref{sec:realization} contain the quantitative perturbation argument.
Section~\ref{sec:induction} carries out the iteration, and completes the limiting argument to prove Theorem \ref{thm:blowup}.

\textbf{Notation.}
Unless otherwise indicated, spatial norms are taken on $\T^2$. We abbreviate $\|\cdot\|_{L^2}$ and $\|\cdot\|_{L^\infty}$
to $\|\cdot\|_{2}$ and $\|\cdot\|_{\infty}$, respectively. 
The symbols $I$ and $\id$ denote the identity matrix and the identity map.
Constants $C,c>0$ may change from line to line; dependence on parameter is indicated by a subscript. 

We use $r$ for density perturbations and their expansion coefficients,
$R_0,R_*,R_j$ for material radii, and $\gamma$ for the dimensionless
Gevrey radius. 
These quantities are defined explicitly in the relevant sections.

\section{Local theory and elementary identities}\label{sec:local}
We first establish some standard local estimates which will be used to correct the source term and to
compare the successive solutions. We then express the transport equation and
Darcy's law along a particle map. 

The symbol \eqref{eq:periodicsymbol} implies
\[
\norm{\cT f}_{H^s} \leq \norm f_{H^s},\qquad
\partial_{x_i}\cT f = \cT\partial_{x_i}f
\]
for every $s\in\R$. In view of $\nabla\cdot u=0$, we have the standard estimate for equation \eqref{eq:ipm}
\begin{equation}\label{eq:tametransport}
\frac{d}{dt}\norm\rho_{H^m}
\leq C_m\bigl(\norm{\nabla u}_\infty+\norm{\nabla\rho}_\infty\bigr)
\norm\rho_{H^m}, \qquad \forall m\geq 3.
\end{equation}
The multiplier bound above and the embedding
$H^3(\T^2)\hookrightarrow W^{1,\infty}(\T^2)$ therefore imply
\[
\frac{d}{dt}\norm\rho_{H^m}
\leq C_m\norm\rho_{H^3}\norm\rho_{H^m}.
\]
In particular, the $H^3$ norm satisfies a closed estimate.

\begin{lemma}[Local solution and comparison]\label{lem:local}
Smooth initial data $\rho_{\rm{in}}$ admit a unique smooth solution $\rho$ to the IPM system \eqref{eq:ipm} on $[0,T]$ for some $T>0$
depending only on $\|\rho_{\rm{in}}\|_{H^3}$. The solution satisfies
\begin{equation}\label{eq:energy}
\norm{\rho(t)}_2 = \norm{\rho(0)}_2,\qquad
\norm{u(t)}_2 \leq \norm{\rho(0)}_2.
\end{equation}
Moreover, if $\rho+w$ is another smooth solution on
$[0,T']$ with $T'\leq T$, then
\begin{equation}\label{eq:differenceestimate}
\frac{d}{dt}\norm w_{H^3}
\leq C\norm\rho_{H^4}\norm w_{H^3}
+C\norm w_{H^3}^2.
\end{equation}
\end{lemma}
\begin{proof}
The local existence of smooth solution is standard, for instance, see \cite{CGO} for $\mathbb R^2$ case and
\cite{KY} for the periodic case. The conservation in \eqref{eq:energy} follows from $\nabla\cdot u=0$, 
\[
\frac12\frac{d}{dt}\int_{\T^2}\rho^2\,dx
= -\frac12\int_{\T^2}u\cdot\nabla(\rho^2)\,dx = 0.
\]
And $\norm{u(t)}_2 \leq \norm{\rho(t)}_2= \norm{\rho(0)}_2$.

Denote $u = \cT\rho$ and $v = \cT w$. Subtracting the two
equations of $\rho$ and $\rho+w$ yields
\[
w_t+(u+v)\cdot\nabla w+v\cdot\nabla\rho = 0.
\]
It follows from the above difference equation and the fact $\nabla\cdot (u+v)=0$ that
\[
\begin{split}
\frac12\frac{d}{dt}\|w\|_{H^3}^2 &\leq C\|u+v\|_{H^3}\|w\|_{H^3}^2 +\|v\cdot\nabla\rho\|_{H^3}\|w\|_{H^3}\\
&\leq C\bigl(\norm\rho_{H^3}+\norm w_{H^3}\bigr)\norm w_{H^3}^2+\|v\cdot\nabla\rho\|_{H^3}\|w\|_{H^3}
\end{split}
\]
where we also used Leibniz’s rule and the embeddings $H^3\hookrightarrow W^{1,\infty}$ and $H^1\hookrightarrow L^4$. 
On the other hand, we have
\[
\norm{v\cdot\nabla\rho}_{H^3}
\leq C\bigl(\norm v_\infty\norm\rho_{H^4}
+\norm v_{H^3}\norm{\nabla\rho}_\infty\bigr)
\leq C\norm w_{H^3}\norm\rho_{H^4}.
\]
Combining the last two inequalities yields \eqref{eq:differenceestimate}. 

\end{proof}

We note oddness is preserved for \eqref{eq:ipm} in the sense that the reflected triple
\[
\bigl(-\rho(t,-x),-u(t,-x),\pi(t,-x)\bigr)
\]
satisfies \eqref{eq:ipm} as well.
If $\rho_{\rm in}$ is odd, uniqueness identifies the solution with
$(\rho,u,\pi)$. Thus the density and velocity remain odd,  while the pressure
is even, and $u(t,0)=0$.

We now introduce the particle coordinates used in the analysis throughout the text.
Let $(\bar\rho,U,\bar\pi)$ be a smooth background solution, and let $X$ be the particle map satisfying
\[
\partial_tX(t,y) = U(t,X(t,y)),\qquad X(0,y) = y.
\]
Here $y\in\T^2$ labels a particle by its initial position, while
$x = X(t,y)$ is its physical position at time $t$. We fix the notation $q\circ X=(q\circ X)(t,y) = q(t,X(t,y))$. Denote
\[
g = (\nabla_x\bar\rho)\circ X,\qquad
M = (\nabla_xU)\circ X,\qquad H = (\nabla_x^2\bar\pi)\circ X.
\]
Thus $g$, $M$, and $H$ are physical derivatives evaluated along trajectories. Their time derivatives
are taken at fixed $y$. We use the conventions 
\[
M_{ij}(t,y) = \partial_{x_j}U_i(t,X(t,y)),\qquad
(a\otimes b)_{ij} = a_i b_j.
\]
One can verify
\begin{equation}\label{eq:basicjets}
\begin{split}
&\det(D_yX) = 1,\qquad \partial_t(D_yX) = M(D_yX),\qquad
g = (D_yX)^{-T}\nabla_y\bar\rho(0,y),\\
&M = -H-e_2\otimes g,\qquad
H = H^T,\qquad \tr H = -g_2.
\end{split}
\end{equation}
Indeed, since $D_yX(0,y) = I$,
Jacobi's formula and incompressibility imply $\det(D_yX) = 1$.
The transport equation implies $\bar\rho(t,X(t,y)) = \bar\rho(0,y)$, while the chain rule reads
\[
\nabla_y(q\circ X) = (D_yX)^T((\nabla_xq)\circ X).
\]
Taking $q = \bar\rho$ yields the formula for $g$. The last three identities in \eqref{eq:basicjets}
follow from Darcy's law and the symmetry of the physical pressure
Hessian. In particular, differentiation of
$(D_yX)^Tg = \nabla_y\bar\rho(0,y)$ also yields $g_t = -M^Tg$.

\section{Heuristics of the construction}\label{sec:heuristics}
The construction begins with a smooth stationary background solution and
proceeds by iterative perturbations. At each stage, the perturbation has
four components: a principal oscillation with rapidly growing frequncy, phase-independent higher-order
corrections, higher-order oscillatory corrections, and a final nonlinear
correction that removes the remaining residual. We highlight a few features
of the construction.
\begin{enumerate}[label=(\arabic*),leftmargin=*,itemsep=3pt,topsep=4pt]
\item At the end of each stage we obtain a smooth solution of the unforced IPM
system~\eqref{eq:ipm}, which serves as the background solution for the next stage.
\item The principal oscillation generates a large gradient through the
angular amplification mechanism pioneered in~\cite{CMZ}. The two families
of higher-order corrections cancel large residual terms
generated by this oscillation. The oscillation frequency increases triple-exponentially from one stage to the next.
\item At stage $j$, the principal phase and amplitude are prescribed at
$t_{j+1}$ and continued backward to $t=0$. The phase-independent corrections
are selected through backward linearized equations with zero terminal
condition at $S_{j+1}$. The principal oscillation
and the phase-independent corrections generally have nonzero initial values.
\item The perturbation is defined throughout $[0,S_{j+1}]$, with three
regimes: the earlier history $[0,t_j]$, the amplification interval
$[t_j,t_{j+1}]$, and the retention interval $[t_{j+1},S_{j+1}]$, during which
the configuration needed for the next stage persists.
\item A global one-sided upper bound for the pressure Hessian controls
backward growth. Together with the correction estimates and compatible
choices of frequencies and amplification parameters, this bound ensures
small gradient increments on $[0,t_j]$ and in particular the limiting initial density is
smooth. This use of pressure control for the earlier stage
is inspired by the Euler blowup construction of OpenAI~\cite{OpenAIEuler}.
\end{enumerate}
More detail on the construction of perturbations is provided below.

\subsection{The principal oscillation}
Let $(\bar\rho,U,\bar\pi)$ be a smooth solution on $[0,S]$, and let
$X$ be its particle map with $X(0,y)=y$, as in Section~\ref{sec:local}. The quantities
$D_yX,g,M,H$ are evaluated in the particle coordinates, and time
derivatives are taken at fixed $y$.

Choose a smooth local phase $\phi_0$ and a cutoff $0\leq\chi\leq1$ compactly
supported in its domain, with $\nabla_y\phi_0\ne0$ near $\supp\chi$. Define
\[
\phi(t,X(t,y)) = \phi_0(y),\qquad
\xi(t,y) = (\nabla_x\phi)(t,X(t,y)) = (D_yX(t,y))^{-T}\nabla_y\phi_0(y).
\]
For a smooth zero-mean $2\pi$-periodic function $f$ and a amplitude
$\mu>0$, consider the principal oscillation
\[
r_\lambda(t,X(t,y))
= \lambda^{-1}\mu(t,y)\chi(y)f(\lambda\phi_0(y)), \qquad \lambda\geq1,
\]
extended by zero outside the transported support. We will choose the principal data so that the dynamics
amplify $\nabla r_\lambda$ and prepare its direction for the next stage.
The higher-order and nonlinear corrections introduced in the next subsection remove the residual while
preserving this leading effect.

The transported phase determines the evolution of $\xi$, and we choose
$\mu$ to cancel the leading linearized residual. Along each background
trajectory, the resulting principal system is
\[
\partial_t\xi = -M^T\xi,\qquad
\partial_t\mu = -m(\xi)\cdot g\,\mu.
\]
Indeed, with $D_t=\partial_t+U\cdot\nabla_x$, it follows from the transport equation
\[
D_t\phi=0,\qquad
D_t\nabla_x\phi=-(\nabla_xU)^T\nabla_x\phi.
\]
Composition with $X$ yields the equation for $\xi$. On the other hand, we have from multiplier
expansion
\[
\bigl(D_t r_\lambda+(\cT r_\lambda)\cdot
 \nabla_x\bar\rho\bigr)\circ X=\lambda^{-1}\chi
 \bigl[\partial_t\mu+m(\xi)\cdot g\,\mu\bigr]
 f(\lambda\phi_0)+O(\lambda^{-2}).
\]
The equation for $\mu$ follows from cancelling the leading term coefficient. Moreover,
the potentially leading nonlinear self-interaction vanishes because
$m(\xi)\cdot\xi=0$.
The remaining terms are handled by the corrections in the next subsection. 
It is noteworthy that the application of such phase-amplitude system of the leading linearized dynamics
was pioneered by Friedlander and Vishik \cite{FV91, VF93} and Lifschitz and Hameiri \cite{LH91, LH93} independently to study
hydrodynamic instability in the early 1990s. 

Denote the unit phase normal by $p=\xi/|\xi|$ and the principal gradient
amplitude by $A=\mu|\xi|$. In view of 
\[
(\nabla_xr_\lambda)\circ X
=\mu\chi\xi f'(\lambda\phi_0)
+\lambda^{-1}(D_yX)^{-T}\nabla_y(\mu\chi)f(\lambda\phi_0),
\]
at a label where $\chi=1$ and $\phi_0=0$, the leading gradient increment
is $Ap$ if $f'(0)=1$. If also $f(0)=0$, as for the odd profile used later,
the second term vanishes there, so the gradient increment is $Ap$.
We discuss the growth of $A$ from the principal dynamics in the following

Lemma~\ref{amp:lem:pressure} derives the pressure equation from the phase and amplitude system above
\[
\partial_t\log A = -p^\perp\cdot Hp^\perp,\qquad p^\perp=(-p_2,p_1).
\]
Thus negative pressure curvature in the direction $p^\perp$ 
produces growth of $A$. The accumulated growth depends on the pressure
curvature by the evolving phase direction.

To describe the angular amplification, we consider the constant-coefficient
reference system of Subsection~\ref{sec:ampstart} along one trajectory,
measuring time from the start of this reference evolution. For illustration,
take $h>0$ and $0<\theta\leq1/10$, and denote
\[
n=(\sin\theta,\cos\theta),\qquad
v=(-\cos\theta,\sin\theta).
\]
Here $n$ is the nearly vertical background gradient direction and $v\perp n$
is the initial phase direction. We then have the reference coefficients 
\[
g_*=hn,\qquad M_*=-h\sin\theta\,v\otimes n,\qquad
H_*=-h\cos\theta\,n\otimes n.
\]
For $\xi(0)=v$, it follows from the phase and pressure equations that
\[
\xi(t)=v+h\sin\theta\,t\,n,\qquad
\log\frac{A(t)}{A(0)}
=\cot\theta\,\arctan(h\sin\theta\,t),
\]
where the arctangent is the angle through which the phase direction has rotated
from $v$ toward $n$. An order-one rotation therefore yields a logarithmic
gain of size $\theta^{-1}$, and hence an exponential gradient
amplification as $\theta$ decreases. At the stopping time chosen in
Proposition~\ref{amp:prop:frozen}, the phase direction is nearly upward
vertical on the opposite side from $n$, preparing the background geometry for the next
stage. 

Figure~\ref{fig:angular-amplification} illustrates this reference evolution.
\begin{figure}[tbp]
\centering
\includegraphics[width=\textwidth]{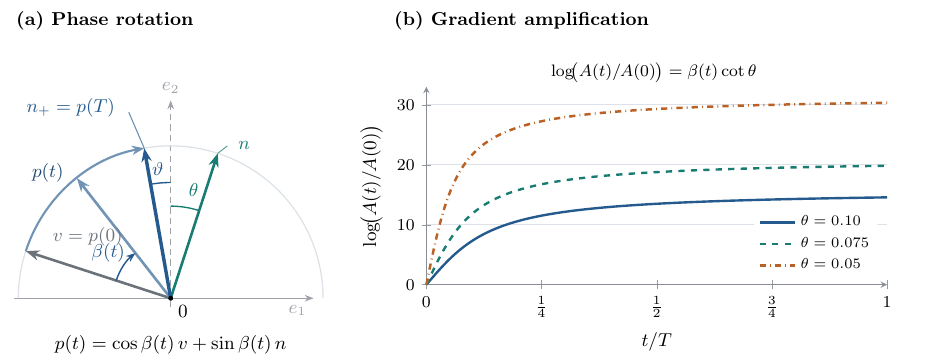}
\caption{Angular amplification for the constant-coefficient system.
(a) The phase normal rotates from $v$ to
$n_+=(-\sin\vartheta,\cos\vartheta)$, $0<\vartheta\leq\theta$,
stopping before the vertical. The angles are enlarged here for illustration.
(b) Reference curves show logarithmic gain for $\vartheta=\theta^2$, with $T$ from
\eqref{amp:eq:targettime} and
$\beta(t)=\arctan(h\sin\theta\,t)$.
Each curve has finite terminal gain $\Gamma$.}
\label{fig:angular-amplification}
\end{figure}

Proposition~\ref{amp:prop:robust} generalizes this amplification to the
variable-coefficient system with prescribed terminal phase gradient and
principal gradient amplitude. 
The terminal data select the outgoing direction needed for the next stage. 
Finally, the pressure equation also controls the earlier stage through an upper bound $H\leq K_+(t)I$ which implies
\[
A(s)\leq A(t)\exp\left(\int_s^tK_+(\tau)\,d\tau\right),\qquad s\leq t.
\]

\subsection{Approximate and exact perturbations}\label{ss:perturbation}
The principal oscillation produces the intended gradient growth, but it
solves the IPM equation only to leading order. The
higher-order corrections will be introduced to cancel large residual terms while
preserving the leading gradient. A final correction then removes
the remaining residual. We explain the two steps and how their
estimates preserve the amplification.

To cancel the large errors generated by the principal oscillation, we introduce
\[
\begin{aligned}
\rho^{a,N}(t,X(t,y)) ={}&\bar\rho(0,y)
+\lambda^{-1}\mu(t,y)\chi(y)f(\lambda\phi_0(y))
+\sum_{n=2}^{N}\lambda^{-n}
\bigl[\bar r_n(t,y)+r_n(t,y,\lambda\phi_0(y))\bigr].
\end{aligned}
\]
Here $\bar r_n$ is independent of the phase variable $\omega$, whereas
$r_n(t,y,\omega)$ is periodic and has zero mean in $\omega$. The
phase-independent term cancels the phase average of the residual, and the oscillatory term cancels the remaining zero-mean part.

The coefficients $\bar r_n$ and $r_n$ are constructed successively in $n$. At each order, the
averaged equation is the linearized IPM equation, with a source
determined by lower-order coefficients. Lemma~\ref{a:lem:slow} solves this
equation through a backward problem with vanishing data at time $S<1$. This choice determines $\bar r_n$ on the whole interval and
selects its generally nonzero initial value. The velocity induced by
$\bar r_n$ then enters the scalar equation for $r_n$, which is solved with
$r_n(0,y,\omega)=0$. 

The estimates in Proposition~\ref{a:prop:majorant} control both spatial
derivatives and expansion order $N$. 
Choose a fixed sufficiently small $\nu>0$ and truncate at
$N=\lfloor\lambda^\nu\rfloor$. 
Denoting $\rho^a=\rho^{a,N}$ and $u^a=\cT\rho^a$, we obtain for some fixed $\bar{\mu}\geq 2$
\[
\partial_t\rho^a+u^a\cdot\nabla_x\rho^a=\mathcal R,
\qquad
\sup_{0\leq t\leq S}\norm{\mathcal R(t)}_{H^4_{1/\lambda,x}}
\leq \bar{\mu}^C\exp(-c\lambda^\nu\log\lambda),
\]
which is the estimate \eqref{a:eq:Gresidual} converted to physical
coordinates. 
Thus the finite expansion in $\rho^{a,N}$ greatly reduces the error while retaining the principal amplification.

To remove the remaining residual, let $\rho^+$ be the local solution to the IPM equation \eqref{eq:ipm} with
initial density $\rho^a(0)$, and denote $u^+=\cT\rho^+$ and
$e=\rho^+-\rho^a$. Subtracting the approximate
equation from the exact one, we obtain
\[
e_t+(u^a+\cT e)\cdot\nabla_x e
+(\cT e)\cdot\nabla_x\rho^a=-\mathcal R,\qquad e(0)=0.
\]
To show that $\rho^+$ exists on $[0,S]$ and stays close enough to preserve the desired gradients, weighted energy
estimate \eqref{a:eq:finiteenergy} is applied, which results in the estimate
\eqref{a:eq:finitecorrection} for appropriate choice of parameters
\[
\sup_{0\leq t\leq S}\norm{e(t)}_{H^4_{1/\lambda,x}}
\leq\exp(-\lambda^{\nu/2}).
\]
For each fixed $\lambda$, this also controls the $H^3$ norm, so
Lemma~\ref{lem:local} continues $\rho^+$ through $S$.
Combining the approximate-gradient estimates with the correction estimate above 
implies
\[
\sup_{0\leq t\leq S}
\norm{(\nabla_x\rho^+-\nabla_x\bar\rho)\circ X
-\mu\chi\xi f'(\lambda\phi_0)}_{L^\infty_y}
\leq\lambda^{-1/2}.
\]
which indicates that the higher-order and exact corrections
preserve the principal increments to the accuracy required by the
iteration. 

Since $e(0)=0$ and
$r_n(0)=0$ for $n\geq2$, the initial density is determined as
\[
\rho^+(0)-\bar\rho(0)=r_\lambda(0)+b_\lambda,\qquad
b_\lambda=\sum_{n=2}^N\lambda^{-n}\bar r_n(0).
\]
Theorem~\ref{thm:main} provides estimates to show that the limiting initial data is smooth.


\subsection{Iteration mechanism}
Now we argue that the one stage construction above can be iterated.
At stage $j$, the solution $\rho_j$ (renamed $\rho^+$ from the construction above) on $[0,S_j]$ serves as the
background for a new perturbation in next stage $j+1$. 

We first explain how the geometry passes from stage $j$ to $j+1$. On a
material neighborhood of the origin, the induction controls the density
gradient, velocity gradient, and pressure Hessian of $\rho_j$ near
\[
\bigl(h_jn_j,\ h_jm(n_j)\otimes n_j,
-h_j\cos\theta_j\,n_j\otimes n_j\bigr),\qquad
n_j=((-1)^j\sin\theta_j,\cos\theta_j),
\]
throughout $[t_j,S_j]$, with the error in \eqref{geo:eq:IHgeometry}.
Thus the background gradient is nearly upward vertical, with its tilt
alternating from one stage to the next. The separate global upper pressure
bound \eqref{geo:eq:IHH} controls the earlier history and the backward
inverse.

At the origin, let $p_j$ and $A_j$ denote the phase direction and principal
gradient amplitude of the new oscillation. Prescribe its terminal phase
and amplitude as in \eqref{geo:eq:terminaldata}, so that
\[
p_j(t_{j+1})=n_{j+1},\qquad A_j(t_{j+1})=h_{j+1}.
\]
The phase and amplitude are then determined backward to time zero.
By Proposition~\ref{amp:prop:robust}, the incoming direction satisfies
\[
p_j(t_j)\approx v_j,\qquad
v_j=\bigl((-1)^{j+1}\cos\theta_j,\sin\theta_j\bigr)\perp n_j.
\]
The incoming phase normal is nearly horizontal, whereas its outgoing
direction $n_{j+1}$ lies closer to the vertical on the opposite side from
$n_j$. At the origin, the outgoing principal density-gradient increment
is $h_{j+1}n_{j+1}$. Indeed, we have
\[\nabla\rho_{j+1}(t_{j+1},0)=\nabla\rho_j(t_{j+1},0)+h_{j+1}n_{j+1}+\text{small error}.\]
Since $h_{j+1}\gg h_j$,
this increment determines the dominant direction of the new background.
The following stage introduces a different oscillation, with its own
phase transported by the velocity of $\rho_{j+1}$. Its incoming direction
satisfies
\[
p_{j+1}(t_{j+1})\approx v_{j+1}\perp n_{j+1},
\qquad p_{j+1}(t_{j+2})=n_{j+2}.
\]
For the next oscillation, terminal preparation at $t_{j+2}$ again determines the phase and amplitude back to time zero.
The resulting solution $\rho_{j+1}$ is constructed on $[0,S_{j+1}]\subset[0,S_j]$. 

Figure~\ref{fig:angular-stage-transfer} summarizes this transfer.
\begin{figure}[tbp]
\centering
\includegraphics[width=\textwidth]{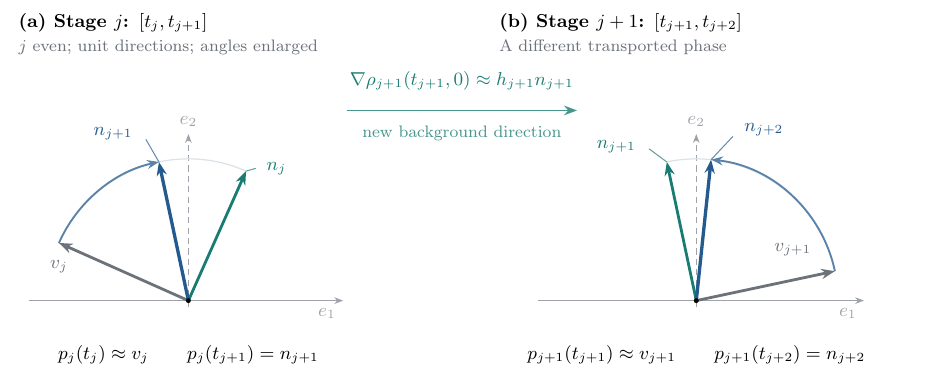}
\caption{Successive angular amplification stages, shown for even $j$.
The outgoing gradient increment $h_{j+1}n_{j+1}$, with $h_{j+1}\gg h_j$,
determines the dominant direction of the next background.
The next oscillation has its own phase and incoming direction
$p_{j+1}(t_{j+1})\approx v_{j+1}\perp n_{j+1}$;
both oscillations are prepared from time zero.
Arrows represent unit directions; angles are enlarged, whereas the
construction uses $\theta_{j+1}=\theta_j^2$.}
\label{fig:angular-stage-transfer}
\end{figure}

Next we emphasize the roles of the three time regimes on each stage $j$: the principal gradient is small on the earlier
stage $[0,t_j]$, grows fast on the amplification interval $[t_j,t_{j+1}]$,
and retains the configuration needed for the next stage on $[t_{j+1},S_{j+1}]$. 

We iteratively choose the angle and gradient amplitude parameters 
\[
\theta_{j+1}=\theta_j^2,\qquad
h_{j+1}=Kh_j\theta_j^{-14},\qquad h_0=\theta_0^{-4},
\]
where $K\geq1$ is sufficiently large and $0<\theta_0\leq1/10$ is small. We then choose the time iteratively
\[
t_{j+1}-t_j=
\frac{\cot(\theta_j+\theta_j^2)}{h_j\sin\theta_j},\qquad
S_j=t_j+\frac4{h_j\theta_j^2},\qquad t_0=0.
\]
The amplification time and the remaining existence interval shrink since
\[
t_{j+1}-t_j\leq\frac2{h_j\theta_j^2},\qquad
\frac{h_j\theta_j^2}{h_{j+1}\theta_{j+1}^2}
=K^{-1}\theta_j^{12}\leq\frac14.
\]
It is clear that $S_{j+1}<S_j$ and $S_j-t_j\to0$. The tail estimate following \eqref{geo:eq:finiteT}
implies, with $T_\infty=\lim_jt_j$,
\[
t_j<t_{j+1}<T_\infty<S_{j+1}<S_j,\qquad
S_j\to T_\infty<\infty.
\]
In particular, the preparation of geometry configuration interval $[t_{j+1},S_{j+1}]$ contains the
entire amplification interval $[t_{j+1},t_{j+2}]$ of next stage.

In the end, we discuss the rapid growth of oscillation frequency in the iteration process.
The prescribed frequency and amplification exponent are 
\[
\lambda_j=\exp(\theta_j^{-1/2}),\qquad
\Gamma_j=\cot\theta_j\left(\frac\pi2-\theta_j-\theta_j^2\right).
\]
It follows the frequency recurrence
\[
\log\lambda_{j+1}=(\log\lambda_j)^2,\qquad
\lambda_{j+1}=\lambda_j^{\log\lambda_j},
\]
and equivalently,
\[
\lambda_j=\exp\!\left((\log\lambda_0)^{2^j}\right).
\]
The frequencies therefore grow triple-exponentially with the iteration index.

This frequency separation accommodates the derivative and localization costs
inherited from the preceding solution. Let $\bar{\mu}_j$ be the data constant for
the $j$th perturbation, constructed from $\rho_j$ before using $\lambda_j$.
Subsections~\ref{sec:inductionstep} and~\ref{geo:sec:budgets} prove the bound
$\log \bar{\mu}_{j+1}\leq C\log\lambda_j$, together with
$\lambda_j\geq\max(\underline{\lambda},\bar{\mu}_j^Q)$, where
$\underline{\lambda}\geq1$ is the fixed absolute frequency threshold
in the perturbation construction. Moreover,
$\Gamma_j\approx \theta_j^{-1}=(\log\lambda_j)^2$ and 
\eqref{geo:eq:Pwindow}--\eqref{geo:eq:allthresholds} imply
\[
\log \bar{\mu}_j=o(\log\lambda_j),\qquad
\log\lambda_j=o(\Gamma_j)
\qquad\text{as }j\to\infty.
\]
Consequently, the initial-increment estimate
\[
\norm{\rho_{j+1}(0)-\rho_j(0)}_{H^m}
\leq C_m\bar{\mu}_j^{C_m}
\bigl(\mathrm e^{-\Gamma_j}\lambda_j^{m-1}+\lambda_j^{-2}\bigr)
\]
is summable for every fixed $m$, as proved in
\eqref{geo:eq:summableinitial}. This shows that a sequence of initial densities converges to a smooth
$\rho_{\rm in}$, while the retained central gradients tend to infinity.

\subsection{Comparison with constructions for the forced IPM}\label{sec:forcedcomparison}

In \cite{CMZ}, Taylor approximations of the background transport and localized velocity expansions lead to a
correction scheme. Temporal activation, phase-independent interactions,
and the remaining errors contribute to the external source. In
\cite{ABC}, the oscillation is transported by a Fourier-truncated background. The cutoff commutator and perturbation
residual enter the force, which is space--time smooth. In both \cite{CMZ,ABC}, their estimates control the source terms, but do
not establish a correction that removes them while preserving amplification.

In contrast, to avoid the external source in our construction,
the first requirement is that perturbation starts from the initial time and has
controlled earlier stage estimates.
Suppose a smooth background has a prescribed source and an added
perturbation $r$ satisfies $r(0)=0$. If a correction $e$, with $e(0)=0$,
removes the additional residual while preserving that source, the corrected
density and the background solve the same classical equation with the same
initial density. Uniqueness on their common smooth interval implies
$e=-r$. Thus the correction cancels the activated perturbation.

We therefore prepare a small, generally nonzero initial increment and extend its phase
and amplitude throughout the earlier interval. The phase is transported by
the actual background particle map. Terminal preparation selects the later
geometry, as in the strategy of \cite{CMZ, DaiEMHD}, while the exact
nonlinear solution starts from the prepared initial density. 

The second requirement concerns nonlinear interactions of oscillations.
Conjugate pair of phase profiles generate a non-oscillatory coefficient. The leading self-interaction cancels through
$m(\xi)\cdot\xi=0$, but lower-order interactions leave the generally
nonzero phase average in \eqref{a:eq:quadraticmean}. We incorporate this
term into the density through the linearized equation
\eqref{a:eq:linearized}. Lemma~\ref{a:lem:slow} selects its initial value
by a zero terminal condition. 
The backward estimate depends on the time integral of an upper
bound for the pressure Hessian, as in \eqref{a:eq:slowbound}.
Consequently, large negative pressure curvature can coexist with controlled
backward growth.

The third requirement is the removal of the remaining residual.
Simultaneous estimates in the derivative order and the expansion order
permit a finite Gevrey expansion with the exponentially small residual
\eqref{a:eq:Gresidual}, which makes it possible to construct an unforced solution, as explained in Subsection \ref{ss:perturbation}.

\subsection{Comparison with the Euler construction by OpenAI}\label{sec:eulercomparison}
We first describe the similarity and difference between the 2D IPM system and the 3D Euler equations.
The IPM system \eqref{eq:ipm} formulated for $\nabla^\perp\rho$ becomes
\[
(\partial_t+u\cdot\nabla)\nabla^\perp\rho=(\nabla^\perp\rho\cdot\nabla)u;
\]
while for an Euler velocity $u_E$ in 3D, its vorticity
$\omega_E=\nabla\times u_E$ satisfies
\[(\partial_t+u_E\cdot\nabla)\omega_E=(\omega_E\cdot\nabla)u_E.\]
Thus $\nabla^\perp\rho$ satisfies the same transport-stretching equation,
and $|\nabla^\perp\rho|=|\nabla\rho|$. Along the respective particle flows,
both fields evolve by multiplication by the deformation matrix.
However, the velocity recovery is different for the two equations, i.e.
\[
u=\partial_1(-\Delta)^{-1}\nabla^\perp\rho,
\qquad
u_E=\nabla\times(-\Delta)^{-1}\omega_E.
\]
Both operators have order $-1$ in the stretched field, but the IPM
operator retains a distinguished direction. In particular,
$\partial_1u_2-\partial_2u_1=-\partial_1\rho$, so $\nabla^\perp\rho$ is a different quantity from the physical vorticity of the IPM
velocity.

Since the work of \cite{OpenAIEuler} for the Euler equations is also inspired by the pioneering work of \cite{CMZ}, our construction shares 
some common features with that of \cite{OpenAIEuler}. 
For instance, each background amplifies a localized oscillation whose resulting gradient has the geometry required
for the next stage iteration; later geometric requirements determine a
small earlier history extending to time zero; nonlinear interactions generate
non-oscillatory corrections in addition to the principal oscillation. In particular, the argument of applying an 
upper bound for the pressure Hessian to control earlier stage of the perturbations in our construction was
adapted from \cite{OpenAIEuler}.

Nevertheless, there are several major distinctions between our construction and that of \cite{OpenAIEuler}.
First, the amplification mechanisms are different. The IPM construction follows a
scalar density amplitude and a phase direction in the plane. For the
constant-coefficient reference, the phase-amplitude-pressure equation \eqref{amp:eq:pressuregrowth} yields the gain
\eqref{amp:eq:frozenray}; the terminal direction lies closer to the vertical, on the opposite side
from the background density gradient, and prepares the next background configuration.
This angular mechanism is inherited from \cite{CMZ}. The Euler construction follows a phase normal and an independent transverse
velocity amplitude in a moving three-dimensional frame. The large shear couples to the older background, and the new normal and transverse velocity
define the next frame. 

Second, the backward preparation also has a different structure. Darcy's law
yields the first-order terminal displacement equation \eqref{a:eq:slowdisplacement}, while the Euler displacement
problem is second order. This distinction is visible in the following equations respectively for the IPM and Euler equations 
\[
\begin{aligned}
(\partial_t+u\cdot\nabla)\nabla^\perp\rho
&=-(\nabla^2\pi)\nabla^\perp\rho,\\
(\partial_t+u_E\cdot\nabla)^2\omega_E
&=-(\nabla^2\pi_E)\omega_E
\end{aligned}
\]
where $\pi_E$ denotes the Euler pressure.

Third, the oscillatory frequency grows at different rates.
Our choice of $\theta_{j+1}=\theta_j^2$ and
$\lambda_j=\exp(\theta_j^{-1/2})$ indicates a triple-exponentially growth of $\lambda_j$ with the stage index $j$.
In the Euler construction, the frequency $k_j$ grows as $k_j=\exp(\exp(2j\log j+O(j)))$ which is slightly faster than 
double exponential growth. 
In both constructions, the large separation of frequencies makes it possible to realize amplification with controlled 
errors and compatible derivative costs of the oscillatory initial increments. 
The different growth rates arise from different geometric structures and amplification mechanisms in the two systems. 

\section{Principal amplification and persistence}\label{sec:amplification}
To carry on the iteration of the construction, we need to establish two important properties of the principal system: amplification of
a small oscillation and persistence of the resulting gradient
configuration during the next interval. We first obtain these properties for
constant coefficients in the principal system and then study their stability under perturbation.

\subsection{The phase-amplitude-pressure equation and the constant-coefficient system}\label{sec:ampstart}
Fix a background trajectory $x = X(t,y)$ and suppress the particle label $y$.
We recall the principal phase-amplitude system
\begin{equation}\label{eq:principal}
\dot\xi = -M^T\xi,\qquad
\dot\mu = -m(\xi)\cdot g\,\mu,\qquad \mu>0,\quad\xi\ne0,
\end{equation}
with coefficients $g$ and $M$ (hence $H$) along this trajectory,
where the dot denotes differentiation at fixed $y$. 
Nonzero initial $\xi$ and positive initial $\mu$ remain nonzero and positive,
respectively. 

Denote $A=\mu|\xi|$ and $p=\xi/|\xi|$.

\begin{lemma}\label{amp:lem:pressure}
The triplet $(A, p, H)$ satisfies
\begin{equation}\label{amp:eq:pressuregrowth}
\frac{d}{dt}\log A = -p^\perp\cdot Hp^\perp.
\end{equation}
Assume $H(\tau)\leq K_+(\tau)I$ for $s\leq\tau\leq t$ with
$K_+\in L^1(s,t)$. Then $A(t)$ satisfies
\begin{equation}\label{amp:eq:historybasic}
A(s)\leq A(t)\exp\left(\int_s^tK_+(\tau)\,d\tau\right).
\end{equation}
\end{lemma}
\begin{proof}
Since $m(p) = p_2p-e_2$ and $M^T = -H-g\otimes e_2$, we have
\begin{align*}
\frac{d}{dt}\log A
& = -m(p)\cdot g-p\cdot M^Tp\\
& = -p_2(p\cdot g)+g_2+p\cdot Hp+p_2(p\cdot g)\\
& = -\tr H+p\cdot Hp = -p^\perp\cdot Hp^\perp,
\end{align*}
where in the last equality we used the orthonormal basis $p,p^\perp$ in two dimensions. The estimate in \eqref{amp:eq:historybasic} 
follows immediately from the equation above and the upper bound assumption of $H$.
\end{proof}

Take
\begin{equation}\label{amp:eq:parameters}
0<\vartheta\leq \theta\leq \tfrac1{10},\qquad h>0,\qquad \sigma\in\{-1,1\},
\end{equation}
and define
\begin{equation}\label{amp:eq:frame}
n = (\sigma\sin\theta,\cos\theta),\quad
v = (-\sigma\cos\theta,\sin\theta),\quad
n_+ = (-\sigma\sin\vartheta,\cos\vartheta).
\end{equation}
We note $n,v$ are orthonormal. Define the reference coefficients
\begin{equation}\label{amp:eq:frozenjets}
g_* = hn,\qquad M_* = -h\sin\theta\,v\otimes n,
\qquad H_* = -h\cos\theta\,n\otimes n,
\end{equation}
which satisfy
\[
M_* = -H_*-e_2\otimes g_*,\qquad
\tr H_* = -(g_*)_2,\qquad M_*^2 = 0,\qquad M_*^Tg_* = 0.
\]

Choose
\begin{equation}\label{amp:eq:targettime}
T = \frac{\cot(\theta+\vartheta)}{h\sin\theta},\qquad
\Gamma = \cot\theta\left(\frac\pi2-\theta-\vartheta\right).
\end{equation}

\begin{proposition}\label{amp:prop:frozen}
For the coefficients \eqref{amp:eq:frozenjets}, the solution of
\eqref{eq:principal} with $\xi(0)=v$ and any $\mu(0)>0$ satisfies,
\begin{equation}\label{amp:eq:frozenray}
\xi(t) = v+h\sin\theta\,t\,n,\qquad
\frac{A(t)}{A(0)}
= \exp\bigl(\cot\theta\arctan(h\sin\theta\,t)\bigr), \quad t\geq0.
\end{equation}
In particular, we have
\begin{equation}\label{amp:eq:exacttransfer}
p(T) = n_+,\quad |\xi(T)| = \csc(\theta+\vartheta),\quad
A(T) = \mathrm e^\Gamma A(0),
\end{equation}
\begin{equation}\label{amp:eq:timegainestimates}
\frac1{3h\theta^2}\leq T\leq \frac2{h\theta^2},
\qquad \frac1\theta\leq \Gamma\leq \frac2\theta,
\end{equation}
\begin{equation}\label{amp:eq:signestimates}
p_2(t)\geq\sin\theta,\quad
-\sigma p_1(t)\geq\sin\vartheta, \qquad 0\leq t\leq T.
\end{equation}
\end{proposition}
\begin{proof}
With \eqref{amp:eq:frozenjets}, the phase equation becomes $\dot\xi = h\sin\theta\,n(v\cdot\xi)$.
Taking its scalar product with $v$ implies $v\cdot\xi(t)=v\cdot\xi(0)=1$, and hence we have $\xi(t) = v+h\sin\theta\,t\,n$.
Consequently,
\[
|\xi(t)|^2 = 1+(h\sin\theta\,t)^2,\qquad
(p^\perp\cdot n)^2 = \frac1{1+(h\sin\theta\,t)^2}.
\]
The phase-amplitude-pressure equation \eqref{amp:eq:pressuregrowth} turns into
\[
\frac{d}{dt}\log A
= \frac{h\cos\theta}{1+(h\sin\theta\,t)^2},
\]
which implies \eqref{amp:eq:frozenray}.

Denote $\beta(t) = \arctan(h\sin\theta\,t)$. We have
\[
p(t) = \cos\beta(t)\,v+\sin\beta(t)\,n
= \bigl(-\sigma\cos(\theta+\beta(t)),\sin(\theta+\beta(t))\bigr),
\]
which combined with $\beta(T)=\pi/2-\theta-\vartheta$ and
\eqref{amp:eq:frozenray} implies \eqref{amp:eq:exacttransfer}. Moreover, we note
\[
\theta\leq\theta+\beta(t)\leq\frac\pi2-\vartheta,
\qquad 0\leq t\leq T,
\]
which implies \eqref{amp:eq:signestimates}.
For $\theta<\theta+\vartheta\leq2\theta\leq1/5$, elementary estimates show
\[
\begin{split}
&\frac{\cos(1/5)}{2h\theta^2}\leq T\leq\frac2{h\theta^2},\\
&\frac{\cos(1/10)(\pi/2-1/5)}{\theta}
\leq\Gamma\leq\frac{\pi}{2\theta},
\end{split}
\]
which verify \eqref{amp:eq:timegainestimates}.
\end{proof}

\subsection{Stability of the principal amplification}\label{amp:sec:robust}
Along a fixed background trajectory, take the reference coefficients
\eqref{amp:eq:frozenjets} and parameters
\eqref{amp:eq:parameters}, \eqref{amp:eq:targettime}, and assume
\begin{equation}\label{amp:eq:coefficienterror}
\sup_{0\leq t\leq T}
\bigl(\norm{M(t)-M_*}+\norm{H(t)-H_*}\bigr)\leq\eps h
\end{equation}
with
\begin{equation}\label{amp:eq:smallness}
0\leq\eps\leq c_0\min\{\theta^6,\vartheta\theta^4\}.
\end{equation}
Denote $\bar\xi(t)=v+h\sin\theta\,t\,n$ for the reference solution.

\begin{proposition}\label{amp:prop:robust}
Under the assumptions \eqref{amp:eq:coefficienterror}--\eqref{amp:eq:smallness},
there are absolute constants $c_0,C>0$ such that the solution of \eqref{eq:principal} with terminal data
\[
\xi(T) = \csc(\theta+\vartheta)n_+,\qquad A(T) = h_+>0
\]
satisfies
\begin{equation}\label{amp:eq:rayerror}
\sup_{[0,T]}\left(
\frac{|\xi-\bar\xi|}{|\bar\xi|}
+\left|p-\frac{\bar\xi}{|\bar\xi|}\right|\right)
\leq C\eps\theta^{-4},
\end{equation}
\begin{equation}\label{amp:eq:amplitudeerror}
\left|\log\frac{A(t)}{h_+}+\Gamma
-\cot\theta\arctan(h\sin\theta\,t)\right|
\leq C\eps\theta^{-6}, \qquad 0\leq t\leq T.
\end{equation}
Consequently, we have
\begin{equation}\label{amp:eq:robustconclusions}
p(T) = n_+,\qquad
\tfrac12 h_+\mathrm e^{-\Gamma}\leq A(0)\leq2h_+\mathrm e^{-\Gamma},
\qquad \sup_{[0,T]}A\leq2h_+,
\end{equation}
\begin{equation}\label{amp:eq:robustsigns}
p_2\geq\tfrac12\sin\theta,\qquad
-\sigma p_1\geq\tfrac12\sin\vartheta
\quad\text{on }[0,T].
\end{equation}
\end{proposition}
\begin{proof}
The terminal problem \eqref{eq:principal}, with
$\mu(T)=h_+\sin(\theta+\vartheta)$, has a unique solution on $[0,T]$
with $\xi\ne0$ and $\mu>0$.

Denote $E$ by the reference phase propagator, which satisfies $\dot E=-M_*^TE$. In view of $(n\otimes v)^2=0$, we have
\[
E(t) = I+h\sin\theta\,t\,n\otimes v, \qquad E(t)^{-1}=I-h\sin\theta\,t\,n\otimes v.
\]
We also have $\bar\xi(t)=E(t)v$ and $|\bar\xi(t)|^2=1+(h\sin\theta\,t)^2$. Hence
\begin{equation}\label{amp:eq:nilpotentestimates}
\sup_{[0,T]}\bigl(\norm E+\norm{E^{-1}}\bigr)\leq C\theta^{-1},
\qquad \norm{E(t)}\leq C|\bar\xi(t)|
\end{equation}
where we used $h\sin\theta\,T=\cot(\theta+\vartheta)\leq\theta^{-1}$ for the first inequality.

With the substitution $\xi=EZ$, we have the terminal problem for $Z$
\[
\dot Z = -E^{-1}(M-M_*)^TEZ,\qquad Z(T) = v,
\]
where the terminal condition follows from $\xi(T)=\bar\xi(T)=E(T)v$.
It follows from assumption \eqref{amp:eq:coefficienterror} and \eqref{amp:eq:nilpotentestimates} that
\begin{equation}\label{amp:eq:interactionbound}
\int_0^T\norm{E^{-1}(M-M_*)^TE}\,dt
\leq C\theta^{-2}\eps hT
\leq C\eps\theta^{-4}.
\end{equation}
Backward Gronwall's inequality and \eqref{amp:eq:smallness} imply
\[
\sup_{[0,T]}|Z-v|\leq\exp(C\eps\theta^{-4})-1
\leq C\eps\theta^{-4}.
\]
Since $\xi-\bar\xi=E(Z-v)$ and
$\bigl|p-\bar\xi/|\bar\xi|\bigr|\leq2|\xi-\bar\xi|/|\bar\xi|$,
\eqref{amp:eq:nilpotentestimates} yields \eqref{amp:eq:rayerror}.

Let $\bar A$ be the reference amplitude with the same terminal value
\[
\bar A(t) = h_+\exp\bigl(-\Gamma+
\cot\theta\arctan(h\sin\theta\,t)\bigr).
\]
The equation \eqref{amp:eq:pressuregrowth} for $A$ and $\bar A$ implies
\begin{equation}\notag
\left|\frac{d}{dt}\log\frac{A}{\bar A}\right|
=\left|p^\perp\cdot Hp^\perp
-\frac{\bar\xi^\perp\cdot H_*\bar\xi^\perp}{|\bar\xi|^2}\right|
\leq\eps h+2\norm{H_*}
\left|p-\frac{\bar\xi}{|\bar\xi|}\right|
\leq C\eps h\theta^{-4}.
\end{equation}
Estimate \eqref{amp:eq:amplitudeerror} then follows from the equation above and the facts
$A(T)=\bar A(T)$ and $T\leq2/(h\theta^2)$.
Choose $c_0$ so that $C\eps\theta^{-6}\leq\log2$ and hence
$\tfrac12\bar A\leq A\leq2\bar A$. Since
$\bar A(0)=h_+\mathrm e^{-\Gamma}$ and $\bar A(t)\leq h_+$ on
$[0,T]$, this establishes \eqref{amp:eq:robustconclusions}; the terminal
direction follows directly from the prescribed data.

Since $\eps\leq c_0\vartheta\theta^4$ and
$\sin\vartheta\geq\vartheta/2$, decreasing $c_0$ ensures
\[
C\eps\theta^{-4}\leq\tfrac12\sin\vartheta
\leq\tfrac12\sin\theta.
\]
Combining \eqref{amp:eq:rayerror} with the reference estimates \eqref{amp:eq:signestimates} yields \eqref{amp:eq:robustsigns}.
\end{proof}

\subsubsection{Backward continuation and construction of the principal oscillation}
Denote $t_1=t_0+T$ with $t_0\geq0$. Suppose the background solution is smooth on $[0,t_1]$. Apply Proposition~\ref{amp:prop:robust} on $[t_0,t_1]$ with
$t$ replaced by $t-t_0$. On the earlier stage of the trajectory $[0,t_0]$, assume
\begin{equation}\label{early-hypo}
H(t)\leq K_+(t)I, \quad 0\leq t\leq t_0; \qquad
K_+\geq0,\qquad K_+\in L^1(0,t_0).
\end{equation}
The principal system continues backward to time zero, with $\xi\ne0$ and
$\mu>0$. Lemma~\ref{amp:lem:pressure} and the estimate for $A(t_0)$ imply
\begin{equation}\label{amp:eq:fullhistory}
\sup_{0\leq t\leq t_0} A(t)
\leq2h_+\exp\left(-\Gamma+\int_0^{t_0}K_+(t)\,dt\right).
\end{equation}

Let $\mathcal O$ be a label neighborhood such that
$X(t_1,\mathcal O)$ lies in a torus coordinate chart, and choose
$y_*\in\mathcal O$. In a lift of this chart, define
\[
\phi_0(y)=\csc(\theta+\vartheta)n_+\cdot
\bigl(X(t_1,y)-X(t_1,y_*)\bigr)
\]
and denote $\phi(t,X(t,y))=\phi_0(y)$.
In particular, its physical gradient along the flow is
\[
\xi(t,y) = (D_yX(t,y))^{-T}\nabla_y\phi_0(y)
= \csc(\theta+\vartheta)(D_yX(t,y))^{-T}(D_yX(t_1,y))^Tn_+.
\]
We note $\xi$ satisfies the phase equation with $\xi\neq 0$, and it has the prescribed terminal value for every
$y\in\mathcal O$.
Choose the terminal amplitude
$\mu(t_1,y)=h_+\sin(\theta+\vartheta)>0$ on $\mathcal O$ and solve the amplitude
equation backward. Then $A(t_1,y)=h_+$ throughout this neighborhood.

Take $0\leq\chi\leq1$ compactly supported in $\mathcal O$ with $\chi\equiv 1$ on a smaller neighborhood of
$y_*$. We choose the principal oscillation  
\[
r_\lambda(t,X(t,y))
= \lambda^{-1}\mu(t,y)\chi(y)f(\lambda\phi_0(y)),
\]
extended by zero outside $X(t,\mathcal O)$, where the profile function $f$ is to be determined in next subsection.

\subsection{Choice of the periodic profile}
We choose a profile with unit slope near phase zero and a small negative
part of its derivative to control the pressure increment from above.

\begin{lemma}\label{amp:lem:profile}
There are absolute constants $\delta_{\max},C>0$ such that, for every
$0<\delta\leq\delta_{\max}$, there is an odd zero-mean and $2\pi$-periodic
Gevrey-two function $f_\delta$ satisfying $f_\delta(s)=s$ for $|s|\leq\delta/2$ and
\begin{equation}\label{amp:eq:profile}
\begin{split}
&f_\delta'(0) = 1,\qquad -C\delta\leq f_\delta'\leq1,
\qquad \norm{f_\delta}_\infty\leq C\delta,\\
&\norm{f_\delta^{(k)}}_\infty
\leq C^{k+1}((k-1)!)^2\delta^{1-k}, \qquad k\geq1.
\end{split}
\end{equation}
\end{lemma}
\begin{proof}
Fix an even bump function $b\in C_c^\infty((-1,1))$ such that $0\leq b\leq1$,
$b=1$ on $[-1/2,1/2]$, and
\[
\norm{b^{(n)}}_\infty\leq C_b^{n+1}(n!)^2, \qquad n\geq0.
\]
The constant $C_b$ is independent of $\delta$.

Denote
\[
a_\delta = \frac{\delta}{2\pi}\int_\R b(s)\,ds.
\]
Choose $\delta_{\max}\leq1$ so that $a_\delta\leq1/2$, and define on
$[-\pi,\pi]$
\[
f_\delta'(s) = \frac{b(s/\delta)-a_\delta}{1-a_\delta},
\qquad f_\delta(s) = \int_0^s f_\delta'(\tau)\,d\tau.
\]
The normalization yields
\[
\int_{-\pi}^{\pi}f_\delta'(s)\,ds
=\frac{\delta\int_\R b(s)\,ds-2\pi a_\delta}{1-a_\delta}=0.
\]
Since $f_\delta'$ is even and constant near $\pm\pi$, its primitive
extends to an odd smooth $2\pi$-periodic function with zero mean.
Since $f_\delta'=1$ on $|s|\leq\delta/2$, we have $f_\delta(s)=s$ there.
Moreover, we have
\[
-\frac{a_\delta}{1-a_\delta}\leq f_\delta'\leq1,
\qquad \norm{f_\delta'}_\infty\leq1,
\]
and
\[
\norm{f_\delta}_\infty
\leq\int_{-\pi}^{\pi}|f_\delta'(s)|\,ds
\leq\frac{2\delta\int_\R b(s)\,ds}{1-a_\delta}
\leq C\delta.
\]
It follows for $k\geq2$
\[
f_\delta^{(k)}(s)
= \frac{\delta^{1-k}}{1-a_\delta}b^{(k-1)}(s/\delta),
\]
which combined with estimate for $b$ implies \eqref{amp:eq:profile}.
\end{proof}

Use $f=f_\delta$ in the principal oscillation and denote
$\omega=\lambda\phi_0(y)$. Its principal gradient increments in label
coordinates are
\begin{equation}\label{amp:eq:leadingjets}
\begin{aligned}
g_\lambda& = A\chi p f_\delta'(\omega),\\
M_\lambda& = A\chi m(p)\otimes p f_\delta'(\omega),\\
H_\lambda& = -A\chi p_2\,p\otimes p f_\delta'(\omega)
\end{aligned}
\end{equation}
These tensors are extended by zero outside $\supp\chi$.

Let the amplification interval be $[t_0,t_1]$ with $t_1=t_0+T$, and
use the terminal data from the preceding subsection. At a terminal point
with $\phi_0=0$ and $\chi=1$, we have
\begin{equation}\label{amp:eq:outgoingtriple}
(g_\lambda, M_\lambda, H_\lambda)(t_1)=
\left(h_+n_+,\quad h_+m(n_+)\otimes n_+,\quad
-h_+\cos\vartheta\,n_+\otimes n_+\right),
\end{equation}
which has the form \eqref{amp:eq:frozenjets}, with parameters
$(h,\theta,\sigma)$ replaced by $(h_+,\vartheta,-\sigma)$.

Suppose the hypotheses of Proposition~\ref{amp:prop:robust} hold uniformly
on the material neighborhood supporting the oscillation. Then $p_2>0$
and $A\leq2h_+$ on $[t_0,t_1]$, and consequently,
\begin{equation}\label{amp:eq:hessianupper}
H_\lambda(t)\leq C\delta h_+ I, \quad t_0\leq t\leq t_1.
\end{equation}

On $[0,t_0]$, assume \eqref{early-hypo} holds for every
$y\in\supp\chi$ with the same nonnegative $K_+\in L^1(0,t_0)$.
Since $\norm{f_\delta'}_\infty\leq1$ and $0\leq\chi\leq1$ imply
$\norm{H_\lambda}\leq A$, \eqref{amp:eq:fullhistory} yields
\begin{equation}\label{amp:eq:historypressure}
\sup_{0\leq t\leq t_0,\,y}\norm{H_\lambda(t,y)}
\leq2h_+\exp\left(-\Gamma+\int_0^{t_0}K_+(\tau)\,d\tau\right).
\end{equation}

\subsection{Persistence on the next interval}\label{geo:sec:retention}
Translate the amplification interval to $[0,T]$ and continue the principal
system along the background flow.

Recall $0<\vartheta\leq\theta\leq1/10$ and the reference coefficients
\eqref{amp:eq:frozenjets}. Along a fixed background trajectory, suppose
\eqref{eq:basicjets} holds and
\begin{equation}\label{geo:eq:retbackground}
\norm{M-M_*}+\norm{H-H_*}\leq\eps h,
\qquad 0\leq\eps\leq\theta^6,
\end{equation}
on $[T,T+4/(h_+\vartheta^2)]$. Continue \eqref{eq:principal} with $p(T)=n_+$ and $A(T)=h_+$.
The reference triple for the next step is \eqref{amp:eq:outgoingtriple}.

\begin{lemma}\label{geo:lem:retention}
There are absolute constants $c,C>0$ such that if $\frac{h}{h_+}\left(\frac{\theta}{\vartheta}\right)^2\leq c$, then
\begin{equation}\label{geo:eq:retphase}
|p-n_+|\leq C\frac h{h_+}\frac{\theta^3}{\vartheta^2},
\qquad
|\log(A/h_+)|\leq C\frac h{h_+}\frac{\theta^2}{\vartheta^2},
\qquad p_2\geq\frac12
\end{equation}
on $[T, T+\frac{4}{h_+\vartheta^2}]$ and $h_+/2\leq A\leq2h_+$. If the fixed background trajectory has label $y$ with
$\phi_0(y)=0$, $\chi(y)=1$, and $f'(0)=1$, then
\[
\begin{aligned}
\sup_{T\leq t\leq T+4/(h_+\vartheta^2)}\Bigl(
&|g(t)+g_\lambda(t,y)-h_+n_+|
+\norm{M(t)+M_\lambda(t,y)-h_+m(n_+)\otimes n_+}\\
&+\norm{H(t)+H_\lambda(t,y)+h_+\cos\vartheta\,n_+\otimes n_+}\Bigr)
\leq Ch(\theta/\vartheta)^2,
\end{aligned}
\]
where $g$, $M$, and $H$ are evaluated along $X(t,y)$.
\end{lemma}
\begin{proof}
It follows from \eqref{eq:principal}, \eqref{amp:eq:pressuregrowth} and \eqref{amp:eq:frozenjets}
\[
\begin{aligned}
p_t &= h\sin\theta\,(v\cdot p)(I-p\otimes p)n
       -(I-p\otimes p)(M-M_*)^Tp,\\
(\log A)_t &= h\cos\theta\,(p^\perp\cdot n)^2
       -p^\perp\cdot(H-H_*)p^\perp.
\end{aligned}
\]
On a bootstrap interval where $|p-n_+|\leq\theta$, we have
\[
|p-n|\leq |p-n_+|+|n_+-n| \leq\theta+2\sin\frac{\theta+\vartheta}{2}\leq3\theta.
\]
Since $n,v$ are orthonormal and $|p|=1$, we have
\[
|v\cdot p| = |(I-p\otimes p)n|
= |p^\perp\cdot n|\leq3\theta
\]
and hence
\begin{equation}\label{geo:eq:retderivatives}
|p_t|\leq Ch\theta^3+\eps h\leq Ch\theta^3,
\qquad |(\log A)_t|\leq Ch\theta^2+\eps h\leq Ch\theta^2.
\end{equation}
Integrating \eqref{geo:eq:retderivatives} from $T$ yields the first two
estimates in \eqref{geo:eq:retphase} on the bootstrap interval.
Choose $c$ so that $Cc\leq\min\{1/2,\log2\}$. The first estimate
improves to $|p-n_+|\leq\theta/2$, so it holds on the
whole interval by continuity. In particular, we have
\[
p_2\geq\cos\vartheta-\theta/2\geq1/2.
\]
The second estimate implies $h_+/2\leq A\leq2h_+$ and
\[
|A-h_+|+h_+|p-n_+|\leq Ch(\theta/\vartheta)^2.
\]

At $\phi_0(y)=0$ with $\chi(y)=1$ and $f'(0)=1$, we have
\[
g_\lambda=Ap,\qquad M_\lambda=Am(p)\otimes p,\qquad
H_\lambda=-Ap_2p\otimes p.
\]
Since $|p|=|n_+|=1$, $(n_+)_2=\cos\vartheta$, and
$m(p)=p_2p-e_2$, we have
\[
\norm{m(p)\otimes p-m(n_+)\otimes n_+}
+\norm{p_2p\otimes p-\cos\vartheta\,n_+\otimes n_+}
\leq6|p-n_+|,
\]
which implies
\[
\begin{aligned}
&|g_\lambda-h_+n_+|
+\norm{M_\lambda-h_+m(n_+)\otimes n_+}+\norm{H_\lambda+h_+\cos\vartheta\,n_+\otimes n_+}\\
\leq& C\bigl(|A-h_+|+h_+|p-n_+|\bigr)
\leq Ch(\theta/\vartheta)^2.
\end{aligned}
\]
For the background, it follows from \eqref{eq:basicjets} and
\eqref{geo:eq:retbackground}
\[
|g-g_*|\leq\norm{M-M_*}+\norm{H-H_*}\leq\eps h.
\]
Hence \eqref{amp:eq:frozenjets} and $\vartheta\leq\theta$ imply
\[
|g|+\norm M+\norm H\leq Ch\leq Ch(\theta/\vartheta)^2.
\]
The claim of the lemma follows from this estimate and the principal-triple estimate.
\end{proof}

If the coefficient estimates and terminal data hold uniformly on the
supporting material neighborhood, then $A\leq2h_+$ and $p_2\geq1/2$.
Thus \eqref{amp:eq:profile} extends \eqref{amp:eq:hessianupper} to
$[T,T+4/(h_+\vartheta^2)]$.

\begin{corollary}\label{geo:cor:exactretention}
Assume the hypotheses of Lemma~\ref{geo:lem:retention} hold along the
background trajectory through the origin. Let $(\rho^+,u^+,\pi^+)$ be
a smooth solution of \eqref{eq:ipm} on $[0,T+4/(h_+\vartheta^2)]$, and suppose
$X(t,0)=X^+(t,0)=0$, $\phi_0(0)=0$, $\chi(0)=1$, and $f'(0)=1$.
Assume for some $E\geq0$,
\[
\begin{aligned}
\sup_{0\leq t\leq T+4/(h_+\vartheta^2)}\Bigl(
&|\nabla_x\rho^+(t,0)-g(t)-g_\lambda(t,0)|
+\norm{\nabla_x u^+(t,0)-M(t)-M_\lambda(t,0)}\\
&+\norm{\nabla_x^2\pi^+(t,0)-H(t)-H_\lambda(t,0)}\Bigr)
\leq E,
\end{aligned}
\]
where $g$, $M$, and $H$ are evaluated along $X(t,0)=0$. Then
\begin{equation}\label{geo:eq:exactretention}
\begin{aligned}
\sup_{T\leq t\leq T+4/(h_+\vartheta^2)}\Bigl(
&|\nabla_x\rho^+(t,0)-h_+n_+|
+\norm{\nabla_x u^+(t,0)-h_+m(n_+)\otimes n_+}\\
&+\norm{\nabla_x^2\pi^+(t,0)+h_+\cos\vartheta\,n_+\otimes n_+}\Bigr)
\leq Ch(\theta/\vartheta)^2+E.
\end{aligned}
\end{equation}
\end{corollary}
\begin{proof}
It follows from Lemma~\ref{geo:lem:retention} at $y=0$ and the triangle inequality.
\end{proof}

The upper bound in \eqref{geo:eq:exactretention} is small relative to $h_+$
provided both $\frac h{h_+}(\theta/\vartheta)^2$ and $E/h_+$ are small.

\section{The main iterative theorem}\label{sec:iterative}
In this section we establish a perturbation theorem that can be applied iteratively to produce a new smooth solution 
of the IPM system \eqref{eq:ipm} from a background smooth solution. We emphasize that since the IPM equation is unforced, 
the initial data of the new solution is perturbed from that of the background solution.  

\subsection{The iterative perturbation theorem}\label{a:sec:statement}
Let $(\bar\rho,U,\bar\pi)$ be a smooth solution of
\eqref{eq:ipm} on $[0,S]$, where $0<S\leq1$, and let $X$ be its
particle map with $X(0,y)=y$.
Let $\mathcal O$ be a label neighborhood contained in a lifted torus
chart, and choose $\chi\in C_c^\infty(\mathcal O)$ with $0\leq\chi\leq1$.
Let $\phi_0$ be a smooth real phase on $\mathcal O$, and let
$\mu=\mu(t,y)>0$ on $[0,S]\times\mathcal O$ satisfy
\begin{equation}\label{a:eq:phasecarrier}
\xi = (D_yX)^{-T}\nabla_y\phi_0,\qquad
\partial_t\mu+m(\xi)\cdot g\,\mu = 0,\qquad \mu>0.
\end{equation}

Let $f\in C^\infty(\R/(2\pi\mathbb Z))$ satisfy
$\int_0^{2\pi}f(\omega)\,d\omega=0$.
For $n\in\mathbb N_0$ and a Banach space $B$, denote
\[
\norm{D^nw}_B=\sum_{|I|=n}\norm{\partial^I w}_B,
\qquad \norm{D^0w}_B=\norm w_B.
\]
Fix $\bar{\mu}\geq2$. All time-dependent data estimates below are uniform
on $[0,S]$. The $W^{8,\infty}$ norms are taken on $\T^2$ for background
quantities and on $\mathcal O$ for phase and amplitude quantities. For
\[
w\in \{D_yX,\ (D_yX)^{-1},\ M,\ H,\ \nabla_y\bar\rho(0),\ \nabla_y\phi_0,
\ \xi,\ |\xi|^{-2},\ \nabla_y\log\mu\},
\]
we assume
\begin{equation}\label{a:eq:data}
\norm{D^n w}_{W^{8,\infty}}\leq \bar{\mu}^{n+1}(n!)^2,
\qquad n\in\mathbb N_0,
\end{equation}
while
\begin{equation}\label{a:eq:data2}
\norm{D^n w}_{H^8}\leq \bar{\mu}^{n+1}(n!)^2,
\qquad n\in\mathbb N_0, \qquad w\in \{\bar\rho(0), \chi, f\}.
\end{equation}
In \eqref{a:eq:data2}, the norms are on $\T^2$ for $\bar\rho(0)$ and the
zero extension of $\chi$, and on $\R/(2\pi\mathbb Z)$ for $f$.
We also assume
\begin{equation}\label{a:eq:carriersize}
0<\mu\leq \bar{\mu},\qquad
\sup_{y\in\mathcal O}\mu(0,y)\leq \bar{\mu}\mathrm e^{-\Gamma},\qquad \Gamma\geq0,
\end{equation}
and let $K_+\in L^1(0,S)$ be nonnegative with
\begin{equation}\label{a:eq:positiveH}
H(t,y)\leq K_+(t)I, \quad y\in\T^2; \qquad
\exp\left(\int_0^S K_+(t)\,dt\right)\leq \bar{\mu}.
\end{equation}

The principal increments evaluated in background particle coordinates are
\begin{equation}\label{a:eq:principaljets}
\begin{aligned}
g_\lambda &= \mu\chi\xi f'(\lambda\phi_0),\\
M_\lambda &= \mu\chi m(\xi)\otimes\xi f'(\lambda\phi_0),\\
H_\lambda &= -\mu\chi\frac{\xi_2}{|\xi|^2}\xi\otimes\xi f'(\lambda\phi_0).
\end{aligned}
\end{equation}
In label coordinates, they are extended by zero outside $\supp\chi$. Note
$M_\lambda=-H_\lambda-e_2\otimes g_\lambda$.

\begin{theorem}\label{thm:main}
Under assumptions
\eqref{a:eq:data}--\eqref{a:eq:positiveH}, there are absolute constants
$C,Q\geq1$ such that, for every
$\lambda\geq\bar{\mu}^Q$, there is a smooth solution
$(\rho^+,u^+,\pi^+)$ of \eqref{eq:ipm} on $[0,S]$ satisfying
\begin{equation}\label{a:eq:jeterror}
\begin{aligned}
\sup_{0\leq t\leq S}\Bigl(
&\norm{(\nabla_x\rho^+-\nabla_x\bar\rho)\circ X-g_\lambda}_{L^\infty_y}
+\norm{(\nabla_xu^+-\nabla_xU)\circ X-M_\lambda}_{L^\infty_y}\\
&+\norm{(\nabla_x^2\pi^+-\nabla_x^2\bar\pi)\circ X-H_\lambda}_{L^\infty_y}
\Bigr)\leq\lambda^{-1/2}.
\end{aligned}
\end{equation}
Its initial density is $\rho^+(0)=\bar\rho(0)+r_\lambda(0)+b_\lambda$, where
\[
r_\lambda(0,y)=\lambda^{-1}\mu(0,y)\chi(y)f(\lambda\phi_0(y))
\]
is extended by zero outside $\mathcal O$, and
\begin{equation}\label{a:eq:initial}
\begin{split}
&\supp b_\lambda\subset\supp\nabla_y\bar\rho(0),\\
&\norm{r_\lambda(0)}_{H^m}
\leq C_m\bar{\mu}^{C_m}\mathrm e^{-\Gamma}\lambda^{m-1},\qquad
\norm{b_\lambda}_{H^m}\leq C_m\bar{\mu}^{C_m}\lambda^{-2},
\quad m\in\mathbb N_0
\end{split}
\end{equation}
for some constant $C_m$ depending only on $m$.
Its particle map $X^+$ starting from the identity satisfies
\begin{equation}\label{a:eq:mapzero}
\sup_{0\leq t\leq S}\left(
\norm{X^+-X}_{L^\infty_y}
+\norm{(X^+)^{-1}-X^{-1}}_{L^\infty_x}
\right)\leq \bar{\mu}^C\lambda^{-1},
\end{equation}
\begin{equation}\label{a:eq:maphigh}
\sup_{0\leq t\leq S}\left(
\norm{D^n(D_yX^+)}_{L^\infty_y}+\norm{D^n((D_yX^+)^{-1})}_{L^\infty_y}
\right)\leq \bar{\mu}^C(\lambda \bar{\mu}^C)^n(n!)^2, \quad n\in\mathbb N_0.
\end{equation}
For every $n\in\mathbb N_0$, the physical inverse map and gradients satisfy
\begin{equation}\label{a:eq:inversehigh}
\sup_{0\leq t\leq S}
\norm{D_x^n\bigl(D_x(X^+)^{-1}\bigr)}_{L^\infty_x}
\leq\bar{\mu}^C(\lambda\bar{\mu}^C)^n(n!)^2,
\end{equation}
\begin{equation}\label{a:eq:fieldhigh}
\begin{aligned}
\sup_{0\leq t\leq S}\Bigl(
&\norm{D_y^n((\nabla_x\rho^+)\circ X^+)}_{L^\infty_y}
+\norm{D_y^n((\nabla_xu^+)\circ X^+)}_{L^\infty_y}\\
&+\norm{D_y^n((\nabla_x^2\pi^+)\circ X^+)}_{L^\infty_y}\Bigr)
\leq\bar{\mu}^C(\lambda\bar{\mu}^C)^n(n!)^2.
\end{aligned}
\end{equation}
If $\mathcal O=-\mathcal O$, $\bar\rho(0),\phi_0,f$ are odd, and
$\chi,\mu(0)$ are even, the solution can be chosen so that
\[
\begin{aligned}
\rho^+(t,-x)&=-\rho^+(t,x),& u^+(t,-x)&=-u^+(t,x),\\
\pi^+(t,-x)&=\pi^+(t,x),& X^+(t,0)&=0.
\end{aligned}
\]
\end{theorem}
The proof is given in Subsections~\ref{sec:insertion} and~\ref{sec:realization}.

\subsection{The approximate solution}\label{sec:insertion}
Starting from the principal oscillation, we construct an approximation to \eqref{eq:ipm} with an exponentially small residual. 

\subsubsection{Gevrey estimates}\label{a:sec:calculus}
Physical differentiation expressed in particle coordinates is given by
\begin{equation}\label{a:eq:d}
d_i = \sum_a((D_yX)^{-1})_{ai}\partial_{y_a},\qquad
\mathbf d = (d_1,d_2).
\end{equation}
Thus $d_i(q\circ X)=(\partial_{x_i}q)\circ X$ and $d_1d_2=d_2d_1$.
Since $\det(D_yX)=1$, the Piola identity takes the form
\begin{equation}\label{a:eq:piola}
(\diver_x w)\circ X = \diver_y\bigl((D_yX)^{-1}(w\circ X)\bigr).
\end{equation}

For $0<\kappa\leq1$, $\gamma>0$, and an integer $s\geq0$, 
we define the semiclassical Sobolev norms and Gevrey norms
\begin{equation}\label{a:eq:scalednorm}
\begin{split}
&\norm w_{H^s_{\kappa}}^2 =
\sum_{|I|\leq s}\kappa^{2|I|}\norm{\partial_y^Iw}_2^2,\qquad
\norm w_{W^{s,\infty}_{\kappa}} =
\sum_{|I|\leq s}\kappa^{|I|}\norm{\partial_y^Iw}_\infty,\\
&\G_{\gamma,\kappa}(w) =
\sum_{m\geq0}\frac{\gamma^m\kappa^m}{(m!)^2}
\norm{D^m w}_{H^4_{\kappa}},\qquad
\mathcal A_{\gamma,\kappa}(w)=
\sum_{m=0}^{\infty}
\frac{\gamma^m\kappa^m}{(m!)^2}
\|D^m w\|_{W^{4,\infty}_{\kappa}}.
\end{split}
\end{equation}
An omitted coordinate subscript here denotes $y$; we will also use the norms with the subscript $x$ 
which specifies physical derivatives.

For $\Lambda_{\mathrm G}\geq2$ and an integer $d\geq0$, we say that
$w$ has shift $d$ if
\begin{equation}\label{a:eq:shift}
\norm{D^n w}_{H^8}\leq \Lambda_{\mathrm G}^{n+d}((n+d)!)^2, \qquad n\geq0.
\end{equation}
For coefficients depending on $(y,\omega)$, we use $D^n=D_{y,\omega}^n$
and $H^8=H^8(\T^2\times(\R/(2\pi\mathbb Z)))$, extending compactly
supported local coefficients by zero in $y$. For multipliers, we replace
$H^8$ by $W^{8,\infty}$ on their domains. We choose
$\Lambda_{\mathrm G}=\bar{\mu}^{C_0}$.

The radius can be chosen independently of $d$, since
$((n+d)!)^2\leq4^{n+d}(n!)^2(d!)^2$ yields
\[
\sum_{n\geq0}\frac{(8\Lambda_{\mathrm G})^{-n}}{(n!)^2}\norm{D^n w}_{H^8}
\leq2(4\Lambda_{\mathrm G})^d(d!)^2.
\]

\begin{lemma}\label{a:lem:calculus}
The following estimates hold with constants independent of the nonnegative
integer shifts.

\smallskip\noindent
(a) For the spaces $H^8$ and $W^{8,\infty}$, a product of coefficients with shifts $d_1,d_2$ has shift $d_1+d_2+q$
for a fixed integer $q$. A coordinate derivative adds one shift. Multiplication
by a coefficient satisfying \eqref{a:eq:data} adds a fixed shift when
$\Lambda_{\mathrm G}=\bar{\mu}^{C_0}$ is sufficiently large. 

\smallskip\noindent
(b) Let $c\geq0$ be fixed. Suppose $Z_n,B_n\geq0$ satisfy
\begin{equation}\label{a:eq:triangularabstract}
Z_n\leq \bar{\mu}^c\left(B_n+\sum_{l=1}^n\binom nl\bar{\mu}^{c(l+1)}(l!)^2Z_{n-l}\right)
\end{equation}
and $B_n\leq \Lambda_{\mathrm G}^{n+d}((n+d)!)^2$. For sufficiently large $C_0$ and an integer $q$, we have
\[
Z_n\leq \Lambda_{\mathrm G}^{n+d+q}((n+d+q)!)^2, \qquad n\geq0.
\]

\smallskip\noindent
(c) Let $A,L\geq0$, $\Lambda\geq1$, and suppose
\[
\norm{D^lw}_\infty\leq A\Lambda^l((l+d)!)^2, \quad l\geq0; \qquad
\norm{D^k\Psi}_\infty\leq L\Lambda^{k-1}(k!)^2, \quad k\geq1.
\]
The image of $\Psi$ is assumed to lie in the domain of $w$. Then
\begin{equation}\label{a:eq:composition}
\norm{D^n(w\circ\Psi)}_\infty
\leq A[C(1+L)\Lambda]^n((n+d)!)^2,\qquad n\geq0.
\end{equation}
Suppose $\Psi:\Omega\to\Omega'$ is a diffeomorphism, and denote
$J=\norm{\det D\Psi^{-1}}_{L^\infty(\Omega')}<\infty$.
If $\norm{D^lw}_2\leq A\Lambda^l((l+d)!)^2$ for $l\geq0$, then
\[
\norm{D^n(w\circ\Psi)}_{L^2(\Omega)}
\leq J^{1/2}A[C(1+L)\Lambda]^n((n+d)!)^2.
\]
For every integer $s\geq0$, it follows that
\[
\norm{D^n(w\circ\Psi)}_{H^s(\Omega)}
\leq C_sJ^{1/2}A[C(1+L)\Lambda]^{n+s}((n+d+s)!)^2.
\]
In (c), the constants depend only on the fixed dimensions and on $s$ where indicated.
\end{lemma}
\begin{proof}
For (a), the algebra property of $H^8$ in dimensions two and three and
Leibniz's rule reduce the product estimate to
\begin{equation}\label{a:eq:factorialproduct}
\binom nl\frac{((l+d_1)!)^2((n-l+d_2)!)^2}
{((n+d_1+d_2)!)^2}\leq\binom nl^{-1},\qquad
\sum_{l=0}^n\binom nl^{-1}\leq3.
\end{equation}
More generally, suppose a multiplier $b$ satisfies
$\norm{D^l b}_{W^{8,\infty}}\leq \bar{\mu}^{c(l+1)}(l!)^2$ with fixed $c\geq0$.
For a coefficient $w$ with shift $d$, Leibniz's rule yields, whenever
$\Lambda_{\mathrm G}\geq2\bar{\mu}^c$,
\[
\begin{aligned}
\frac{\norm{D^n(bw)}_{H^8}}
{\Lambda_{\mathrm G}^{n+d}((n+d)!)^2}
&\leq C\bar{\mu}^c\sum_{l=0}^n
\left(\frac{\bar{\mu}^c}{\Lambda_{\mathrm G}}\right)^l
\binom nl\frac{(l!)^2((n-l+d)!)^2}{((n+d)!)^2}\\
&\leq C\bar{\mu}^c\sum_{l=0}^n
\left(\frac{\bar{\mu}^c}{\Lambda_{\mathrm G}}\right)^l
\leq\frac{C\bar{\mu}^c}{1-\bar{\mu}^c/\Lambda_{\mathrm G}}
\leq2C\bar{\mu}^c.
\end{aligned}
\]
For $C_0\geq c+1$ and a fixed integer $q\geq1$ with $2^q\geq2C$,
we have $2C\bar{\mu}^c\leq\Lambda_{\mathrm G}^q$.
This proves the assertion in both spaces, uniformly in $n$, $d$, and $\bar{\mu}$.

For (b), choose $C_0\geq3c+2$ and $q\geq1$.
Assuming the conclusion at orders below $n$, apply
\eqref{a:eq:factorialproduct} with $d_1=0$ and $d_2=d+q$ to
\eqref{a:eq:triangularabstract} to obtain
\[
\begin{aligned}
\frac{Z_n}{\Lambda_{\mathrm G}^{n+d+q}((n+d+q)!)^2}
&\leq \bar{\mu}^c\Lambda_{\mathrm G}^{-q}
+\bar{\mu}^{2c}\sum_{l=1}^n
\left(\frac{\bar{\mu}^c}{\Lambda_{\mathrm G}}\right)^l\\
&\leq \bar{\mu}^c\Lambda_{\mathrm G}^{-q}
+\bar{\mu}^{2c}\frac{\bar{\mu}^c/\Lambda_{\mathrm G}}
{1-\bar{\mu}^c/\Lambda_{\mathrm G}}
\leq\frac14+\frac13<1.
\end{aligned}
\]
This also covers $n=0$, when the sum is empty, and completes the induction.

For (c), the case $n=0$ follows from $\norm{w\circ\Psi}_\infty\leq\norm w_\infty$.
For $n\geq1$, Fa\`a di Bruno's formula yields
\[
\norm{D^n(w\circ\Psi)}_\infty
\leq C^n\sum_{l=1}^n
\sum_{\substack{k_1+\cdots+k_l=n\\ k_i\geq1}}
\frac{n!}{l!\prod_{i=1}^l k_i!}
\norm{D^lw}_\infty
\prod_{i=1}^l\norm{D^{k_i}\Psi}_\infty,
\]
where $C\geq1$ depends only on the dimensions.
For each ordered composition $k_1+\cdots+k_l=n$ with $k_i\geq1$,
\begin{equation}\label{a:eq:onefrequency}
\Lambda^l\prod_{i=1}^l\Lambda^{k_i-1}
=\Lambda^{l+\sum_{i=1}^l(k_i-1)}=\Lambda^n.
\end{equation}
Since $l!\prod_{i=1}^l k_i!\leq n!$, we have
\[
\frac{n!}{l!\prod_{i=1}^l k_i!}
\frac{((l+d)!)^2\prod_{i=1}^l(k_i!)^2}{((n+d)!)^2}
=\frac{l!\prod_{i=1}^l k_i!}{n!}
\left(\prod_{j=1}^d\frac{l+j}{n+j}\right)^2
\leq1.
\]
The derivative estimates therefore yield
\[
\norm{D^n(w\circ\Psi)}_\infty
\leq A(C\Lambda)^n((n+d)!)^2
\sum_{l=1}^n\binom{n-1}{l-1}L^l.
\]
Together with
\[
\sum_{l=1}^n\binom{n-1}{l-1}L^l
=L(1+L)^{n-1}\leq(1+L)^n,
\]
this verifies \eqref{a:eq:composition}.

For the $L^2$ version, we infer
\[
\norm{(D^lw)\circ\Psi}_{L^2(\Omega)}
\leq J^{1/2}\norm{D^lw}_{L^2(\Omega')}.
\]
The same partition estimate applies. Finally, $\norm{D^n(w\circ\Psi)}_{H^s}\leq C_s\sum_{j=0}^s\norm{D^{n+j}(w\circ\Psi)}_{L^2}$
yields the stated estimate.
\end{proof}

In (c), the derivative scale becomes $C(1+L)\Lambda$ which is independent of $d$.
By \eqref{a:eq:data}, we have
\[
|D_y^k\log\mu|\leq\bar{\mu}^k((k-1)!)^2,\qquad k\geq1.
\]
We thus infer from the same factorial and composition-count estimates as in (c), for a fixed $c$,
\begin{equation}\label{a:eq:murelative}
\begin{split}
\mu^{-1}|D_y^n\mu|
&\leq C^n\sum_{l=1}^n
\sum_{\substack{k_1+\cdots+k_l=n\\ k_i\geq1}}
\frac{n!}{l!\prod_{i=1}^l k_i!}
\prod_{i=1}^l|D_y^{k_i}\log\mu|\\
&\leq(C\bar{\mu})^n(n!)^2
\leq\bar{\mu}^c(\bar{\mu}^c)^n(n!)^2,\qquad n\geq1.
\end{split}
\end{equation}
Applying $\mu\leq\bar{\mu}$ and these estimates through order $n+8$,
part (a) shows that multiplication by $\mu$ adds a fixed shift for sufficiently large $C_0$.
The pulled-back physical derivatives obey
\begin{equation}\label{a:eq:normalizedderivative}
\mu^{-1}d_i(\mu a)=d_i a+(d_i\log\mu)a.
\end{equation}
The coefficients $(D_yX)^{-1}$ and $\nabla_y\log\mu$ satisfy the data estimates, so
part (a) shows that this normalized derivative also adds a fixed shift. 

\subsubsection{Backward inverse for the linearized equation}\label{a:sec:slow}
The phase-independent coefficients satisfy the linearized equation
\begin{equation}\label{a:eq:linearized}
r_t+U\cdot\nabla_x r+(\cT r)\cdot\nabla_x\bar\rho = \mathfrak f.
\end{equation}
We select initial data $r(0)$ through a terminal condition
on an auxiliary displacement. 

Let $0<S\leq1$, and suppose
\begin{equation}\label{source-est}
\int_0^S\norm{D^n(\mathfrak f\circ X)}_{H^8}\,dt
\leq \Lambda_{\mathrm G}^{n+d}((n+d)!)^2, \qquad n\geq0.
\end{equation}

\begin{lemma}\label{a:lem:slow}
Under the background hypotheses \eqref{a:eq:data}, \eqref{a:eq:positiveH} and \eqref{source-est}, 
for a sufficiently large fixed $C_0$ there
is a linear choice of solution $r(t)$ to \eqref{a:eq:linearized} such that
$\supp r(0)\subset\supp\nabla_y\bar\rho(0)$ and
\[
\sup_{0\leq t\leq S}\bigl\{
\norm{D^n(r\circ X)}_{H^8}+\norm{D^n((\cT r)\circ X)}_{H^8}\bigr\}
\leq \Lambda_{\mathrm G}^{n+d+q}((n+d+q)!)^2, \qquad n\geq0,
\]
where the integer $q$ and $C_0$ are independent of $d$, $n$, $S$, and
$\bar{\mu}$. Moreover, the solution $r$ is smooth if the source $\mathfrak f$ is smooth.
\end{lemma}
\begin{proof}
Denote $a(y)=\nabla_y\bar\rho(0,y)$ and define
\begin{equation}\label{a:eq:b}
b(t,y) = \int_0^t \mathfrak f(s,X(s,y))\,ds.
\end{equation}
We seek a divergence-free vector field $z$ in label coordinates satisfying
\begin{equation}\label{a:eq:slowdisplacement}
((D_yX)^T(D_yX))z_t+\nabla_y\varpi = (D_yX)^Te_2(a\cdot z-b),\qquad
\diver_y z = 0,\qquad z(S) = 0
\end{equation}
with a pressure function $\varpi$.

Let $G=(D_yX)^T(D_yX)$, $B=(D_yX)^Te_2\otimes a$, and let $\cP_y$ be the periodic Helmholtz
projection onto the closed subspace
$\mathcal V=\{w\in L^2(\T^2;\R^2):\diver_y w=0\}$. On $\mathcal V$, we have
\[
\bar{\mu}^{-2}\norm w_2^2\leq\langle w,\cP_y G\cP_y w\rangle
=\norm{(D_yX)w}_2^2\leq \bar{\mu}^2\norm w_2^2.
\]
Thus $\cP_y G\cP_y$ is invertible on $\mathcal V$, with inverse norm at most
$\bar{\mu}^2$. Projecting \eqref{a:eq:slowdisplacement} yields the linear equation
\[
z_t=(\cP_y G\cP_y|_{\mathcal V})^{-1}\cP_y(Bz-(D_yX)^Te_2b).
\]
The solution with zero terminal value is unique. Helmholtz decomposition recovers a
periodic pressure $\varpi$, which we normalize to have zero mean. 

Define physical fields $r,v$ by
\begin{equation}\label{a:eq:slowrecovery}
r\circ X = b-a\cdot z,\qquad v\circ X = (D_yX)z_t.
\end{equation}
By \eqref{a:eq:slowdisplacement} and the Piola identity \eqref{a:eq:piola},
\[
v=-\nabla_x(\varpi\circ X^{-1})-re_2,\qquad
(\diver_xv)\circ X=\diver_yz_t=0,
\]
so $v=\cT r$. Since $g=(D_yX)^{-T}a$,
\[
\partial_t(r\circ X)+(v\circ X)\cdot g
=\mathfrak f\circ X-a\cdot z_t+((D_yX)z_t)\cdot(D_yX)^{-T}a
=\mathfrak f\circ X,
\]
which verifies \eqref{a:eq:linearized}. At $t=0$,
\begin{equation}\label{a:eq:slowsupport}
r(0)=-a\cdot z(0),
\end{equation}
which implies the initial support inclusion. Since
$\int_{\T^2}a\cdot z\,dy=-\int_{\T^2}\bar\rho(0)\diver_yz\,dy=0$,
\[
\int_{\T^2}r(t,x)\,dx=\int_{\T^2}b(t,y)\,dy
=\int_0^t\int_{\T^2}\mathfrak f(s,x)\,dx\,ds;
\]
in particular, $r(0)$ has zero spatial mean.
If a closed set $K\subset\T^2$ contains $\supp(\mathfrak f(s)\circ X(s))$
for almost every $s\in[0,t]$, then
\[
\supp(r(t)\circ X(t))\subset K\cup\supp a.
\]
If the background density and source are odd in physical coordinates,
reflection and uniqueness in \eqref{a:eq:slowdisplacement} imply that $z$, $r$, and $v$ are odd.

For the derivative estimates, consider the same operator with a vector source:
\begin{equation}\label{a:eq:slowgeneral}
Gz_t-Bz+\nabla_y\varpi=\mathfrak f_0,\qquad
\diver_y z=0,\qquad z(S)=0.
\end{equation}
The preceding existence argument applies. Using
$\partial_t(D_yX)=M(D_yX)$ and testing with $z$ yields
\begin{equation}\label{a:eq:slowenergy}
\frac12\frac{d}{dt}\norm{(D_yX)z}_2^2+
\int_{\T^2}((D_yX)z)\cdot H((D_yX)z)\,dy
=\int_{\T^2}z\cdot\mathfrak f_0\,dy.
\end{equation}
Indeed, the pressure term vanishes, while $a\cdot z=g\cdot(D_yX)z$ and
$M=-H-e_2\otimes g$ imply
\[
((D_yX)z)\cdot M((D_yX)z)+((D_yX)z)_2(a\cdot z)
=-((D_yX)z)\cdot H((D_yX)z).
\]
Since $H\leq K_+I$ and $z(S)=0$, backward integration yields
\[
\norm{(D_yX(t))z(t)}_2\leq
\int_t^S\exp\left(\int_t^sK_+(\tau)\,d\tau\right)
\norm{(D_yX(s))^{-T}\mathfrak f_0(s)}_2\,ds,
\]
where the exponential is at most $\bar{\mu}$ by \eqref{a:eq:positiveH}.
For $\mathfrak f_0=-(D_yX)^Te_2b$, this yields
\begin{equation}\label{a:eq:slowbound}
\norm{(D_yX(t))z(t)}_2\leq
\int_t^S\exp\left(\int_t^sK_+(\tau)\,d\tau\right)\norm{b(s)}_2\,ds.
\end{equation}
The general vector-source bound and $S\leq1$ imply
$\norm z_{L^\infty_tL^2_y}\leq\bar{\mu}^3\norm{\mathfrak f_0}_{L^\infty_tL^2_y}$.
Testing the projected equation with $z_t$ and using $\norm B_\infty\leq\bar{\mu}^2$ yields
\[
\bar{\mu}^{-2}\norm{z_t}_2^2
\leq(\bar{\mu}^2\norm z_2+\norm{\mathfrak f_0}_2)\norm{z_t}_2,
\]
and hence
\[
\norm z_{L^\infty_tL^2_y}+\norm{z_t}_{L^\infty_tL^2_y}
\leq C\bar{\mu}^7\norm{\mathfrak f_0}_{L^\infty_tL^2_y}.
\]

Let $I\in\mathbb N_0^2$ with $|I|=n$.
It follows from \eqref{a:eq:slowgeneral} that
\begin{equation}\label{a:eq:slowcommutators}
G\partial_t\partial^Iz-B\partial^Iz+\nabla_y\partial^I\varpi
=\partial^I\mathfrak f_0
+\sum_{0<J\leq I}\binom IJ
\Bigl[(\partial^JB)\partial^{I-J}z-(\partial^JG)\partial^{I-J}z_t\Bigr].
\end{equation}
Moreover, $\diver_y\partial^Iz=0$ and $\partial^Iz(S)=0$.
By \eqref{a:eq:data} and \eqref{a:eq:factorialproduct},
\[
\norm{D^lG}_{L^\infty_tL^\infty_y}
+\norm{D^lB}_{L^\infty_tL^\infty_y}
\leq3C\bar{\mu}^{l+2}(l!)^2
\leq\bar{\mu}^{c(l+1)}(l!)^2,\qquad l\geq0,
\]
for a fixed $c\geq0$. Define
\[
Z_n=\norm{D^nz}_{L^\infty_tL^2_y}
+\norm{D^nz_t}_{L^\infty_tL^2_y}.
\]
Applying the vector-source estimate to \eqref{a:eq:slowcommutators}, summing over $|I|=n$,
and enlarging $c$ by a fixed amount yields
\begin{equation}\label{a:eq:slowtriangular}
Z_n\leq\bar{\mu}^c\left[
\norm{D^n\mathfrak f_0}_{L^\infty_tL^2_y}
+\sum_{l=1}^n\binom nl\bar{\mu}^{c(l+1)}(l!)^2Z_{n-l}\right],
\qquad n\geq0.
\end{equation}

For the scalar inverse, take $\mathfrak f_0=-(D_yX)^Te_2b$. By \eqref{a:eq:b} and \eqref{source-est},
\[
\norm{D^nb}_{L^\infty_tH^8_y}
\leq\int_0^S\norm{D^n(\mathfrak f\circ X)}_{H^8}\,dt
\leq\Lambda_{\mathrm G}^{n+d}((n+d)!)^2.
\]
Lemma~\ref{a:lem:calculus}(a) therefore yields a fixed integer $q_0\geq0$ with
\[
\norm{D^n\mathfrak f_0}_{L^\infty_tL^2_y}
\leq\norm{D^n\mathfrak f_0}_{L^\infty_tH^8_y}
\leq\Lambda_{\mathrm G}^{n+d+q_0}((n+d+q_0)!)^2.
\]
Applying Lemma~\ref{a:lem:calculus}(b) to \eqref{a:eq:slowtriangular},
with input shift $d+q_0$ and one additional shift, yields
\[
Z_n=\norm{D^nz}_{L^\infty_tL^2_y}
+\norm{D^nz_t}_{L^\infty_tL^2_y}
\leq\Lambda_{\mathrm G}^{n+d+q_0+1}((n+d+q_0+1)!)^2.
\]
Finally, we obtain
\[
\norm{D^nz}_{L^\infty_tH^8_y}
+\norm{D^nz_t}_{L^\infty_tH^8_y}
\leq C\sum_{j=0}^8 Z_{n+j}
\leq9C\Lambda_{\mathrm G}^{n+d+q_0+9}((n+d+q_0+9)!)^2.
\]
The recovery formulas \eqref{a:eq:slowrecovery} and Lemma~\ref{a:lem:calculus}(a)
then imply the asserted bound, after absorbing the coefficient products and fixed constants into a fixed total shift $q$.
\end{proof}

\subsubsection{Expansion of the localized multiplier}\label{a:sec:multiplier}
We construct local pressure coefficients and estimate the nonlocal velocity remainder.

Let $r(t,y,\omega)$ be smooth, $2\pi$-periodic with zero mean in $\omega$,
and supported in a fixed compact label subset of $\mathcal O$. For any such coefficient $a$, denote
\[
a^\lambda(t,X(t,y))=a(t,y,\lambda\phi_0(y)),
\]
with zero extension outside $X(t,\mathcal O)$. Compact label support makes
$a^\lambda$ smooth and periodic for every real $\lambda\geq1$.

We have
\[
(\partial_{x_i}a^\lambda)\circ X
=\bigl(d_i a+\lambda\xi_i\partial_\omega a\bigr)
(t,y,\lambda\phi_0(y)).
\]
Define
\begin{equation}\label{a:eq:ellipticoperators}
L_0 = -|\xi|^2\partial_\omega^2,\qquad
L_1 = -(2\xi\cdot\mathbf d+\mathbf d\cdot\xi)\partial_\omega,\qquad
L_2 = -\mathbf d\cdot\mathbf d
\end{equation}
Thus we have
\[
\begin{aligned}
-\Delta_x a^\lambda
&=\bigl(\lambda^2L_0a+\lambda L_1a+L_2a\bigr)^\lambda,\\
\partial_{x_2}r^\lambda
&=\bigl(\lambda\xi_2\partial_\omega r+d_2r\bigr)^\lambda.
\end{aligned}
\]

Let $\partial_\omega^{-1}$ denote the zero-mean periodic primitive, and
$L_0^{-1}=-|\xi|^{-2}\partial_\omega^{-2}$.
Matching powers of $\lambda$ in $-\Delta_x\pi=\partial_{x_2}r^\lambda$ we define
\begin{equation}\label{a:eq:Qrecursion}
\begin{aligned}
Q_0[r] &= L_0^{-1}(\xi_2\partial_\omega r),\\
Q_1[r] &= L_0^{-1}(d_2r-L_1Q_0[r]),\\
Q_j[r] &= -L_0^{-1}(L_1Q_{j-1}[r]+L_2Q_{j-2}[r]),\qquad j\geq2.
\end{aligned}
\end{equation}
We then obtain the velocity coefficients from Darcy's law 
\begin{equation}\label{a:eq:Trecursion}
\cT_0r = m(\xi)r,\qquad
\cT_jr = -\xi\partial_\omega Q_j[r]-\mathbf d Q_{j-1}[r], \quad j\geq1.
\end{equation}
The same recursion implies the divergence-free identities
\begin{equation}\label{a:eq:divhierarchy}
\xi\cdot\partial_\omega\cT_0r = 0,\qquad
\mathbf d\cdot\cT_{j-1}r+\xi\cdot\partial_\omega\cT_jr = 0, \quad j\geq1.
\end{equation}

\begin{lemma}\label{a:lem:localcosts}
There is a fixed integer $q$ such that $Q_j$, $\cT_j$ and their conjugates
$\mu^{-1}Q_j(\mu\,\cdot)$, $\mu^{-1}\cT_j(\mu\,\cdot)$ increase the derivative shift
by at most $q(j+1)$ for $j\geq0$. The pressure-defect operators
$L_1Q_L+L_2Q_{L-1}$ and $L_2Q_L$ increase it by at most $q(L+1)$ for $L\geq1$.
\end{lemma}
\begin{proof}
Let $\mathcal C_d$ denote the supported zero-phase-mean coefficients
satisfying \eqref{a:eq:shift} uniformly on $[0,S]$. Since $\Lambda_{\mathrm G}\geq2$,
\[
\mathcal C_d\subset\mathcal C_{d+1},\qquad
\mathcal C_d+\mathcal C_d\subset\mathcal C_{d+1}.
\]
For $a\in\mathcal C_d$ and $p\in\{1,2\}$,
\[
\norm{D^m\partial_\omega^{-p}a}_{H^8}\leq\norm{D^ma}_{H^8}.
\]
As $\mu$ is independent of $\omega$, we also have
\[
\mu^{-1}\partial_\omega^{-p}(\mu a)=\partial_\omega^{-p}a,\qquad
\mu^{-1}d_i(\mu a)=d_i a+(d_i\log\mu)a.
\]
By \eqref{a:eq:data}, \eqref{a:eq:normalizedderivative}, and
Lemma~\ref{a:lem:calculus}(a), there is a fixed integer $s\geq1$ such that
\[
A(\mathcal C_d),\quad[\mu^{-1}A(\mu\,\cdot)](\mathcal C_d)
\subset\mathcal C_{d+s}
\]
for every
\[
A\in\left\{
L_0^{-1}\xi_2\partial_\omega,\ L_0^{-1}d_2,\ L_0^{-1}L_1,\ L_0^{-1}L_2,
\ \xi\partial_\omega,\ \mathbf d,\ m(\xi),\ L_1,\ L_2
\right\}.
\]
Induction in \eqref{a:eq:Qrecursion}, with cost at most $s+1$ per order
including the sum, yields
\[
Q_j(\mathcal C_d),\quad[\mu^{-1}Q_j(\mu\,\cdot)](\mathcal C_d)
\subset\mathcal C_{d+(s+1)(j+1)},\qquad j\geq0.
\]
The ordinary and conjugated recursions have the same elementary costs.
It follows from \eqref{a:eq:Trecursion} that
\[
\cT_j(\mathcal C_d),\quad[\mu^{-1}\cT_j(\mu\,\cdot)](\mathcal C_d)
\subset\mathcal C_{d+(s+1)(j+2)},\qquad j\geq0.
\]
Similarly, for $L\geq1$,
\[
(L_1Q_L+L_2Q_{L-1})(\mathcal C_d),\quad(L_2Q_L)(\mathcal C_d)
\subset\mathcal C_{d+(s+1)(L+2)}.
\]
The choice $q=2(s+1)$ proves the lemma.
\end{proof}

\begin{lemma}\label{a:lem:graph}
If a supported local coefficient $r$ has shift $d$, then for every 
integer $0\leq s\leq8$,
\begin{equation}\label{a:eq:graphestimate}
\norm{D_y^m(r^\lambda\circ X)}_{H^s_{1/\lambda,y}}
\leq (\lambda \bar{\mu}^c)^m \bar{\mu}^{c(d+1)}((m+d+c)!)^2.
\end{equation}
The same estimate holds for $\norm{D_x^m r^\lambda}_{H^s_{1/\lambda,x}}$.
The integer $c$ is independent of $d,m,\lambda$. 
\end{lemma}
\begin{proof}
For any coefficient $a$, periodic Sobolev embedding in the phase variable
implies
\begin{equation}\label{est-trace}
\norm{a(y,\lambda\phi_0(y))}_{L^2_y}
\leq C\norm a_{L^2_yH^1_\omega}.
\end{equation}
Fix $t$ and, for $y\in\mathcal O$, denote
$\Psi_\lambda(y)=(y,\lambda\phi_0(y))$, so that
$r^\lambda\circ X=r\circ\Psi_\lambda$. By \eqref{a:eq:data}, we infer
\[
\norm{D_y^j\Psi_\lambda}_\infty
\leq C\lambda \bar{\mu}^j((j-1)!)^2
\leq\lambda \bar{\mu}^{c_0j}(j!)^2, \quad j\geq 1
\]
where $c_0$ is fixed. The shift bound \eqref{a:eq:shift} and the estimate \eqref{est-trace} imply for
$l\geq0$,
\[
\norm{(D_{y,\omega}^lr)\circ\Psi_\lambda}_{L^2_y}
\leq C\norm{D_{y,\omega}^lr}_{L^2_yH^1_\omega}
\leq C\norm{D_{y,\omega}^lr}_{H^8_{y,\omega}}
\leq C\Lambda_{\mathrm G}^{l+d+1}((l+d+1)!)^2.
\]
For $k_1+\cdots+k_l=k$ with $k_i\geq1$, the frequency factors satisfy
\[
\Lambda_{\mathrm G}^{l+d+1}\prod_{i=1}^l\lambda\bar{\mu}^{c_0k_i}
\leq\lambda^k\bar{\mu}^{(c_0+C_0)k+C_0(d+1)}.
\]
It follows from the Fa\`a di Bruno formula and the factorial and composition-count
estimates in Lemma~\ref{a:lem:calculus}(c), with outer shift $d+1$, that
\[
\norm{D_y^k(r^\lambda\circ X)}_{L^2_y}
\leq(\lambda\bar{\mu}^{c_1})^k\bar{\mu}^{c_1(d+1)}((k+d+1)!)^2,
\qquad k\geq0.
\]
Here $k=0$ follows directly from \eqref{est-trace}.

To pass to physical coordinates, let $Y=X(t,\cdot)^{-1}$ and
$h=(D_yX)^{-1}$, so that $D_xY=h\circ Y$.
In the following inverse-map estimate, $\|D^k\cdot\|_{\mathrm{op},\infty}$
denotes the supremum of the multilinear operator norm; conversion to the
derivative sums used elsewhere costs at most $C^k$ for $k\geq1$.
By \eqref{a:eq:data}, a fixed power $A=\bar{\mu}^C$ satisfies
\[
\|D^l h\|_{\infty}\leq A^{l+1}(l!)^2,
\qquad
u_k=\frac{\|D_x^kY\|_{\infty}}{(k!)^2}.
\]
We have $u_1\leq A$. For $k\geq2$, differentiating $D_xY=h\circ Y$
and using $l!\prod_i k_i!\leq(k-1)!$ for positive $k_i$ with
$\sum_i k_i=k-1$ yields
\[
u_k\leq\frac{A}{k^2}\sum_{l=1}^{k-1}A^l
\sum_{\substack{k_1+\cdots+k_l=k-1\\ k_i\geq1}}
\prod_{i=1}^l u_{k_i}.
\]
These quantities are majorized by the coefficients of the formal series
\[
\begin{split}
&\mathcal U(\zeta)=\frac{A\zeta}{1-A\mathcal U(\zeta)},\qquad
\mathcal U(0)=0,\\
&[\zeta^k]\mathcal U
=\frac1k\binom{2k-2}{k-1}A^{2k-1}
\leq A(4A^2)^{k-1}.
\end{split}
\]
Converting back to derivative sums and increasing a fixed exponent yields
\[
\norm{D_x^kY}_\infty
\leq\bar{\mu}^{c_2}(\bar{\mu}^{c_2})^{k-1}(k!)^2,
\qquad k\geq1.
\]
Since $r^\lambda=(r^\lambda\circ X)\circ Y$ and $\det D_xY=1$,
the $L^2$ composition estimate in Lemma~\ref{a:lem:calculus}(c),
with derivative scale $\lambda\bar{\mu}^C$, yields the same unweighted
derivative bound in physical coordinates, after increasing the exponent.

Finally, we have in either coordinate system
\[
\norm{D^m w}_{H^s_{1/\lambda}}
\leq C_s\sum_{j=0}^s\lambda^{-j}\norm{D^{m+j} w}_{L^2}
\]
which concludes the proof of the lemma.
\end{proof}

\begin{lemma}\label{a:lem:globalremainder}
For every integer $L\geq1$ and every supported coefficient $r$ with zero
phase mean, we have
\begin{equation}\label{a:eq:velocityexpansion}
\cT r^\lambda = \sum_{j=0}^{L-1}\lambda^{-j}(\cT_jr)^\lambda+\mathcal E_L[r].
\end{equation}
If $r$ has shift $d$, then
\begin{equation}\label{a:eq:velocityremainder}
\norm{D^m(\mathcal E_L[r]\circ X)}_{H^4_{1/\lambda}}
\leq \lambda^{-L}(\lambda \bar{\mu}^c)^m
\bar{\mu}^{c(d+L+1)}((m+d+c(L+1))!)^2,
\end{equation}
where the integer $c$ is independent of $m,d,L,\lambda$. 
\end{lemma}
\begin{proof}
Denote
\[
\Pi_L=\sum_{j=0}^{L}\lambda^{-j-1}(Q_j[r])^\lambda,\qquad
h_L=-\Delta_x\Pi_L-\partial_{x_2}r^\lambda,\qquad
w_L=-\nabla_x\Pi_L-r^\lambda e_2
\]
which are smooth and periodic. It follows from \eqref{a:eq:Qrecursion} that
\begin{equation}\label{a:eq:pressuredefect}
h_L = \lambda^{-L}(L_1Q_L+L_2Q_{L-1})^\lambda
+\lambda^{-L-1}(L_2Q_L)^\lambda.
\end{equation}
We also have
\[
w_L=\sum_{j=0}^{L}\lambda^{-j}(\cT_jr)^\lambda
-\lambda^{-L-1}(\mathbf d Q_L)^\lambda.
\]
The periodic Helmholtz projection $\cP_x=I-\nabla_x\Delta_x^{-1}\diver_x$
satisfies $\cP_xw_L=\cT r^\lambda$ and preserves
$\langle w_L\rangle_x=-e_2\langle r^\lambda\rangle_x$. Hence
\begin{equation}\label{a:eq:explicitvelocityerror}
\mathcal E_L[r] = \lambda^{-L}(\cT_Lr)^\lambda
-\lambda^{-L-1}(\mathbf d Q_L)^\lambda-\nabla_x\Delta_x^{-1}h_L.
\end{equation}
Since $h_L=\diver_xw_L$, we have $\langle h_L\rangle_x=0$ and
\begin{equation}\label{a:eq:lowfrequency}
\norm{\nabla\Delta^{-1}h_L}_2\leq\norm{h_L}_2,
\qquad
\norm{D_x^k\nabla\Delta^{-1}h_L}_2\leq C^k\norm{D_x^{k-1}h_L}_2,
\quad k\geq1.
\end{equation}
By Lemma~\ref{a:lem:localcosts}, the coefficients in
\eqref{a:eq:pressuredefect} and \eqref{a:eq:explicitvelocityerror}
have shifts at most $d+C(L+1)$.
Lemma~\ref{a:lem:graph} estimates $h_L$ and the two local terms in
\eqref{a:eq:explicitvelocityerror}; \eqref{a:eq:lowfrequency} yields the same
bound for $\nabla_x\Delta_x^{-1}h_L$. Consequently, we have for some integer $c_3$,
\[
\norm{D_x^k\mathcal E_L[r]}_{L^2_x}
\leq\lambda^{-L}(\lambda\bar{\mu}^{c_3})^k
\bar{\mu}^{c_3(d+L+1)}((k+d+c_3(L+1))!)^2,\qquad k\geq0.
\]
We apply Lemma~\ref{a:lem:calculus}(c) to $\mathcal E_L[r]\circ X$ with
$\det D_yX=1$, and together with
\[
\norm{D_y^m(\mathcal E_L[r]\circ X)}_{H^4_{1/\lambda,y}}
\leq C\sum_{j=0}^4\lambda^{-j}
\norm{D_y^{m+j}(\mathcal E_L[r]\circ X)}_{L^2_y},
\]
we conclude the proof of \eqref{a:eq:velocityremainder}.
\end{proof}

\subsubsection{Construction of the expansion coefficients}\label{a:sec:hierarchy}
At each order, we first solve for the phase average and then for the oscillatory coefficient.
Consider the approximation
\begin{equation}\label{a:eq:rhoapprox}
\rho^{a,N}\circ X = \bar\rho(0)+
\sum_{n=1}^N\lambda^{-n}(\bar r_n+r_n)(t,y,\lambda\phi_0(y)),
\end{equation}
where $\bar r_n$ is independent of $\omega$, $\av{r_n}=0$, $\bar r_1=0$,
and $r_1=\mu\chi f(\omega)$. 

Denote $\bar v_n=[\cT(\bar r_n\circ X^{-1})]\circ X$. In particular,
$\bar v_1=0$. The velocity coefficient at order $n$ is
\begin{equation}\label{a:eq:RnVn}
V_n = \bar v_n+\sum_{l=1}^{n}\cT_{n-l}r_l.
\end{equation}
For $n\geq2$, the coefficient of $\lambda^{-n}$ in the density equation,
after the local velocity expansion, is
\begin{equation}\label{a:eq:fullcoefficient}
\partial_t(\bar r_n+r_n)+V_n\cdot g+
\sum_{i+k=n}V_i\cdot \mathbf d(\bar r_k+r_k)+
\sum_{i+k=n+1}(V_i\cdot\xi)\partial_\omega r_k,
\end{equation}
where time derivative is taken at fixed $(y,\omega)$, and $\mathbf d$ differentiates
in $y$ with $\omega$ fixed. Since $m(\xi)\cdot\xi=0$,
\[
V_1\cdot\xi=0,\qquad (\cT_0r_n)\cdot\xi=0.
\]
Define the lower-order source
\begin{equation}\label{a:eq:source}
\mathcal F_n=-\sum_{l=1}^{n-1}\Bigl[
(\cT_{n-l}r_l)\cdot g+
(\cT_{n-l}r_l\cdot\xi)\partial_\omega r_1\Bigr]
-\sum_{i+k=n}V_i\cdot \mathbf d(\bar r_k+r_k)
-\mathop{\sum_{i+k=n+1}}_{2\leq i,k\leq n-1}
(V_i\cdot\xi)\partial_\omega r_k.
\end{equation}
Equating \eqref{a:eq:fullcoefficient} to zero, we first solve its phase average:
\begin{equation}\label{a:eq:meaneq}
\partial_t\bar r_n+\bar v_n\cdot g = \av{\mathcal F_n}
\end{equation}
Apply Lemma~\ref{a:lem:slow} to the physical source
$\av{\mathcal F_n}\circ X^{-1}$; this selects the generally nonzero
initial value $\bar r_n(0)$. We then solve
\begin{equation}\label{a:eq:osceq}
\begin{split}
&\partial_t r_n+m(\xi)\cdot g\,r_n
= \mathcal F_n-\av{\mathcal F_n}
-(\bar v_n\cdot\xi)\partial_\omega r_1,\\
&r_n(0,y,\omega)=0.
\end{split}
\end{equation}
The right side has zero phase average, so $\av{r_n}=0$ for all time.
Since $\mu_t+m(\xi)\cdot g\,\mu=0$, the oscillatory equation \eqref{a:eq:osceq} is equivalent to
\begin{equation}\label{a:eq:normalizedinverse}
\frac{r_n(t)}{\mu(t)} = \int_0^t
\mu(s)^{-1}\left[
\mathcal F_n-\av{\mathcal F_n}
-(\bar v_n\cdot\xi)\partial_\omega r_1
\right](s)\,ds.
\end{equation}

At second order, it follows from \eqref{a:eq:divhierarchy} and integration by parts in $\omega$ 
\begin{equation}\label{a:eq:quadraticmean}
\begin{aligned}
\av{\mathcal F_2}
&=-\av{(\cT_1r_1\cdot\xi)\partial_\omega r_1
 +(\cT_0r_1)\cdot\mathbf d r_1}\\
&=-\mathbf d\cdot\av{(\cT_0r_1)r_1}
=-\mathbf d\cdot\left(m(\xi)(\mu\chi)^2\av{f^2}\right).
\end{aligned}
\end{equation}
This average is generally nonzero, so the slow correction $\bar r_2$ is needed.


\begin{proposition}\label{a:prop:majorant}
For a sufficiently large $C_0$ in $\Lambda_{\mathrm G}=\bar{\mu}^{C_0}$, we have
\begin{equation}\label{a:eq:majorant}
\begin{aligned}
\norm{D^m(r_n/\mu)}_{H^8}
&\leq \Lambda_{\mathrm G}^{m+10q+140q(n-1)}\bigl((m+10q+140q(n-1))!\bigr)^2,
&& n\geq1,\\
\norm{D^m(\bar r_n,\bar v_n)}_{H^8}
&\leq \Lambda_{\mathrm G}^{m+30q+140q(n-2)}\bigl((m+30q+140q(n-2))!\bigr)^2,
&& n\geq2,
\end{aligned}
\end{equation}
for any $m\geq0$ and uniformly on $[0,S]$. Moreover, $r_n$ and the zero extension
of $r_n/\mu$ are supported in $\supp\chi$ in the label variable. Moreover,
$\supp\bar r_n(0)\subset\supp\nabla_y\bar\rho(0)$.
Under the reflection hypotheses of Theorem~\ref{thm:main}, we have
\[
r_n(t,-y,-\omega)=-r_n(t,y,\omega),\qquad 
(\bar r_n,\bar v_n)(t,-y)=-(\bar r_n,\bar v_n)(t,y).
\]
\end{proposition}
\begin{proof}
In this proof, we denote
\[
d_n=q(140n-130)\quad n\geq1;\qquad
\bar d_n=q(140n-250),\quad n\geq2.
\]
It follows from \eqref{a:eq:data2} and \eqref{a:eq:factorialproduct} that
\[
\norm{D^m(r_1/\mu)}_{H^8}
=\norm{D^m(\chi f)}_{H^8}\leq3C\bar{\mu}^{m+2}(m!)^2
\leq\Lambda_{\mathrm G}^{m+d_1}((m+d_1)!)^2
\]
for sufficiently large $C_0$. The phase projections satisfy
\[
\norm{D_y^m\av{w}}_{H^8_y}
+\norm{D^m(w-\av{w})}_{H^8_{y,\omega}}
\leq C\norm{D^mw}_{H^8_{y,\omega}},
\]
with $C$ independent of $m$.

Fix $n\geq2$ and assume the estimates for every lower order.
The identities
\[
\av{\cT_jr_l}=0,\qquad
\av{\mathbf d r_l}=\mathbf d\av{r_l}=0,\qquad
\av{\partial_\omega r_l}=0,
\]
\[
(\cT_0r_l)\cdot\xi=m(\xi)\cdot\xi\,r_l=0
\]
remove the fast interactions with $j=0$ and determine the terms surviving
phase averaging. The purely slow terms cancel in the oscillatory equation:
\[
\sum_{i+k=n}\bar v_i\cdot\mathbf d\bar r_k
-\av{\sum_{i+k=n}\bar v_i\cdot\mathbf d\bar r_k}=0.
\]
We normalize the remaining terms in \eqref{a:eq:normalizedinverse} using \eqref{a:eq:normalizedderivative} and Lemma~\ref{a:lem:localcosts}.

For a term $h$ of $\av{\mathcal F_n}$ with shift $\delta$, denote
$\mathfrak f=h\circ X^{-1}$. Since $S\leq1$, we have
\[
\int_0^S\norm{D_y^m(\mathfrak f\circ X)}_{H^8_y}\,dt
\leq S\sup_{0\leq t\leq S}\norm{D_y^mh(t)}_{H^8_y}
\leq\Lambda_{\mathrm G}^{m+\delta}((m+\delta)!)^2
\]
which implies \eqref{source-est}.

We summarize the shift after the inverse for each source term. The indices $l,k\geq1$ below are those of \eqref{a:eq:source} and the additional
coupling in \eqref{a:eq:osceq}.
\begin{center}\small
\begin{tabular}{@{}p{.40\textwidth}p{.56\textwidth}@{}}
\toprule
Source and index relation & Shift after the inverse \\
\midrule
$(\cT_jr_l)\cdot\mathbf d r_k$,  $l+k+j=n$, $j\geq0$
& $d_l+d_k+q(j+1)=\bar d_n-(9+139j)q$ \\[3pt]
$(\cT_jr_l\cdot\xi)\partial_\omega r_k$, $l+k+j=n+1$, $j\geq1$
& $d_l+d_k+q(j+1)=\bar d_n-(139j-131)q$ \\[3pt]
$\bar v_i\cdot\mathbf d\bar r_k$, $i+k=n$, $i,k\geq2$
& $\bar d_i+\bar d_k+q=\bar d_n-249q$ \\[3pt]
$(\cT_jr_{n-j})\cdot g$, $1\leq j\leq n-1$
& $d_{n-j}+q(j+1)=d_n-(139j-1)q$ \\[3pt]
$(\cT_jr_l)\cdot\mathbf d\bar r_k$, $l+k+j=n$, $k\geq2$, $j\geq0$
& $d_l+\bar d_k+q(j+1)=d_n-(249+139j)q$ \\[3pt]
$\bar v_i\cdot\mathbf d r_k$, $i+k=n$, $i\geq2$
& $\bar d_i+d_k+q=d_n-249q$ \\[3pt]
$(\bar v_i\cdot\xi)\partial_\omega r_k$, $i+k=n+1$, $2\leq i,k\leq n-1$
& $\bar d_i+d_k+q=d_n-109q$ \\[3pt]
$(\bar v_n\cdot\xi)\partial_\omega r_1$
& $\bar d_n+d_1+q=d_n-109q$ \\
\bottomrule
\end{tabular}
\end{center}
Only the first three rows contribute to the mean, with shifts at most
$\bar d_n-8q$. In the normalized oscillatory equation, the third row cancels;
since $d_n-\bar d_n=120q$, the remaining shifts are at most $d_n-109q$.

For either target $D\in\{\bar d_n,d_n\}$, each summand has shift at most $D-4$ and
\[
D-3\geq\bar d_n-3\geq140n-253\geq9(n+1),\qquad n\geq2.
\]
It then follows
\[
\begin{aligned}
\frac{\displaystyle\sup_{0\leq t\leq S}\norm{D^mw(t)}_{H^8}}
{\Lambda_{\mathrm G}^{m+D}((m+D)!)^2}
&\leq C(n+1)^2\Lambda_{\mathrm G}^{-4}
\left(\frac{(m+D-4)!}{(m+D)!}\right)^2\\
&\leq C\Lambda_{\mathrm G}^{-4}(n+1)^{-6}\leq1
\end{aligned}
\]
for sufficiently large $C_0$.
Applying this first to $w=(\bar r_n,\bar v_n)$ with $D=\bar d_n$,
and then to $w=r_n/\mu$ with $D=d_n$ and applying the established bound for
$\bar v_n$ in the last row implies \eqref{a:eq:majorant}.

For the support assertions, denote
$a=\nabla_y\bar\rho(0)$ and $K=\supp a\cup\supp\chi$.
The local operators preserve label support:
\[
\supp_y(\cT_jw),\quad \supp_y(d_iw),\quad
\supp_y(\partial_\omega^{\pm1}w)\subset\supp_yw.
\]
It follows from induction in \eqref{a:eq:source} that the purely slow
terms are supported in $K$, while the terms containing an oscillatory
coefficient are supported in $\supp\chi$. The purely slow terms cancel upon subtracting the phase
average, and hence we have 
\[
\supp_y\av{\mathcal F_n(t)}\subset K,\qquad
\supp_y(\mathcal F_n-\av{\mathcal F_n})\subset\supp\chi.
\]
Applying the support and mean properties of the inverse in Lemma~\ref{a:lem:slow}, we obtain
\[
\supp_y\bar r_n(t)\subset K,\qquad
\supp\bar r_n(0)\subset\supp a,\qquad
\int_{\T^2}\bar r_n(0)\,dy=0.
\]
Since $(\bar v_n\cdot\xi)\partial_\omega r_1$ is supported in
$\supp\chi$, \eqref{a:eq:normalizedinverse} yields
\[
\supp_y(r_n/\mu),\ \supp_y r_n\subset\supp\chi,
\]
where $r_n/\mu$ is extended by zero outside $\mathcal O$.

Under the reflection hypotheses of Theorem~\ref{thm:main}, define
$\mathscr Rw(y,\omega)=w(-y,-\omega)$.
Since $D_yX,g,\xi,\mu$ are even in $y$, we have
\[
\mathscr R d_i=-d_i\mathscr R,\qquad
\mathscr R\partial_\omega=-\partial_\omega\mathscr R,\qquad
\mathscr RL_k=L_k\mathscr R,\quad k=0,1,2.
\]
The recursions \eqref{a:eq:Qrecursion}--\eqref{a:eq:Trecursion} imply
\[
\mathscr RQ_j=-Q_j\mathscr R,\qquad
\mathscr R\cT_j=\cT_j\mathscr R,\quad j\geq0.
\]
Starting with $\mathscr Rr_1=-r_1$, induction in
\eqref{a:eq:source} yields $\mathscr R\mathcal F_n=-\mathcal F_n$.
Thus $\av{\mathcal F_n}$ is odd in $y$. Applying Lemma~\ref{a:lem:slow} and
\eqref{a:eq:normalizedinverse} with $\mu$ even, we obtain
\[
(\bar r_n,\bar v_n)(t,-y)=-(\bar r_n,\bar v_n)(t,y),\qquad
\mathscr Rr_n=-r_n.
\]
\end{proof}

\subsubsection{The residual estimate}\label{a:sec:remainder}
For the approximation \eqref{a:eq:rhoapprox}, denote
$u^{a,N}=\cT\rho^{a,N}$ and
$\mathcal R_N=\partial_t\rho^{a,N}+u^{a,N}\cdot\nabla_x\rho^{a,N}$.

\begin{proposition}\label{a:prop:joint}
For every integer $N\geq2$, every $\lambda\geq1$, and every integer $m\geq0$,
\begin{equation}\label{a:eq:joint}
\norm{D^m(\mathcal R_N\circ X)}_{H^4_{1/\lambda}}
\leq \lambda^{-N+c}\bar{\mu}^{c(N+1)}
(\lambda \bar{\mu}^c)^m((m+c(N+1))!)^2,
\end{equation}
where $c$ is a positive integer independent of $m,N,\lambda,\Gamma,\bar{\mu}$
and $S$.
\end{proposition}
\begin{proof}
For $1\leq i\leq N+1$, denote
\begin{equation}\label{a:eq:truncatedvelocitycoeff}
\widetilde V_i = \mathbf1_{i\leq N}\bar v_i+
\sum_{l=1}^{\min(i,N)}\cT_{i-l}r_l,
\end{equation}
with $\bar v_1=0$. Thus $\widetilde V_i=V_i$ for $i\leq N$. 

Applying Lemma~\ref{a:lem:globalremainder} to $r_l$ with truncation
$L=N+2-l$, we obtain
\begin{equation}\label{a:eq:split}
u^{a,N}-U =
\sum_{i=1}^{N+1}\lambda^{-i}\widetilde V_i^\lambda+\mathfrak E_N,
\qquad
\mathfrak E_N = \sum_{l=1}^N\lambda^{-l}\mathcal E_{N+2-l}[r_l].
\end{equation}
Since $l+L=N+2$ and $r_l$ has shift at most $C(l+1)$,
it follows from \eqref{a:eq:majorant} and Lemma~\ref{a:lem:globalremainder} that
\begin{equation}\label{a:eq:Ejoint}
\norm{D^m(\mathfrak E_N\circ X)}_{H^4_{1/\lambda}}
\leq \lambda^{-N-2}\bar{\mu}^{c(N+1)}
(\lambda \bar{\mu}^c)^m((m+c(N+1))!)^2.
\end{equation}

The order-one coefficient vanishes by \eqref{a:eq:phasecarrier} and
$V_1\cdot\xi=0$; the coefficients of orders $2\leq n\leq N$ vanish
by the hierarchy. The remaining terms are
\begin{equation}\label{a:eq:polyresidual}
\begin{split}
\mathscr P_N=&\lambda^{-N-1}\widetilde V_{N+1}\cdot g
+\sum_{\substack{1\leq i\leq N+1,\ 1\leq k\leq N\\i+k\geq N+1}}
\lambda^{-i-k}\widetilde V_i\cdot\mathbf d(\bar r_k+r_k)\\
&+\sum_{\substack{1\leq i\leq N+1,\ 1\leq k\leq N\\i+k-1\geq N+1}}
\lambda^{1-i-k}(\widetilde V_i\cdot\xi)\partial_\omega r_k.
\end{split}
\end{equation}
We then have 
\begin{equation}\label{a:eq:residualdecomposition}
\mathcal R_N = \mathscr P_N^\lambda+
\mathfrak E_N\cdot\nabla_x\rho^{a,N}.
\end{equation}
A term in $\widetilde V_i$ has shift at most $C(i+1)$; multiplication by a density
coefficient of order $k$ and one derivative costs at most $C(i+k+1)$.
Hence $\mathscr P_N^\lambda\circ X$ is bounded by the right side of \eqref{a:eq:joint}.

The proof of Lemma~\ref{a:lem:graph}, using $H^8\hookrightarrow L^\infty$
for coefficient derivatives, also yields
\begin{equation}\label{a:eq:graphsup}
\norm{D^m(a^\lambda\circ X)}_{W^{4,\infty}_{1/\lambda}}
\leq(\lambda\bar{\mu}^c)^m\bar{\mu}^{c(d+1)}((m+d+c)!)^2
\end{equation}
for a supported coefficient $a$ of shift $d$. Since
\[
(\nabla_x\rho^{a,N})\circ X
=g+\sum_{k=1}^N\lambda^{-k}
\bigl[\mathbf d(\bar r_k+r_k)+\lambda\xi\partial_\omega r_k\bigr]
(t,y,\lambda\phi_0(y)),
\]
it then follows
\[
\norm{D^m((\nabla_x\rho^{a,N})\circ X)}_{W^{4,\infty}_{1/\lambda}}
\leq \bar{\mu}^{c(N+1)}(\lambda \bar{\mu}^c)^m((m+c(N+1))!)^2.
\]
We infer from \eqref{a:eq:factorialproduct}
\[
\begin{aligned}
&\norm{D^m\bigl((\mathfrak E_N\cdot\nabla_x\rho^{a,N})\circ X\bigr)}_{H^4_{1/\lambda}}\\
&\quad\leq C\sum_{l=0}^m\binom ml
\norm{D^l(\mathfrak E_N\circ X)}_{H^4_{1/\lambda}}
\norm{D^{m-l}((\nabla_x\rho^{a,N})\circ X)}_{W^{4,\infty}_{1/\lambda}}\\
&\quad\leq\lambda^{-N-2}\bar{\mu}^{c(N+1)}
(\lambda\bar{\mu}^c)^m((m+c(N+1))!)^2
\end{aligned}
\]
which together with the estimate for $\mathscr P_N^\lambda\circ X$ and
\eqref{a:eq:residualdecomposition} verifies \eqref{a:eq:joint}.
\end{proof}

The factorial inequality
\begin{equation}\label{a:eq:factorialseparation}
(m+d)!\leq2^{m+d}m!d!
\end{equation}
separates the derivative order from the truncation order. Let $c_*$ be the
fixed exponent in \eqref{a:eq:joint}, and take
$\kappa=(\lambda \bar{\mu}^{C_1})^{-1}$. With $D=c_*(N+1)$, we obtain
\[
\begin{aligned}
\G_{1,\kappa}(\mathcal R_N\circ X)
&\leq\lambda^{-N+c_*}\bar{\mu}^{c_*(N+1)}4^D(D!)^2
\sum_{m\geq0}(4\bar{\mu}^{c_*-C_1})^m\\
&\leq2\lambda^{-N+c_*}\bar{\mu}^{c_*(N+1)}4^D(D!)^2,
\end{aligned}
\]
provided $C_1\geq c_*+3$. Since $D!\leq D^D\leq(CN)^{CN}$, we have
\begin{equation}\label{a:eq:untruncatedG}
\G_{1,\kappa}(\mathcal R_N\circ X)
\leq \lambda^{-N+c}(\bar{\mu}^cN^c)^{cN+c}.
\end{equation}

Denote
\begin{equation}\label{a:eq:N}
N = \lfloor\lambda^\nu\rfloor.
\end{equation}
For each integer $c\geq1$, choose $\nu>0$ small and $Q$ large so that
\[
c^2(\nu+Q^{-1})\leq\tfrac14,\qquad
\lfloor2^{Q\nu}\rfloor\geq4c+1.
\]
We impose these conditions for the finitely many fixed exponents used below.
If $\lambda\geq\bar{\mu}^Q$, then $N\geq4c+1$ and
$\tfrac12\lambda^\nu\leq N\leq\lambda^\nu$. Consequently,
\begin{equation}\label{a:eq:truncationbudget}
\begin{split}
&\log\!\left[\lambda^{-N+c}(\bar{\mu}^cN^c)^{cN+c}\right]\\
&\quad\leq[-N+c+c^2(N+1)(\nu+Q^{-1})]\log\lambda\\
&\quad\leq-\tfrac12N\log\lambda
\leq-\tfrac14\lambda^\nu\log\lambda.
\end{split}
\end{equation}
Applying \eqref{a:eq:truncationbudget} to \eqref{a:eq:untruncatedG} yields
\begin{equation}\label{a:eq:Gresidual}
\G_{1,\kappa}(\mathcal R_N\circ X)
\leq \exp(-c_1\lambda^\nu\log\lambda)
\end{equation}
for some $c_1>0$. 

With $N$ given by \eqref{a:eq:N}, denote
$(\rho^a,u^a,\mathcal R)=(\rho^{a,N},u^{a,N},\mathcal R_N)$.

\begin{lemma}\label{a:lem:approximation}
Under \eqref{a:eq:data}--\eqref{a:eq:positiveH}, there are constants
$\nu>0$ and $C,Q\geq1$ such that, for
$\lambda\geq\bar{\mu}^Q$ and $N=\lfloor\lambda^\nu\rfloor\geq2$,
the approximation \eqref{a:eq:rhoapprox} satisfies
\begin{equation}\label{a:eq:approxjetestimates}
\begin{split}
&\norm{(\nabla_x\rho^a-\nabla_x\bar\rho)\circ X-g_\lambda}_\infty\\
&\quad+\norm{(\nabla_xu^a-\nabla_xU)\circ X-M_\lambda}_\infty
\leq \bar{\mu}^C\lambda^{-1},\\
&\norm{D^m((\nabla_x\rho^a,\nabla_xu^a)\circ X)}_{L^\infty}
\leq \bar{\mu}^C(\lambda \bar{\mu}^C)^m(m!)^2
\end{split}
\end{equation}
and 
\begin{equation}\label{a:eq:relativevelocity}
\begin{split}
&(D_yX)^{-1}\bigl((u^a-U)\circ X\bigr)=V^0+R^a,\\
&V^0(t,y)=\lambda^{-1}\mu(t,y)\chi(y)(D_yX(t,y))^{-1}m(\xi(t,y))
f(\lambda\phi_0(y)),
\end{split}
\end{equation}
with
\begin{equation}\label{a:eq:Ra}
\norm{D^mR^a}_{L^\infty}
\leq\lambda^{-2}\bar{\mu}^C(\lambda \bar{\mu}^C)^m(m!)^2,
\end{equation}
where the constants are independent of $m,N,\lambda,\bar{\mu},\Gamma,S$.
\end{lemma}

\begin{proof}
Let $w_n$ denote either
$(\bar r_n+r_n)^\lambda$, for $1\leq n\leq N$, or
$\widetilde V_n^\lambda$, for $1\leq n\leq N+1$.
Using Sobolev embedding for the phase-independent terms and
\eqref{a:eq:graphsup} for the oscillatory terms, together with
Proposition~\ref{a:prop:majorant}, Lemma~\ref{a:lem:localcosts}, and
the derivative identity
$\partial_{x_i}w^\lambda=(d_iw+\lambda\xi_i\partial_\omega w)^\lambda$,
factorial separation \eqref{a:eq:factorialseparation} yields positive integers
$a,b$ such that
\[
\begin{split}
&B_n=\bar{\mu}^{b(n+1)}\bigl((a(n+1))!\bigr)^2,\\
&\norm{D_y^m(w_n\circ X)}_\infty
+\lambda^{-1}\norm{D_y^m((\nabla_xw_n)\circ X)}_\infty
\leq B_n(\lambda\bar{\mu}^b)^m(m!)^2.
\end{split}
\]
Since $N+2\leq3\lambda^\nu$ and $\bar{\mu}\leq\lambda^{1/Q}$, choosing
$2a\nu\leq1/8$ and $b/Q\leq1/8$ yields
\[
\begin{aligned}
\frac{\lambda^{-n-1}B_{n+1}}{\lambda^{-n}B_n}
&=\lambda^{-1}\bar{\mu}^b\prod_{j=1}^a(a(n+1)+j)^2\\
&\leq(3a)^{2a}\lambda^{-1+b/Q+2a\nu}
\leq\lambda^{-1/2}\leq\frac12,
\qquad 1\leq n\leq N,
\end{aligned}
\]
after increasing the absolute threshold for $\lambda$. Consequently,
\begin{equation}\label{a:eq:approximationtails}
\begin{split}
&\sum_{n=2}^{N+1}\lambda^{-n}B_n
\leq\frac{\lambda^{-2}B_2}{1-\lambda^{-1/2}}
\leq\lambda^{-2}\bar{\mu}^C,\\
&\sum_{n=2}^{N+1}\lambda^{1-n}B_n\leq\lambda^{-1}\bar{\mu}^C.
\end{split}
\end{equation}

At order one, $r_1=\mu\chi f(\omega)$ and
$\cT_0r_1=m(\xi)r_1$. Applying \eqref{a:eq:principaljets} yields
\[
\begin{aligned}
\bigl(\nabla_x[\lambda^{-1}r_1^\lambda]\bigr)\circ X
&=g_\lambda+\lambda^{-1}(\mathbf d r_1)^\lambda\circ X,\\
\bigl(\nabla_x[\lambda^{-1}(\cT_0r_1)^\lambda]\bigr)\circ X
&=M_\lambda+
\lambda^{-1}\bigl(d_j(\cT_0r_1)_i\bigr)_{i,j}^\lambda\circ X.
\end{aligned}
\]
The coefficient estimates imply
\[
\begin{aligned}
&\norm{(\mathbf d r_1)^\lambda\circ X}_\infty
+\norm{\bigl(d_j(\cT_0r_1)_i\bigr)_{i,j}^\lambda\circ X}_\infty
\leq \bar{\mu}^C,\\
&\norm{D_y^m(g,M,g_\lambda,M_\lambda)}_\infty
\leq \bar{\mu}^C(\lambda \bar{\mu}^C)^m(m!)^2.
\end{aligned}
\]

For the nonlocal remainder in \eqref{a:eq:split}, the basic estimate
\[
\sum_{k\in\mathbb Z^2}
\frac{|k|^{2\ell}}{(1+|k|^2/\lambda^2)^4}
\leq C\lambda^{2+2\ell},\qquad \ell=0,1,
\]
and Cauchy--Schwarz in Fourier variables imply
\[
\norm{\nabla_y^\ell w}_\infty
\leq C\lambda^{1+\ell}\norm w_{H^4_{1/\lambda}},
\qquad \ell=0,1.
\]
Since
\[
(\nabla_x\mathfrak E_N)\circ X
=D_y(\mathfrak E_N\circ X)(D_yX)^{-1},
\]
it follows from \eqref{a:eq:Ejoint}, \eqref{a:eq:data}, and
\eqref{a:eq:factorialproduct} that, after increasing a fixed integer $c\geq2$,
\[
\norm{D_y^m((\mathfrak E_N,\nabla_x\mathfrak E_N)\circ X)}_\infty
\leq\lambda^{-N}\bar{\mu}^{c(N+1)}
(\lambda\bar{\mu}^c)^m((m+c(N+1))!)^2,
\qquad m\geq0.
\]
Factorial separation \eqref{a:eq:factorialseparation} yields
\[
\begin{split}
&\norm{D_y^m((\mathfrak E_N,\nabla_x\mathfrak E_N)\circ X)}_\infty
\leq H_N(\lambda\bar{\mu}^C)^m(m!)^2,\\
&H_N=\lambda^{-N+C}(\bar{\mu}^CN^C)^{CN+C}.
\end{split}
\]
For the common choice of $\nu,Q$ and a sufficiently large fixed frequency
threshold, \eqref{a:eq:truncationbudget} yields
\[
H_N\leq\exp\!\left(-\tfrac14\lambda^\nu\log\lambda\right)
\leq\lambda^{-2}.
\]

It follows from \eqref{a:eq:rhoapprox} and \eqref{a:eq:split} that
\[
\begin{aligned}
\rho^a-\bar\rho
&=\lambda^{-1}r_1^\lambda+
\sum_{n=2}^N\lambda^{-n}(\bar r_n+r_n)^\lambda,\\
u^a-U
&=\lambda^{-1}(\cT_0r_1)^\lambda+
\sum_{n=2}^{N+1}\lambda^{-n}\widetilde V_n^\lambda+\mathfrak E_N.
\end{aligned}
\]
Combining these decompositions with \eqref{a:eq:approximationtails}, the
order-one gradient identities, and the bound for $H_N$ proves
\eqref{a:eq:approxjetestimates}.
We also obtain \eqref{a:eq:relativevelocity} with
\[
R^a=(D_yX)^{-1}\left[
\sum_{n=2}^{N+1}\lambda^{-n}(\widetilde V_n^\lambda\circ X)
+\mathfrak E_N\circ X\right].
\]
By \eqref{a:eq:data}, \eqref{a:eq:factorialproduct},
\eqref{a:eq:approximationtails}, and $H_N\leq\lambda^{-2}$,
\[
\norm{D_y^mR^a}_\infty
\leq\lambda^{-2}\bar{\mu}^C(\lambda\bar{\mu}^C)^m(m!)^2,
\]
which verifies \eqref{a:eq:Ra}.
\end{proof}

\begin{lemma}[Initial phase-independent correction]\label{a:lem:initialslow}
Under the hypotheses of Lemma~\ref{a:lem:approximation}, the 
constants $\nu,C,Q$ can be chosen so that, for
$\lambda\geq\bar{\mu}^Q$ and $N=\lfloor\lambda^\nu\rfloor\geq2$,
the correction
\[
b_\lambda=\sum_{n=2}^N\lambda^{-n}\bar r_n(0)
\]
satisfies
\[
\supp b_\lambda\subset\supp\nabla_y\bar\rho(0)
\]
and
\begin{equation}\label{a:eq:initialslowstrong}
\norm{D^m b_\lambda}_{H^8}
\leq\lambda^{-2}\bar{\mu}^C(\bar{\mu}^C)^m(m!)^2, \quad m\in\mathbb N_0.
\end{equation}
More precisely, at the common spatial radius
$(8\Lambda_{\mathrm G})^{-1}$, we have
\[
\sum_{m\geq0}\frac{(8\Lambda_{\mathrm G})^{-m}}{(m!)^2}
\norm{D^m b_\lambda}_{H^8}\leq\lambda^{-2}\bar{\mu}^C
\]
and consequently
\[
\norm{b_\lambda}_{H^k}\leq C_k\bar{\mu}^{C_k}\lambda^{-2}, \quad k\in\mathbb N_0.
\]
\end{lemma}

\begin{proof}
It follows from Proposition~\ref{a:prop:majorant} that
$\supp b_\lambda\subset\supp\nabla_y\bar\rho(0)$.
Denote $\bar d_n=30q+140q(n-2)$. By \eqref{a:eq:majorant} and
\eqref{a:eq:factorialseparation},
\[
\norm{D^m\bar r_n(0)}_{H^8}
\leq(4\Lambda_{\mathrm G})^m(m!)^2
(4\Lambda_{\mathrm G})^{\bar d_n}(\bar d_n!)^2,
\qquad m\geq0,\ n\geq2.
\]
Using $\bar\mu\leq\lambda^{1/Q}$ and $N\leq\lambda^\nu$, for
$2\leq n<N$ we obtain
\[
\begin{aligned}
\frac{\lambda^{-n-1}(4\Lambda_{\mathrm G})^{\bar d_{n+1}}
(\bar d_{n+1}!)^2}
{\lambda^{-n}(4\Lambda_{\mathrm G})^{\bar d_n}(\bar d_n!)^2}
&=\lambda^{-1}(4\Lambda_{\mathrm G})^{140q}
\prod_{j=1}^{140q}(\bar d_n+j)^2\\
&\leq C_q\lambda^{-1+140C_0q/Q+280q\nu}
\leq\lambda^{-1/2}\leq\tfrac12,
\end{aligned}
\]
provided
\[
\nu\leq\frac1{1120q},\qquad
\frac{140C_0q}{Q}\leq\frac18,\qquad
\lambda\geq\max(4,C_q^8).
\]
These conditions are compatible with the preceding choices; the fixed
frequency threshold can be absorbed into $Q$ since $\bar\mu\geq2$.
Since $\bar d_2=30q$, the geometric sum yields
\[
\begin{aligned}
\norm{D^m b_\lambda}_{H^8}
&\leq(4\Lambda_{\mathrm G})^m(m!)^2
\sum_{n=2}^N\lambda^{-n}(4\Lambda_{\mathrm G})^{\bar d_n}
(\bar d_n!)^2\\
&\leq2\lambda^{-2}(4\Lambda_{\mathrm G})^{30q}((30q)!)^2
(4\Lambda_{\mathrm G})^m(m!)^2\\
&\leq\lambda^{-2}\bar\mu^C(\bar\mu^C)^m(m!)^2.
\end{aligned}
\]
This proves \eqref{a:eq:initialslowstrong}. Multiplying the second bound by
$(8\Lambda_{\mathrm G})^{-m}/(m!)^2$ and summing, with
$\sum_{m\geq0}2^{-m}=2$, yields
\[
\sum_{m\geq0}\frac{(8\Lambda_{\mathrm G})^{-m}}{(m!)^2}
\norm{D^m b_\lambda}_{H^8}
\leq4\lambda^{-2}(4\Lambda_{\mathrm G})^{30q}((30q)!)^2
\leq\lambda^{-2}\bar\mu^C.
\]
For every $k\in\mathbb N_0$, \eqref{a:eq:initialslowstrong} also yields
$\norm{b_\lambda}_{H^k}\leq C_k\bar\mu^{C_k}\lambda^{-2}$,
with the same truncation and frequency threshold for all $k$.\qedhere
\end{proof}

\subsection{The new solution and its particle map}\label{sec:realization}
We complete the proof of Theorem~\ref{thm:main} by removing the residual
and estimating the correction and its particle map in Gevrey norms.

\subsubsection{The approximate particle map}\label{a:sec:maps}
Let $X^a$ be the flow of $u^a$ and denote $Y^a=X^{-1}\circ X^a$, so that
\[
\partial_tX^a=u^a\circ X^a,\qquad X^a(0)=\id.
\]
It follows from \eqref{a:eq:relativevelocity} that
\[
\partial_tY^a=(V^0+R^a)\circ Y^a,\qquad Y^a(0)=\id.
\]

\begin{lemma}\label{a:lem:smallflow}
Let $d\in\{2,3\}$, $0<S\leq1$, and
$V\in C([0,S];C^\infty(\mathbb T^d;\mathbb R^d))$. Suppose that, for some
$C_0\geq0$ and $\Lambda\geq1$,
\[
\norm{D^nV(t)}_\infty\leq C_0\Lambda^n(n!)^2,
\quad n\geq0,\ 0\leq t\leq S,\qquad C_0\Lambda\leq c_d,
\]
where $c_d>0$ is sufficiently small and depends only on $d$.
The flow $\Phi$ of $V$ starting from the identity satisfies
\begin{equation}\label{a:eq:smallflow}
\sup_{0\leq t\leq S}\left(
\norm{D^n(\Phi-\id)}_\infty+
\norm{D^n(\Phi^{-1}-\id)}_\infty\right)
\leq CC_0(C\Lambda)^n(n!)^2,\qquad n\geq0,
\end{equation}
where $C$ depends only on $d$.
\end{lemma}
\begin{proof}
Use the Gevrey norm of \cite[Lemma~B.11]{BF} on the $2\pi$-periodic torus,
with $\alpha=2$, $s=(2\Lambda)^{-1/2}$ and $\sigma=s/2$. Then
\[
\sup_{0\leq t\leq S}|V(t)|_{2,s}
\leq CC_0\sup_{n\geq0}(n+1)^2 2^{-n}\leq CC_0.
\]
Since $\sigma^2=(8\Lambda)^{-1}$, the condition $C_0\Lambda\leq c_d$
implies the smallness hypothesis of that lemma. Its Picard proof applies to
\[
\Phi(t)=\id+\int_0^tV(\tau)\circ\Phi(\tau)\,d\tau:
\]
the spatial estimates are uniform in time, and the geometric derivative
tail above implies continuity in this Gevrey norm. Hence
$\sup_{t\leq S}|\Phi(t)-\id|_{2,s/2}\leq CC_0$.
Since $(s/2)^{-2}=8\Lambda$, this yields the direct estimate in
\eqref{a:eq:smallflow}. For each $t_*\leq S$, apply the same argument to
$-V(t_*-\tau,\cdot)$ on $[0,t_*]$; its terminal map is
$\Phi(t_*,\cdot)^{-1}$.
\end{proof}

\begin{lemma}\label{a:lem:approxmap}
Under the hypotheses of Lemma~\ref{a:lem:approximation}, there are constants
$C,Q,L\geq1$ and $c>0$ such that, for $\lambda\geq\bar{\mu}^Q$ and
$n\in\mathbb N_0$,
\begin{equation}\label{a:eq:approxmapestimates}
\begin{split}
&\sup_{0\leq t\leq S}\left(
\norm{X^a-X}_{L^\infty_y}+
\norm{(X^a)^{-1}-X^{-1}}_{L^\infty_x}\right)
\leq \bar{\mu}^C\lambda^{-1},\\
&\sup_{0\leq t\leq S}\left(
\norm{D^n(D_yX^a)}_{L^\infty_y}+\norm{D^n((D_yX^a)^{-1})}_{L^\infty_y}\right)
\leq \bar{\mu}^C(\lambda \bar{\mu}^C)^n(n!)^2,
\end{split}
\end{equation}
where the second estimate also holds for
$\sup_{0\leq t\leq S}\norm{D_x^n(D_x(X^a)^{-1})}_{L^\infty_x}$.
Moreover,
\begin{equation}\label{a:eq:detapprox}
\det(D_yX^a)=1, \quad t\in[0,S],
\end{equation}
\begin{equation}\label{a:eq:approxflowcoeff}
\sup_{0\leq t\leq S}\norm{D^n((\nabla_x\rho^a)\circ X^a)}_{L^\infty_y}
\leq \bar{\mu}^C(\lambda \bar{\mu}^C)^n(n!)^2.
\end{equation}
For $\kappa_1=\lambda^{-1}\bar{\mu}^{-L}$ and $0<\kappa\leq\kappa_1$,
\[
\sup_{0\leq t\leq S}\G_{1,\kappa}(\mathcal R\circ X^a)
\leq\exp(-c\lambda^\nu\log\lambda).
\]
The constants are independent of $n,N,\lambda,\bar{\mu},\Gamma,S$.
\end{lemma}
\begin{proof}
Denote $b=\mu\chi(D_yX)^{-1}m(\xi)$ on $\mathcal O$ and extend it by zero.
It follows from \eqref{a:eq:data}--\eqref{a:eq:carriersize},
Lemma~\ref{a:lem:calculus}(a), and \eqref{a:eq:murelative} that
\[
\supp_y b\subset\supp\chi\Subset\mathcal O,\qquad
\norm{D_y^n b}_\infty\leq\bar\mu^C(\bar\mu^C)^n(n!)^2.
\]
Moreover,
\begin{equation}\label{a:eq:tangency}
b\cdot\nabla_y\phi_0=\mu\chi\,m(\xi)\cdot\xi=0.
\end{equation}
For an independent $\omega\in\R/(2\pi\mathbb Z)$, consider the lifted flow
\begin{equation}\label{a:eq:liftedflow}
\begin{aligned}
\partial_t\Psi(t,y,\omega)
&=\lambda^{-1}b(t,\Psi(t,y,\omega))f(\omega),\\
\Psi(0,y,\omega)&=y.
\end{aligned}
\end{equation}
The field $(\lambda^{-1}bf,0)$ satisfies Lemma~\ref{a:lem:smallflow}
with $C_0=\lambda^{-1}\bar\mu^C$ and $\Lambda=\bar\mu^C$;
increasing $Q$ ensures $C_0\Lambda=\lambda^{-1}\bar\mu^{2C}\leq c_3$.
Denoting by $\Psi^-$ the inverse in $y$ at fixed $\omega$, we obtain
\[
\norm{D_{y,\omega}^n(\Psi-y)}_\infty+
\norm{D_{y,\omega}^n(\Psi^--y)}_\infty
\leq\lambda^{-1}\bar\mu^C(\bar\mu^C)^n(n!)^2,\qquad n\geq0.
\]
Both displacements vanish outside $\supp\chi$ by uniqueness, so the maps preserve $\mathcal O$.
It follows from \eqref{a:eq:tangency} that, for $y\in\mathcal O$,
\[
\phi_0(\Psi(t,y,\omega))
=\phi_0(\Psi^-(t,y,\omega))=\phi_0(y).
\]
Thus the flow of $V^0$ and its inverse are
\begin{equation}\label{a:eq:Lgraph}
\begin{split}
Y^0(t,y)&=\Psi(t,y,\lambda\phi_0(y)),\\
(Y^0)^{-1}(t,y)&=\Psi^-(t,y,\lambda\phi_0(y)),
\end{split}
\end{equation}
on $\mathcal O$, with identity extension outside.
Applying Lemma~\ref{a:lem:calculus}(c) to the displacements yields
\begin{equation}\label{a:eq:Lestimates}
\begin{split}
&\norm{Y^0-\id}_\infty+\norm{(Y^0)^{-1}-\id}_\infty
\leq \bar{\mu}^C\lambda^{-1},\\
&\norm{D_y^kY^0}_\infty+\norm{D_y^k(Y^0)^{-1}}_\infty
\leq \bar{\mu}^C(\lambda \bar{\mu}^C)^{k-1}(k!)^2, \quad k\geq1.
\end{split}
\end{equation}
We factor the full flow as
\begin{equation}\label{a:eq:approxflowfactor}
\begin{split}
&Y^a(t,y)=Y^0(t,Z^a(t,y)),\qquad Z^a(0)=\id,\\
&\partial_tZ^a=
\bigl[(D_yY^0)^{-1}(R^a\circ Y^0)\bigr]\circ Z^a.
\end{split}
\end{equation}
Differentiating $Y^0\circ Z^a$ verifies the equation for $Y^a$.
Using $(D_yY^0)^{-1}=D_y((Y^0)^{-1})\circ Y^0$,
\eqref{a:eq:Lestimates}, and Lemma~\ref{a:lem:calculus}(a),(c), we obtain
\[
\norm{D_y^n(D_yY^0)^{-1}}_\infty
\leq\bar\mu^C(\lambda\bar\mu^C)^n(n!)^2.
\]
Together with \eqref{a:eq:Ra}, the same product and composition estimates yield
\[
\norm{D_y^n[(D_yY^0)^{-1}(R^a\circ Y^0)]}_\infty
\leq\lambda^{-2}\bar\mu^C(\lambda\bar\mu^C)^n(n!)^2.
\]
Lemma~\ref{a:lem:smallflow}, now with
$C_0=\lambda^{-2}\bar\mu^C$ and $\Lambda=\lambda\bar\mu^C$,
applies after increasing $Q$ so that $C_0\Lambda\leq c_2$, and yields
\[
\norm{D_y^n(Z^a-\id)}_\infty+
\norm{D_y^n((Z^a)^{-1}-\id)}_\infty
\leq\lambda^{-2}\bar\mu^C(\lambda\bar\mu^C)^n(n!)^2.
\]
Since $Y^a=Y^0\circ Z^a$ and
$(Y^a)^{-1}=(Z^a)^{-1}\circ(Y^0)^{-1}$, the positive-order composition
argument preserves the frequency factor:
\[
\lambda^{l-1}\prod_{i=1}^l\lambda^{k_i-1}
=\lambda^{k-1},\qquad
k_i\geq1,\quad\sum_{i=1}^l k_i=k.
\]
Together with the displacement estimates, this proves
\eqref{a:eq:Lestimates} for $Y^a$ and $(Y^a)^{-1}$.
The physical maps satisfy
\[
X^a=X\circ Y^a,\qquad
(X^a)^{-1}=(Y^a)^{-1}\circ X^{-1}.
\]
Hence
\[
\begin{aligned}
\norm{X^a-X}_{L^\infty_y}+\norm{(X^a)^{-1}-X^{-1}}_{L^\infty_x}
&\leq\norm{D_yX}_\infty\norm{Y^a-\id}_\infty
+\norm{(Y^a)^{-1}-\id}_\infty\\
&\leq\bar\mu^C\lambda^{-1}.
\end{aligned}
\]
Applying the same positive-order composition argument,
\eqref{a:eq:Lestimates} for $Y^a$ and its inverse, and the background-map
estimates from \eqref{a:eq:data} and the proof of Lemma~\ref{a:lem:graph}, we infer
\[
\norm{D_y^kX^a}_\infty+\norm{D_x^k(X^a)^{-1}}_\infty
\leq\bar\mu^C(\lambda\bar\mu^C)^{k-1}(k!)^2,\qquad k\geq1.
\]
Since $\diver_xu^a=0$, we have $\det D_yX^a=1$, proving
\eqref{a:eq:detapprox}, and consequently $\det D_yY^a=1$.
In two dimensions,
\[
(D_yX^a)^{-1}=\cof(D_yX^a)^T,\qquad
\norm{D_y^n((D_yX^a)^{-1})}_\infty
\leq C\norm{D_y^{n+1}X^a}_\infty.
\]
Taking $k=n+1$ and absorbing $(n+1)^2\leq4^{n+1}$ proves
\eqref{a:eq:approxmapestimates} and the physical inverse-map estimate.

The identity
$(\nabla_x\rho^a)\circ X^a=((\nabla_x\rho^a)\circ X)\circ Y^a$,
\eqref{a:eq:approxjetestimates}, and Lemma~\ref{a:lem:calculus}(c) prove
\eqref{a:eq:approxflowcoeff}.
For the residual, apply the same composition estimate to
$\mathcal R\circ X^a=(\mathcal R\circ X)\circ Y^a$.
By \eqref{a:eq:Gresidual}, \eqref{a:eq:Lestimates}, and $\det D_yY^a=1$,
there is a fixed sufficiently large $B\geq C_1$ such that
\[
\norm{D_y^k(\mathcal R\circ X^a)}_{L^2_y}
\leq\exp(-c_1\lambda^\nu\log\lambda)
(\lambda\bar\mu^B)^k(k!)^2,\qquad k\geq0.
\]
Choose $L=B+3$ and $\kappa_1=\lambda^{-1}\bar\mu^{-L}$, so
$\kappa_1\lambda\bar\mu^B\leq1/8$. It follows that
\[
\begin{aligned}
\G_{1,\kappa_1}(\mathcal R\circ X^a)
&\leq C\exp(-c_1\lambda^\nu\log\lambda)
\sum_{m\geq0}\sum_{j=0}^4
8^{-(m+j)}\left(\frac{(m+j)!}{m!}\right)^2\\
&\leq C\exp(-c_1\lambda^\nu\log\lambda)
\leq\exp\!\left(-\frac{c_1}{2}\lambda^\nu\log\lambda\right)
\end{aligned}
\]
after increasing $Q$.
Monotonicity in $\kappa$ yields the estimate for $0<\kappa\leq\kappa_1$.
\end{proof}

\subsubsection{The remaining residual correction}\label{a:sec:finite}
\begin{lemma}\label{a:lem:finite}
Let $(\rho^a,u^a)$ be the approximation from Lemma~\ref{a:lem:approximation}.
There is a fixed $Q\geq1$ such that, for $\lambda\geq\bar\mu^Q$, the unique
smooth solution $\rho^+$ of \eqref{eq:ipm} with $\rho^+(0)=\rho^a(0)$ has
maximal smooth lifespan $T_+>S$. The correction $e=\rho^+-\rho^a$ satisfies
\begin{equation}\label{a:eq:finitecorrection}
\sup_{0\leq t\leq S}\norm e_{H^4_{1/\lambda,x}}
\leq\exp(-\lambda^{\nu/2}).
\end{equation}
\end{lemma}
\begin{proof}
Lemma~\ref{lem:local} yields a unique maximal smooth solution $\rho^+$ with
$\rho^+(0)=\rho^a(0)$. On $0\leq t<\min(S,T_+)$, denote
$v=\cT e=u^+-u^a$. Subtracting the approximate equation from \eqref{eq:ipm} yields
\begin{equation}\label{a:eq:errorphysical}
e_t+(u^a+\cT e)\cdot\nabla_x e+(\cT e)\cdot\nabla_x\rho^a
= -\mathcal R,\qquad e(0)=0.
\end{equation}
Denote $E(t)=\norm{e(t)}_{H^4_{1/\lambda,x}}$ as in \eqref{a:eq:scalednorm}.
Since $|m(\ell)|\leq1$, we have $\norm v_{H^4_{1/\lambda,x}}\leq E$.
Applying \eqref{a:eq:d} with the coefficient estimates \eqref{a:eq:data} to
\eqref{a:eq:approxjetestimates} yields
\begin{equation}\label{a:eq:finitecoeff}
\sum_{l=1}^5\lambda^{1-l}\norm{\nabla_x^l\rho^a}_\infty+
\sum_{l=1}^4\lambda^{1-l}\norm{\nabla_x^lu^a}_\infty\leq \bar{\mu}^C.
\end{equation}

\medskip\noindent\emph{The weighted energy estimate.}
For $|I|=k\leq4$,  \eqref{a:eq:finitecoeff} and
$\norm v_{H^4_{1/\lambda,x}}\leq E$ imply
\[
\lambda^{-k}\left(
\norm{[\partial_x^I,u^a\cdot\nabla_x]e}_2
+\norm{\partial_x^I(v\cdot\nabla_x\rho^a)}_2\right)
\leq C\bar\mu^CE.
\]
Fourier Cauchy--Schwarz and the lattice estimate in the proof of
Lemma~\ref{a:lem:approximation} yield
\[
\sum_{\ell\in\mathbb Z^2}|\ell|\,|\widehat e(\ell)|
\leq CE\left(\sum_{\ell\in\mathbb Z^2}
\frac{|\ell|^2}{(1+|\ell|^2/\lambda^2)^4}\right)^{1/2}
\leq C\lambda^2E.
\]
It follows from Young's convolution inequality and $|m(\ell)|\leq1$ that
\[
\begin{aligned}
\lambda^{-k}\norm{[\partial_x^I,v\cdot\nabla_x]e}_2
&\leq C\lambda^{-k}\left(\sum_{\ell\in\mathbb Z^2}
|\ell|\,|\widehat e(\ell)|\right)\norm{D_x^ke}_2\\
&\leq C\lambda^2E^2,\qquad 1\leq k\leq4.
\end{aligned}
\]
The commutator vanishes for $k=0$.
Applying $\lambda^{-|I|}\partial_x^I$ to \eqref{a:eq:errorphysical},
pairing with $\lambda^{-|I|}\partial_x^Ie$, and summing yields, since
$\diver_x(u^a+v)=0$,
\[
\begin{aligned}
\frac12\frac{d}{dt}E^2
&=-\sum_{|I|\leq4}\lambda^{-2|I|}
\bigl\langle[\partial_x^I,(u^a+v)\cdot\nabla_x]e,
\partial_x^Ie\bigr\rangle\\
&\quad-\sum_{|I|\leq4}\lambda^{-2|I|}
\bigl\langle\partial_x^I(v\cdot\nabla_x\rho^a),
\partial_x^Ie\bigr\rangle
-\sum_{|I|\leq4}\lambda^{-2|I|}
\bigl\langle\partial_x^I\mathcal R,\partial_x^Ie\bigr\rangle\\
&\leq C\bar\mu^CE^2+C\lambda^2E^3+
\norm{\mathcal R}_{H^4_{1/\lambda,x}}E,
\end{aligned}
\]
which implies
\begin{equation}\label{a:eq:finiteenergy}
\frac{d}{dt}\norm e_{H^4_{1/\lambda,x}}
\leq C\bar{\mu}^C\norm e_{H^4_{1/\lambda,x}}
+C\lambda^2\norm e_{H^4_{1/\lambda,x}}^2
+\norm{\mathcal R}_{H^4_{1/\lambda,x}}
\end{equation}
where the constants are independent of $\lambda$, $S$, $N$ and $\Gamma$.

\medskip\noindent\emph{The residual in physical coordinates.}
It follows from \eqref{a:eq:d}, \eqref{a:eq:data}, $\det D_yX=1$, and
\eqref{a:eq:Gresidual} that
\[
\norm{\mathcal R}_{H^4_{1/\lambda,x}}
\leq\bar{\mu}^C\norm{\mathcal R\circ X}_{H^4_{1/\lambda,y}}
\leq\bar{\mu}^{C+4C_1}
\G_{1,(\lambda\bar{\mu}^{C_1})^{-1}}(\mathcal R\circ X)
\leq\bar{\mu}^{C+4C_1}e^{-c_1\lambda^\nu\log\lambda}.
\]

\medskip\noindent\emph{Bootstrap and continuation.}
Choose a fixed $b>0$ so that the linear coefficient in
\eqref{a:eq:finiteenergy} and the residual prefactor above are bounded
by $C\bar{\mu}^b$ and $\bar{\mu}^b$, respectively. On the initial
bootstrap interval $C\lambda^2E\leq1$, which is nonempty since $E(0)=0$,
Gronwall's inequality yields
\[
E(t)\leq S\exp\!\left(b\log\bar{\mu}+C\bar{\mu}^b+1
-c_1\lambda^\nu\log\lambda\right).
\]
Choose $Q\nu>b$. Since $\bar{\mu}^b\leq\lambda^{b/Q}$ and
$\log\bar{\mu}\leq Q^{-1}\log\lambda$, increasing $Q$ yields
\[
\begin{aligned}
&b\log\bar{\mu}+C\bar{\mu}^b+1
\leq\tfrac12c_1\lambda^\nu\log\lambda,\\
&E(t)\leq S e^{-(c_1/2)\lambda^\nu\log\lambda}
<\min\{(2C\lambda^2)^{-1},e^{-\lambda^{\nu/2}}\}.
\end{aligned}
\]
All fixed frequency thresholds for $\lambda$ are absorbed by increasing $Q$,
since $\lambda\geq\bar{\mu}^Q\geq2^Q$. Thus continuity closes the
bootstrap on $[0,\min(S,T_+))$.

If $T_+\leq S$, then, for each fixed $\lambda$,
\[
\sup_{t<T_+}\norm{\rho^+(t)}_{H^3}
\leq\sup_{0\leq t\leq S}\norm{\rho^a(t)}_{H^3}
+C\lambda^3e^{-\lambda^{\nu/2}}<\infty.
\]
The existence time in Lemma~\ref{lem:local} depends only on the $H^3$
norm, so restarting the solution before $T_+$ contradicts maximality.
Hence $T_+>S$, and continuity at $S$ yields
\eqref{a:eq:finitecorrection}.
\end{proof}

\subsubsection{The pressure estimate in approximate particle coordinates $[0,S]$}\label{a:sec:sameradius}
Fix $t\in[0,S]$ throughout this subsection.
For a smooth scalar $\widetilde e$ on the label torus, denote
$e=\widetilde e\circ(X^a)^{-1}$ and
\begin{equation}\label{a:eq:pulledoperator}
\cT^a\widetilde e=(\cT e)\circ X^a,
\qquad V^e=(D_yX^a)^{-1}\cT^a\widetilde e.
\end{equation}
Let $\pi_e$ solve $-\Delta_x\pi_e=\partial_{x_2}e$ with zero mean, and denote
$\varpi=\pi_e\circ X^a$ and
\begin{equation}\label{a:eq:ellipticcoeff}
\mathsf A=(D_yX^a)^{-1}(D_yX^a)^{-T},\qquad
\mathsf b=(D_yX^a)^{-1}e_2.
\end{equation}
The Piola identity, $\det D_yX^a=1$, and
\eqref{a:eq:approxmapestimates} yield
\begin{equation}\label{a:eq:elliptic}
\begin{aligned}
-\diver_y(\mathsf A\nabla_y\varpi)
&=\diver_y(\mathsf b\widetilde e),\\
\cT^a\widetilde e
&=-(D_yX^a)^{-T}\nabla_y\varpi-\widetilde e\,e_2,
\end{aligned}
\end{equation}
and
\[
\begin{split}
&\bar\mu^{-C}I\leq\mathsf A\leq\bar\mu^CI,\qquad
\diver_y\mathsf b=0,\qquad \int_{\T^2}\varpi\,dy=0,\\
&V^e=-\mathsf A\nabla_y\varpi-\mathsf b\widetilde e,
\qquad \diver_yV^e=0.
\end{split}
\]
For $\widetilde e=c$ constant, $\varpi=0$, $\cT^ac=-ce_2$, and
$V^e=-c\mathsf b$; no mean-zero condition on the input is required.

For $L_\kappa>0$, take
\begin{equation}\label{a:eq:kappa}
\kappa = \lambda^{-1}\bar{\mu}^{-L_\kappa},\qquad D_\kappa=\kappa D_y.
\end{equation}
For $n_*\in\mathbb N_0$, define
\[
\G^{[n_*]}_{\gamma,\kappa}(w)
=\sum_{n=0}^{n_*}\frac{(\gamma\kappa)^n}{(n!)^2}
\norm{D^n w}_{H^4_\kappa}.
\]
\begin{proposition}\label{a:prop:elliptic}
For sufficiently large $L_\kappa$ and any $1/2\leq \gamma\leq1$, we have
\begin{equation}\label{a:eq:sameradius}
\G_{\gamma,\kappa}(\cT^a\widetilde e)+\G_{\gamma,\kappa}(V^e)
\leq \bar{\mu}^C\G_{\gamma,\kappa}(\widetilde e)
\end{equation}
The same estimate holds with $\G_{\gamma,\kappa}$ replaced by
$\G^{[n_*]}_{\gamma,\kappa}$ for every $n_*\in\mathbb N_0$.
The constant $C$ is independent of $\gamma$, $\lambda$, $n_*$, and $t$.
\end{proposition}
\begin{proof}
It follows from \eqref{a:eq:approxmapestimates} and Leibniz's rule that,
for any fixed $L_0>0$, sufficiently large $L_\kappa$ yields
\begin{equation}\label{a:eq:smallcoeffderivs}
\begin{split}
&\norm w_\infty\leq\bar\mu^C,\qquad
\sum_{1\leq|\beta|\leq4}\norm{D_\kappa^\beta w}_\infty
\leq\bar\mu^{C-L_0},\\
&\kappa^l\norm{D^lw}_{W^{4,\infty}_\kappa}
\leq\bar\mu^{C-L_0l}(l!)^2,\qquad l\geq1,
\end{split}
\end{equation}
for $w\in\{\mathsf A,\mathsf b,(D_yX^a)^{-T},(D_yX^a)^{-1}\}$,
with $C$ independent of $L_0$ and $L_\kappa$.
For
\[
-\diver_y(\mathsf A\nabla_y\psi)=\diver_y\mathfrak f,
\qquad \int_{\T^2}\psi\,dy=0,
\]
test the equations differentiated by $D_\kappa^\beta$, $|\beta|\leq4$,
against $D_\kappa^\beta\psi$. Ellipticity and
\eqref{a:eq:smallcoeffderivs} imply
\[
\norm{\nabla_y\psi}_{H^4_\kappa}
\leq\bar\mu^C\norm{\mathfrak f}_{H^4_\kappa}
+\bar\mu^{C-L_0}\norm{\nabla_y\psi}_{H^4_\kappa}.
\]
Absorbing the last term for sufficiently large $L_0$ proves
\begin{equation}\label{est-source}
\norm{\nabla_y\psi}_{H^4_\kappa}
\leq\bar\mu^C\norm{\mathfrak f}_{H^4_\kappa}.
\end{equation}

It follows from \eqref{a:eq:elliptic} that
\begin{equation}\label{a:eq:ellipticdiff}
-\diver_y(\mathsf A\nabla_yD_\kappa^I\varpi)
=\diver_y(\mathsf bD_\kappa^I\widetilde e)
+\sum_{0<J\leq I}\binom IJ\diver_y\Bigl[
(D_\kappa^J\mathsf b)D_\kappa^{I-J}\widetilde e
+(D_\kappa^J\mathsf A)\nabla_yD_\kappa^{I-J}\varpi\Bigr].
\end{equation}
Using \eqref{est-source}, \eqref{a:eq:smallcoeffderivs}, the scaled
$H^4$ product inequality, and $\binom IJ\leq\binom nl$ for
$|I|=n$, $|J|=l$, we obtain
\begin{equation}\label{a:eq:ellipticrecurrence}
\begin{aligned}
\kappa^n\norm{D^n(\nabla_y\varpi)}_{H^4_{\kappa}}
\leq{}&\bar{\mu}^C\kappa^n\norm{D^n \widetilde e}_{H^4_{\kappa}}\\
&+\bar{\mu}^C\sum_{l=1}^n\binom nl\bar{\mu}^{-L_0l}(l!)^2\kappa^{n-l}
\left(\norm{D^{n-l}(\nabla_y\varpi)}_{H^4_{\kappa}}
+\norm{D^{n-l} \widetilde e}_{H^4_{\kappa}}\right)
\end{aligned}
\end{equation}
for $n\geq0$. Together with
\[
\frac{\binom nl(l!)^2((n-l)!)^2}{(n!)^2}
=\binom nl^{-1}\leq1,
\]
we obtain
\[
\G^{[n_*]}_{\gamma,\kappa}(\nabla_y\varpi)
\leq\bar\mu^C\G^{[n_*]}_{\gamma,\kappa}(\widetilde e)
+\bar\mu^C\sum_{l\geq1}(\gamma\bar\mu^{-L_0})^l
\left[\G^{[n_*]}_{\gamma,\kappa}(\nabla_y\varpi)
+\G^{[n_*]}_{\gamma,\kappa}(\widetilde e)\right].
\]
Choose $L_0$ to satisfy the preceding absorption and
\[
\bar\mu^C\sum_{l\geq1}\bar\mu^{-L_0l}\leq\tfrac14
\qquad(\bar\mu\geq2),
\]
then fix $L_\kappa$ so that \eqref{a:eq:smallcoeffderivs} holds.
Absorption yields
\begin{equation}\label{a:eq:pressureG}
\G^{[n_*]}_{\gamma,\kappa}(\nabla_y\varpi)
\leq\bar\mu^C\G^{[n_*]}_{\gamma,\kappa}(\widetilde e).
\end{equation}
It follows from Leibniz's rule and the same factorial sum that for each coefficient
$w$ in \eqref{a:eq:smallcoeffderivs} and smooth $z$,
\[
\G^{[n_*]}_{\gamma,\kappa}(wz)
\leq\bar\mu^C\G^{[n_*]}_{\gamma,\kappa}(z).
\]
Apply this to $(D_yX^a)^{-T}\nabla_y\varpi$ and
$(D_yX^a)^{-1}\cT^a\widetilde e$. Equations
\eqref{a:eq:elliptic} and \eqref{a:eq:pressureG} imply the asserted
operator estimate for every $n_*$. Letting $n_*\to\infty$ proves
\eqref{a:eq:sameradius}.
\end{proof}

After increasing $L_\kappa$ if necessary, we also have
\begin{equation}\label{a:eq:inputsG}
\A_{1,\kappa}((\nabla_x\rho^a)\circ X^a)\leq \bar{\mu}^C,\qquad
\G_{1,\kappa}(\mathcal R\circ X^a)
\leq\exp(-c_1\lambda^\nu\log\lambda).
\end{equation}
For $\kappa\lambda\bar\mu^c\leq1/8$, the first bound follows from
\eqref{a:eq:approxflowcoeff} and the factorial summation at the end of
the proof of Lemma~\ref{a:lem:approxmap}.
The second follows from that lemma for $\kappa\leq\kappa_1$, after
decreasing $c_1>0$ if necessary. Both conditions hold after increasing
$L_\kappa$.

\subsubsection{Gevrey estimates for the small residual correction}\label{a:sec:exactG}
With $\kappa$ as in \eqref{a:eq:kappa} and the operators in
\eqref{a:eq:pulledoperator}, denote $\widetilde e=e\circ X^a$.
Then \eqref{a:eq:errorphysical} becomes
\begin{equation}\label{a:eq:errorlabels}
\partial_t\widetilde e+V^e\cdot\nabla_y\widetilde e
+(\cT^a\widetilde e)\cdot((\nabla_x\rho^a)\circ X^a)
=-\mathcal R\circ X^a,\qquad \widetilde e(0)=0
\end{equation}
\begin{lemma}\label{a:lem:transportG}
Let $w,v,\mathfrak f$ be smooth periodic functions satisfying
$w_t+v\cdot\nabla_yw=\mathfrak f$ and $\diver_yv=0$.
For $0<\kappa\leq1$, $\gamma\in C^1([0,S];[1/2,1])$, and
$n_*\in\mathbb N_0$, the function
$t\mapsto\G^{[n_*]}_{\gamma(t),\kappa}(w(t))$ is absolutely continuous
and satisfies, almost everywhere,
\begin{equation}\label{a:eq:transportG}
\frac{d}{dt}\G^{[n_*]}_{\gamma(t),\kappa}(w)
\leq C\kappa^{-2}\G^{[n_*]}_{\gamma,\kappa}(v)
\G^{[n_*]}_{\gamma,\kappa}(w)
+\G^{[n_*]}_{\gamma,\kappa}(\mathfrak f)
+\left(\gamma'+C\kappa^{-2}\G^{[n_*]}_{\gamma,\kappa}(v)\right)
\partial_\gamma\G^{[n_*]}_{\gamma,\kappa}(w)
\end{equation}
where the constant $C$ is independent of $n_*$, $\kappa$, and $\gamma$.
\end{lemma}
\begin{proof}
Applying the Fourier estimates used in the proofs of
Lemmas~\ref{a:lem:approximation} and~\ref{a:lem:finite}, we obtain
\[
\begin{aligned}
\norm{ab}_{H^4_\kappa}
&\leq C\kappa^{-1}\norm a_{H^4_\kappa}\norm b_{H^4_\kappa},\\
\kappa^{|J|}\norm{[\partial^J,v\cdot\nabla_y]z}_2
&\leq C\kappa^{-2}\norm v_{H^4_\kappa}\norm z_{H^4_\kappa},
\qquad 1\leq|J|\leq4.
\end{aligned}
\]
The commutator vanishes for $J=0$. Thus $\diver_yv=0$ yields, for
$z_t+v\cdot\nabla_yz=h$,
\[
\frac{d}{dt}\norm z_{H^4_\kappa}
\leq C\kappa^{-2}\norm v_{H^4_\kappa}\norm z_{H^4_\kappa}
+\norm h_{H^4_\kappa}.
\]
Applying it to
\[
\partial_tD_\kappa^Iw+v\cdot\nabla_yD_\kappa^Iw
=D_\kappa^I\mathfrak f
-\kappa^{-1}\sum_{0<J\leq I}\binom IJ
(D_\kappa^Jv)\cdot D_\kappa D_\kappa^{I-J}w
\]
and using $\binom IJ\leq\binom nl$ for $|I|=n$, $|J|=l$, we obtain
\begin{equation}\label{est-Wn}
W_n'\leq C\kappa^{-2}\mathsf V_0W_n+H_n
+C\kappa^{-2}\sum_{l=1}^n\binom nl \mathsf V_lW_{n-l+1}
\end{equation}
with
\[
W_n=\kappa^n\norm{D^n w}_{H^4_{\kappa}},\quad
\mathsf V_n=\kappa^n\norm{D^n v}_{H^4_{\kappa}},\quad
H_n=\kappa^n\norm{D^n \mathfrak f}_{H^4_{\kappa}}.
\]
For $k=n-l+1$ and $1\leq l\leq n\leq n_*$,
\[
\frac{\binom nl(l!)^2(k!)^2}{(n!)^2}
=\frac{k\,l!\,k!}{(l+k-1)!}\leq k,
\qquad \gamma^n=\gamma^l\gamma^{k-1}.
\]
Since $1\leq l,k\leq n_*$, it follows that
\begin{equation}\label{est-VW}
\sum_{n=1}^{n_*}\frac{\gamma^n}{(n!)^2}
\sum_{l=1}^n\binom nl \mathsf V_lW_{n-l+1}
\leq
\left(\sum_{l=1}^{n_*}\frac{\gamma^l\mathsf V_l}{(l!)^2}\right)
\left(\sum_{k=1}^{n_*}\frac{k\gamma^{k-1}W_k}{(k!)^2}\right)
\leq\G^{[n_*]}_{\gamma,\kappa}(v)
\partial_\gamma\G^{[n_*]}_{\gamma,\kappa}(w).
\end{equation}
Applying \eqref{est-Wn} and \eqref{est-VW} verifies \eqref{a:eq:transportG}, including $n_*=0$.
\end{proof}

It follows from Proposition~\ref{a:prop:elliptic} and
\eqref{a:eq:inputsG} that
\begin{equation}\label{a:eq:linearGproduct}
\G^{[n_*]}_{\gamma,\kappa}
\!\left((\cT^a\widetilde e)\cdot((\nabla_x\rho^a)\circ X^a)\right)
\leq C\A_{1,\kappa}((\nabla_x\rho^a)\circ X^a)
\G^{[n_*]}_{\gamma,\kappa}(\cT^a\widetilde e)
\leq\bar\mu^C\G^{[n_*]}_{\gamma,\kappa}(\widetilde e).
\end{equation}
Combining \eqref{a:eq:errorlabels}, \eqref{a:eq:transportG}, and
\eqref{a:eq:linearGproduct} yields, almost everywhere,
\begin{equation}\label{a:eq:fullGenergy}
\begin{split}
\frac{d}{dt}\G^{[n_*]}_{\gamma(t),\kappa}(\widetilde e)
&\leq\bar\mu^C\G^{[n_*]}_{\gamma,\kappa}(\widetilde e)
+\bar\mu^C\kappa^{-2}
\left(\G^{[n_*]}_{\gamma,\kappa}(\widetilde e)\right)^2
+e^{-c_1\lambda^\nu\log\lambda}\\
&\quad+\left(\gamma'+\bar\mu^C\kappa^{-2}
\G^{[n_*]}_{\gamma,\kappa}(\widetilde e)\right)
\partial_\gamma\G^{[n_*]}_{\gamma,\kappa}(\widetilde e).
\end{split}
\end{equation}

\begin{proposition}\label{a:prop:exactG}
There is a sufficiently large constant $Q$ such that for
$\lambda\geq\bar\mu^Q$,
\begin{equation}\label{a:eq:exactG}
\sup_{[0,S]}\left(
\G_{1/2,\kappa}(e\circ X^a)+
\G_{1/2,\kappa}(\cT^a(e\circ X^a))\right)
\leq \bar{\mu}^C\exp(-2\lambda^{\nu/2}).
\end{equation}
\end{proposition}
\begin{proof}
Denote
\[
\gamma(t)=1-te^{-\lambda^{\nu/2}},\qquad
E_{n_*}(t)=\G^{[n_*]}_{\gamma(t),\kappa}(\widetilde e(t)).
\]
For sufficiently large $Q$, $1/2\leq\gamma\leq1$ on $[0,S]$.
Each $E_{n_*}$ is absolutely continuous and $E_{n_*}(0)=0$.
Fix $c_*$ dominating the exponents $C$ in \eqref{a:eq:fullGenergy}
and choose $Q\nu>c_*$. On the bootstrap interval
$E_{n_*}\leq e^{-2\lambda^{\nu/2}}$, increasing $Q$ yields
\[
\bar\mu^{c_*}\kappa^{-2}E_{n_*}
\leq\lambda^{2+(c_*+2L_\kappa)/Q}e^{-2\lambda^{\nu/2}}
\leq\tfrac12e^{-\lambda^{\nu/2}}.
\]
Hence
\[
\gamma'+\bar\mu^{c_*}\kappa^{-2}E_{n_*}
\leq-\tfrac12e^{-\lambda^{\nu/2}},
\qquad
E_{n_*}'\leq(\bar\mu^{c_*}+1)E_{n_*}
+e^{-c_1\lambda^\nu\log\lambda}
\]
almost everywhere. Since $S\leq1$ and
$\bar\mu^{c_*}\leq\lambda^{c_*/Q}$, Gronwall's inequality yields,
after increasing $Q$,
\[
E_{n_*}(t)\leq
S\exp(\bar\mu^{c_*}+1-c_1\lambda^\nu\log\lambda)
<\tfrac12e^{-2\lambda^{\nu/2}}.
\]
Continuity closes the bootstrap on $[0,S]$, uniformly in $n_*$.
Letting $n_*\to\infty$ by monotone convergence yields
\[
\sup_{[0,S]}\G_{\gamma(t),\kappa}(\widetilde e)
\leq\tfrac12e^{-2\lambda^{\nu/2}}.
\]
Since $\gamma\geq1/2$, Proposition~\ref{a:prop:elliptic} implies
\eqref{a:eq:exactG} and
$\G_{1/2,\kappa}(V^e)\leq\bar\mu^Ce^{-2\lambda^{\nu/2}}$.
\end{proof}

\subsubsection{Completion of the perturbation theorem}\label{a:sec:completion}
Let $(\rho^+,u^+)$ be the solution from Lemma~\ref{a:lem:finite},
with the notation of the preceding subsection. 

\medskip\noindent\emph{The particle map of $(\rho^+,u^+)$.}
It follows from \eqref{a:eq:exactG}, Proposition~\ref{a:prop:elliptic},
and scaled Sobolev embedding, after increasing $Q$, that
\begin{equation}\label{a:eq:exactvector}
\norm{D_y^n V^e}_\infty
\leq\exp(-\lambda^{\nu/2})(C\kappa^{-1})^n(n!)^2.
\end{equation}
Let $Z^e$ be the flow map of $V^e$. We have
\begin{equation}\label{a:eq:exactfactor}
X^+(t,y)=X^a(t,Z^e(t,y)),\qquad
Z^e_t=V^e(t,Z^e),\qquad Z^e(0)=\id,
\end{equation}
since
\[
\partial_t(X^a\circ Z^e)
=u^a\circ X^++(D_yX^a\circ Z^e)(V^e\circ Z^e)
=u^+\circ X^+.
\]
Apply Lemma~\ref{a:lem:smallflow} with
$C_0=e^{-\lambda^{\nu/2}}$ and $\Lambda=C\kappa^{-1}$;
increasing $Q$ ensures $C_0\Lambda\leq c_2$. Hence
\[
\norm{D_y^n(Z^e-\id)}_\infty+
\norm{D_y^n((Z^e)^{-1}-\id)}_\infty
\leq Ce^{-\lambda^{\nu/2}}(C\kappa^{-1})^n(n!)^2.
\]
Since $(X^+)^{-1}=(Z^e)^{-1}\circ(X^a)^{-1}$,
\[
\norm{X^+-X^a}_{L^\infty_y}
+\norm{(X^+)^{-1}-(X^a)^{-1}}_{L^\infty_x}
\leq\bar\mu^Ce^{-\lambda^{\nu/2}}.
\]
Together with \eqref{a:eq:approxmapestimates} and
$e^{-\lambda^{\nu/2}}\leq\lambda^{-1}$, this proves \eqref{a:eq:mapzero}.

Since $e^{-\lambda^{\nu/2}}\kappa^{-1}$ is small, for $k\geq1$,
\[
\norm{D_y^k Z^e}_\infty+
\norm{D_y^k(Z^e)^{-1}}_\infty
\leq C(C\kappa^{-1})^{k-1}(k!)^2.
\]
The positive-order composition argument in the proof of
Lemma~\ref{a:lem:approxmap}, applied to \eqref{a:eq:exactfactor} and the
inverse factorization, together with \eqref{a:eq:approxmapestimates}, yields
\[
\norm{D_y^kX^+}_\infty+\norm{D_x^k(X^+)^{-1}}_\infty
\leq\bar\mu^C(\lambda\bar\mu^C)^{k-1}(k!)^2.
\]
Since $\diver_xu^+=0$, we have $\det D_yX^+=1$ and
\begin{equation}\label{a:eq:cofactor}
(D_yX^+)^{-1}=\cof(D_yX^+)^T.
\end{equation}
Taking $k=n+1$, using $(n+1)^2\leq4^{n+1}$ and the linearity of the
cofactor map in two dimensions, proves \eqref{a:eq:maphigh} and
\eqref{a:eq:inversehigh}.

\medskip\noindent\emph{Physical gradient estimates.}
The chain rule yields
\[
\begin{aligned}
(\nabla_xe)\circ X^a&=(D_yX^a)^{-T}\nabla_y\widetilde e,\\
(\nabla_x(u^+-u^a))\circ X^a
&=D_y(\cT^a\widetilde e)(D_yX^a)^{-1}.
\end{aligned}
\]
By \eqref{a:eq:exactG}, \eqref{a:eq:approxmapestimates}, scaled Sobolev
embedding, and \eqref{a:eq:factorialproduct},
\[
\norm{D^n((\nabla_xe,\nabla_x(u^+-u^a))\circ X^a)}_{L^\infty_y}
\leq\bar\mu^C\kappa^{-2}e^{-2\lambda^{\nu/2}}
(C\kappa^{-1})^n(n!)^2.
\]
For $Y^a=X^{-1}\circ X^a$,
\[
((\nabla_x\rho^a,\nabla_xu^a)\circ X^a)
=((\nabla_x\rho^a,\nabla_xu^a)\circ X)\circ Y^a.
\]
The estimates \eqref{a:eq:Lestimates} for $Y^a$, established in the proof
of Lemma~\ref{a:lem:approxmap}, together with \eqref{a:eq:approxjetestimates}
and Lemma~\ref{a:lem:calculus}(c) imply
\[
\norm{D^n((\nabla_x\rho^a,\nabla_xu^a)\circ X^a)}_{L^\infty_y}
\leq\bar\mu^C(\lambda\bar\mu^C)^n(n!)^2.
\]
Adding the correction and composing with $Z^e$ yields
\begin{equation}\label{a:eq:exactjetall}
\norm{D^n((\nabla_x\rho^+,\nabla_xu^+)\circ X^+)}_{L^\infty_y}
\leq \bar{\mu}^C(\lambda \bar{\mu}^C)^n(n!)^2.
\end{equation}
Darcy's law yields
\[
(\nabla_x^2\pi^+)\circ X^+
=-(\nabla_xu^+)\circ X^+
-e_2\otimes((\nabla_x\rho^+)\circ X^+).
\]
Together with \eqref{a:eq:exactjetall}, this proves \eqref{a:eq:fieldhigh}.

\medskip\noindent\emph{Comparison with the principal tensors.}
It follows from Darcy's law 
\begin{equation}\label{a:eq:Hdependent}
\nabla_x^2\pi^+-\nabla_x^2\bar\pi
=-\nabla_x(u^+-U)-e_2\otimes\nabla_x(\rho^+-\bar\rho).
\end{equation}
Since $M_\lambda=-H_\lambda-e_2\otimes g_\lambda$, we have
\[
\norm{(\nabla_x^2\pi^+-\nabla_x^2\bar\pi)\circ X-H_\lambda}_\infty
\leq\norm{\nabla_x(u^+-U)\circ X-M_\lambda}_\infty
+\norm{\nabla_x(\rho^+-\bar\rho)\circ X-g_\lambda}_\infty.
\]
Applying \eqref{a:eq:approxjetestimates}, \eqref{a:eq:finitecorrection}, and
\[
\norm{\nabla_xe}_\infty+\norm{\nabla_x\cT e}_\infty
\leq C\lambda^2\norm e_{H^4_{1/\lambda,x}}\leq C\lambda^2e^{-\lambda^{\nu/2}},
\]
the errors in \eqref{a:eq:jeterror} are bounded by
\[
2(\bar{\mu}^C\lambda^{-1}+C\lambda^2e^{-\lambda^{\nu/2}})\leq\lambda^{-1/2},
\]
after increasing $Q$.

\medskip\noindent\emph{The initial density.}
Since $e(0)=0$, $X(0)=\id$, and $r_n(0)=0$ for $n\geq2$ by
\eqref{a:eq:osceq}, the expansion \eqref{a:eq:rhoapprox} yields
\[
\rho^+(0)=\bar\rho(0)+r_\lambda(0)+b_\lambda,
\qquad b_\lambda=\sum_{n=2}^N\lambda^{-n}\bar r_n(0).
\]
Lemma~\ref{a:lem:initialslow} supplies the support and $H^m$ estimates
for $b_\lambda$ in \eqref{a:eq:initial}.
It follows from \eqref{a:eq:data}--\eqref{a:eq:carriersize} and
\eqref{a:eq:murelative} that
\[
\norm{r_\lambda(0)}_{H^m}
\leq C_m\bar{\mu}^{C_m}\mathrm e^{-\Gamma}\lambda^{m-1},
\qquad m\in\mathbb N_0,
\]
which completes \eqref{a:eq:initial}.

\medskip\noindent\emph{Symmetry.}
Under the reflection hypotheses, $r_\lambda(0)$ is odd, and
Proposition~\ref{a:prop:majorant} implies oddness of $b_\lambda$.
Thus $\rho^+(0)=\rho^a(0)$ is odd. Reflection invariance and uniqueness
imply the symmetries in Theorem~\ref{thm:main}.
If $0\in\mathcal O$, oddness implies $\phi_0(0)=0$ and $X(t,0)=0$,
so the origin is a zero-phase trajectory. This completes the proof of
Theorem~\ref{thm:main}.

\medskip\noindent\emph{A pressure consequence.}
Take $f=f_\delta$ from Lemma~\ref{amp:lem:profile} and recall
$A=\mu|\xi|$, $p=\xi/|\xi|$.
On any interval where $A\chi\leq2h_+$ and $p_2\geq0$ on $\supp\chi$,
it follows from \eqref{a:eq:principaljets} and \eqref{amp:eq:profile} that
\[
H_\lambda=-A\chi p_2f_\delta'(\lambda\phi_0)\,p\otimes p
\leq C\delta h_+I.
\]
After zero extension, \eqref{a:eq:jeterror} and composition with $X^{-1}$ yield
\begin{equation}\label{a:eq:uppernewsolution}
\nabla_x^2\pi^+(t,x)\leq\nabla_x^2\bar\pi(t,x)+
(C\delta h_++\lambda^{-1/2})I.
\end{equation}
On $[0,t_0]\subset[0,S]$, suppose instead that, for a nonnegative
$K_+\in L^1(0,t_0)$,
\[
A(t_0,y)\leq2h_+e^{-\Gamma},\qquad
H(t,y)\leq K_+(t)I
\quad(0\leq t\leq t_0,\ y\in\supp\chi).
\]
Then \eqref{amp:eq:fullhistory} and \eqref{amp:eq:historypressure},
together with \eqref{a:eq:jeterror} imply for $0\leq t\leq t_0$,
\[
\nabla_x^2\pi^+(t,x)\leq\nabla_x^2\bar\pi(t,x)
\quad+\left[2h_+\exp\!\left(-\Gamma+\int_0^{t_0}K_+(\tau)\,d\tau\right)
+\lambda^{-1/2}\right]I.
\]
This yields \eqref{a:eq:uppernewsolution} on $[0,t_0]$ if
$\exp(-\Gamma+\int_0^{t_0}K_+)\leq\delta$.

\subsubsection{Phase and amplitude estimates for the next perturbation}\label{a:sec:carriers}
Use $(\rho^+,u^+,\pi^+)$ as the next background, with $\lambda$
denoting its construction frequency.

\medskip\noindent\emph{The transported phase.}
Fix $T\in[0,S]$ and $\xi_*\in\R^2\setminus\{0\}$. Choose a label
neighborhood $\mathcal O_*$ of zero whose terminal image lies in a
physical coordinate chart, and fix lifts of the particle maps. Define
\begin{equation}\label{a:eq:terminalphase}
\begin{aligned}
\psi_0(y)&=\xi_*\cdot(X^+(T,y)-X^+(T,0)),\\
\zeta(t,y)&=(D_yX^+(t,y))^{-T}(D_yX^+(T,y))^T\xi_*.
\end{aligned}
\end{equation}
For $\psi(t,X^+(t,y))=\psi_0(y)$,
\[
\begin{split}
&\nabla_y\psi_0=(D_yX^+(T,y))^T\xi_*,\qquad
(\nabla_x\psi)\circ X^+=\zeta,\qquad \zeta(T,y)=\xi_*,\\
&\psi(T,x)=\xi_*\cdot(x-X^+(T,0)),\qquad
\psi(t,X^+(t,0))=0.
\end{split}
\]
It follows from \eqref{a:eq:maphigh} and \eqref{a:eq:factorialproduct} that
\begin{equation}\label{a:eq:nonstationary}
\bar{\mu}^{-C}|\xi_*|\leq|\zeta(t,y)|\leq \bar{\mu}^C|\xi_*|
\end{equation}
and
\[
\norm{D^n \zeta}_{L^\infty_y}
\leq |\xi_*|\bar{\mu}^C(\lambda \bar{\mu}^C)^n(n!)^2,
\qquad n\geq0.
\]
Denote $\widehat\zeta=\zeta/|\xi_*|$. Differentiating
$|\widehat\zeta|^2|\widehat\zeta|^{-2}=1$ yields, for $|I|\geq1$,
\[
\partial^I|\widehat\zeta|^{-2}
=-|\widehat\zeta|^{-2}
\sum_{0<J\leq I}\binom{I}{J}
(\partial^J|\widehat\zeta|^2)
(\partial^{I-J}|\widehat\zeta|^{-2}).
\]
By \eqref{a:eq:nonstationary}, \eqref{a:eq:factorialproduct}, and induction,
\[
\norm{D^n(|\zeta|^{-2})}_{L^\infty_y}
\leq |\xi_*|^{-2}\bar{\mu}^C(\lambda \bar{\mu}^C)^n(n!)^2,
\qquad n\geq0.
\]

\medskip\noindent\emph{The amplitude coefficient and logarithmic derivatives.}
Denote
\[
(g^+,M^+,H^+)
=(\nabla_x\rho^+,\nabla_xu^+,\nabla_x^2\pi^+)\circ X^+.
\]
Since $m(\zeta)=m(\widehat\zeta)=\widehat\zeta_2\widehat\zeta/|\widehat\zeta|^2-e_2$,
\eqref{a:eq:fieldhigh} and the preceding product and reciprocal estimates imply
\begin{equation}\label{a:eq:nextcarriercoeff}
\norm{D^n(m(\zeta)\cdot g^+)}_{L^\infty_y}
\leq \bar{\mu}^C(\lambda \bar{\mu}^C)^n(n!)^2, \qquad n\geq0,
\end{equation}
with constants independent of $\xi_*$. For constant $\mu_T>0$, the
solution of $\partial_t\mu_*+m(\zeta)\cdot g^+\,\mu_*=0$,
$\mu_*(T,y)=\mu_T$, is positive and satisfies
\begin{equation}\label{a:eq:logcarrier}
\log\mu_*(t,y)=\log\mu_T
+\int_t^T m(\zeta(s,y))\cdot g^+(s,y)\,ds.
\end{equation}
Differentiate \eqref{a:eq:logcarrier} and use $|T-t|\leq1$ and
$(n+1)^2\leq4^{n+1}$. The pointwise Fa\`a di Bruno argument in
\eqref{a:eq:murelative}, with $\mu_*$ cancelled after division, then yields
\begin{equation}\label{a:eq:carrierderivatives}
\begin{split}
&\norm{D^n(\nabla_y\log\mu_*)}_{L^\infty_y}
\leq(\lambda \bar{\mu})^C((\lambda \bar{\mu})^C)^n(n!)^2, \qquad n\geq0,\\
&\sup_{t,y}\mu_*^{-1}|D_y^n\mu_*|
\leq((\lambda \bar{\mu})^C)^n(n!)^2\qquad n\geq1.
\end{split}
\end{equation}
The constants are independent of $\mu_T$ and $\xi_*$. 

Denote $A_*=\mu_*|\zeta|$ and $p_*=\zeta/|\zeta|$.
Since $\partial_t\zeta=-(M^+)^T\zeta$, Lemma~\ref{amp:lem:pressure} yields
\begin{equation}\label{a:eq:pressureamplitude}
\partial_t\log A_*=-p_*^\perp\cdot H^+p_*^\perp.
\end{equation}
If $K_*\in L^1(0,S)$ is nonnegative and
$H^+(t,y)\leq K_*(t)I$ for every $y\in\mathcal O_*$ and almost every $t$,
then
\begin{equation}\label{a:eq:historyamplitude}
A_*(s,y)\leq A_*(t,y)
\exp\left(\int_s^tK_*(\tau)\,d\tau\right),\qquad 0\leq s\leq t\leq S.
\end{equation}

We have $A_*(T,y)=\mu_T|\xi_*|$. Assume also that, for some
$0\leq t_0\leq T$ and $\Gamma_*\geq0$,
\[
\sup_{y\in\mathcal O_*}A_*(t_0,y)\leq C\mu_T|\xi_*|e^{-\Gamma_*},
\qquad
\sup_{[t_0,S]\times\mathcal O_*}A_*\leq C\mu_T|\xi_*|.
\]
It follows from \eqref{a:eq:nonstationary} and \eqref{a:eq:historyamplitude} that
\[
\begin{aligned}
\sup_{[0,S]\times\mathcal O_*}\mu_*
&\leq C\bar\mu^C\mu_T\exp\left(\int_0^{t_0}K_*(s)\,ds\right),\\
\sup_{y\in\mathcal O_*}\mu_*(0,y)
&\leq C\bar\mu^C\mu_T\exp\left(-\Gamma_*+\int_0^{t_0}K_*(s)\,ds\right).
\end{aligned}
\]
These imply \eqref{a:eq:carriersize} with $\Gamma=\Gamma_*$ when the
next data constant is at least
\[
C\bar\mu^C\mu_T\exp\left(\int_0^{t_0}K_*(s)\,ds\right).
\]

\medskip\noindent\emph{A uniform material neighborhood.}
Suppose $X^+(t,0)=0$ and, for $t\in I\subset[0,S]$,
\[
|g^+(t,0)-g_*|+\norm{M^+(t,0)-M_*}+\norm{H^+(t,0)-H_*}
\leq\tau/4,
\]
where $\tau>0$ and $(g_*,M_*,H_*)$ is a reference triple.
Choose $0<R_0\leq1$ with $\overline{B_{R_0}}\subset\mathcal O_*$ and denote
\begin{equation}\label{a:eq:retainedradius}
\begin{split}
R_*&=c\min\{R_0,\tau(\lambda\bar\mu)^{-C}\},\\
R_*^{-1}&\leq C\bigl(R_0^{-1}+\tau^{-1}(\lambda\bar\mu)^C\bigr).
\end{split}
\end{equation}
For sufficiently small absolute $c>0$, \eqref{a:eq:fieldhigh} with $n=1$
and the mean-value estimate imply for $t\in I$ and $y\in B_{R_*}$,
\[
\begin{split}
&\quad|g^+(t,y)-g_*|+\norm{M^+(t,y)-M_*}+\norm{H^+(t,y)-H_*}\\
&\leq\tau/4+C|y|(\lambda\bar\mu)^C\leq\tau/2.
\end{split}
\]
An even Gevrey-two cutoff $0\leq\chi_*\leq1$, supported in $B_{R_*}$
and equal to one on $B_{R_*/2}$, can be chosen by rescaling a fixed cutoff,
with
\[
\norm{D^n\chi_*}_{H^8(\T^2)}
\leq\bigl[C(1+R_*^{-1})^8\bigr]^{n+1}(n!)^2,
\qquad n\geq0.
\]

\section{The inductive construction}\label{sec:induction}
We combine principal amplification with Theorem~\ref{thm:main} iteratively to prove the main result Theorem \ref{thm:blowup}. 
The induction carries on a local gradient configuration, global upper pressure
bound, and Gevrey estimates on successively smaller material neighborhoods.

\subsection{Parameters and induction hypotheses}\label{geo:sec:setup}
Fix the constants in Theorem~\ref{thm:main}, Proposition~\ref{amp:prop:robust}, Lemma~\ref{geo:lem:retention}, and
Corollary~\ref{geo:cor:exactretention}. We choose first
$0<\eps_*<\min(c_0,1)/100$ with $c_0$ from
Proposition~\ref{amp:prop:robust}; $K\geq1$ sufficiently large for
Lemma~\ref{geo:lem:retention} and the error allocation below; the small profile small constants $d, d_0>0$ chosen below; 
and $\theta_0=\arctan(N_0^{-1})\leq1/10$, with $N_0\in\mathbb N$ sufficiently large.

For $j\in\mathbb N_0$, define
\begin{equation}\label{geo:eq:scalars}
\theta_{j+1} = \theta_j^2,\qquad
h_{j+1} = Kh_j\theta_j^{-14},\qquad
h_0 = \theta_0^{-4}.
\end{equation}
Denote
$n_j=((-1)^j\sin\theta_j,\cos\theta_j)$ and we have the reference triple
\begin{equation}\label{geo:eq:frozentriple}
(h_jn_j,\ h_jm(n_j)\otimes n_j,
-h_j\cos\theta_j\,n_j\otimes n_j),
\end{equation}
which is \eqref{amp:eq:frozenjets} with
$(h,\theta,\sigma)=(h_j,\theta_j,(-1)^j)$.

We choose
\begin{equation}\label{geo:eq:times}
\begin{split}
&t_0 = 0,\qquad t_{j+1} = t_j+
\frac{\cot(\theta_j+\theta_j^2)}{h_j\sin\theta_j},\\
&S_j = t_j+\frac{4}{h_j\theta_j^2},\qquad
\Gamma_j = \cot\theta_j\left(\frac\pi2-\theta_j-\theta_j^2\right),
\end{split}
\end{equation}
where $t_{j+1}-t_j$ and $\Gamma_j$ are the reference duration and
logarithmic gain in \eqref{amp:eq:targettime}.
It follows from \eqref{amp:eq:timegainestimates} and \eqref{geo:eq:scalars},
\begin{equation}\label{geo:eq:widthratio}
t_{j+1}-t_j\leq\frac{2}{h_j\theta_j^2},\qquad \Gamma_j\geq\theta_j^{-1},\qquad
\frac{h_j\theta_j^2}{h_{j+1}\theta_{j+1}^2}
=K^{-1}\theta_j^{12},\qquad S_{j+1}<S_j.
\end{equation}
Indeed,
\[
S_{j+1}-S_j
\leq\frac{-2+4K^{-1}\theta_j^{12}}{h_j\theta_j^2}<0.
\]
Since $t_{j+1}>t_j$ and $S_0=4\theta_0^2<1$, we have
\[
0\leq t_j<t_{j+1}<S_{j+1}<S_j\leq S_0<1.
\]

The profile concentration and prescribed frequencies are
\begin{equation}\label{geo:eq:frequencies}
\delta_j = d\,2^{-j}h_{j+1}^{-2},\qquad
\lambda_j = \exp(\theta_j^{-1/2}).
\end{equation}
Choose $d>0$ small enough for Lemma~\ref{amp:lem:profile}.

At stage $j$, let $(\rho_j,u_j,\pi_j)$ be a smooth solution on
$[0,S_j]$, with $\rho_j$ odd, $u_j=\cT\rho_j$, and normalized periodic pressure $\pi_j$.
Its particle map satisfies
\[
\partial_tX_j(t,y)=u_j(t,X_j(t,y)),\qquad
X_j(0,y)=y.
\]
Oddness implies $X_j(t,0)=0$. 

The induction maintains a label radius $R_j>0$ with the estimate
\begin{equation}\label{geo:eq:IHgeometry}
\begin{aligned}
\sup_{t_j\leq t\leq S_j,\ |y|\leq4R_j}\Bigl(
&| (\nabla_x\rho_j)\circ X_j-h_jn_j|
+\norm{(\nabla_xu_j)\circ X_j-h_jm(n_j)\otimes n_j}\\
&+\norm{(\nabla_x^2\pi_j)\circ X_j+h_j\cos\theta_j\,n_j\otimes n_j}\Bigr)
\leq\eps_*h_j\theta_j^6.
\end{aligned}
\end{equation}

We also require the physical upper bound
\begin{equation}\label{geo:eq:IHH}
\nabla_x^2\pi_j(t,x)\leq B_j I,
\qquad 0\leq t\leq S_j,\quad x\in\T^2,\qquad 0\leq B_j\leq\tfrac12,
\end{equation}
and 
\begin{equation}\label{geo:eq:nestedtubes}
X_j(t,B_{4R_j})\subset X_{j-1}(t,B_{R_{j-1}}),\qquad 0\leq t\leq S_j ,\quad j\geq 1.
\end{equation}

Finally, we require a constant $D_j\geq2$ for the Gevrey estimates
\begin{equation}\label{geo:eq:Ddata}
\begin{split}
&\sup_{0\leq t\leq S_j}\norm{D^n w}_{W^{8,\infty}(\T^2)}
\leq D_j^{n+1}(n!)^2,\\
&w\in\{D_yX_j,(D_yX_j)^{-1},(\nabla_xu_j)\circ X_j,
(\nabla_x^2\pi_j)\circ X_j,\nabla_y\rho_j(0)\},\\
&\norm{D^n(\rho_j(0))}_{H^8(\T^2)}\leq D_j^{n+1}(n!)^2,
\qquad n\in\mathbb N_0.
\end{split}
\end{equation}

If $\norm{D^n w}_{L^\infty}\leq A\Lambda^n(n!)^2$ for 
$n\in\mathbb N_0$ with $A,\Lambda\geq2$, we have on $\T^2$,
\begin{equation}\label{geo:eq:normconversion}
\norm{D^n w}_{W^{8,\infty}}+\norm{D^n w}_{H^8}
\leq CA\Lambda^8(4\Lambda)^n(n!)^2
\leq (CA\Lambda^8)^{n+1}(n!)^2
\end{equation}
where the constant $C$ is independent of $n$.
Indeed, it follows from taking sum the estimates for $D^{n+k}w$ with $0\leq k\leq8$, applying
$(n+k)!\leq2^{n+k}n!k!$ and the estimate
$\norm q_2\leq|\T^2|^{1/2}\norm q_\infty$.

\subsection{The stationary periodic solution}\label{geo:sec:stationary}
Fix $d_0>0$ sufficiently small, independently of $N_0$, and denote
\[
k=(1,N_0),\qquad n_0=k/|k|,\qquad |k|=\csc\theta_0,
\qquad \eta_0=d_0\theta_0^4.
\]
Let $f_{\eta_0}$ be the odd, zero-mean, $2\pi$-periodic profile from
Lemma~\ref{amp:lem:profile}, so that
\[
f_{\eta_0}(s)=s\quad(|s|\leq\eta_0/2),\qquad
-C\eta_0\leq f_{\eta_0}'\leq1,\qquad
\norm{f_{\eta_0}}_\infty\leq C\eta_0,
\]
with $C\eta_0\leq1$. Let $\Phi$ be its zero-mean periodic primitive and define
\begin{equation}\label{geo:eq:seed}
\rho_0(x)=\frac{h_0}{|k|}f_{\eta_0}(k\cdot x), \quad
u_0(x)=m(n_0)\rho_0(x),\quad
\pi_0(x)=-\frac{h_0\cos\theta_0}{|k|^2}\Phi(k\cdot x)
\end{equation}
which are periodic and have zero mean. We note $\rho_0,u_0$ are odd and
$\pi_0$ is even. Since $m(n_0)\cdot k=0$, we have
\[
\begin{aligned}
\nabla_x\pi_0&=-\rho_0\cos\theta_0\,n_0,
&-\nabla_x\pi_0-\rho_0e_2&=u_0,\\
\diver_xu_0&=0,
&u_0\cdot\nabla_x\rho_0&=0.
\end{aligned}
\]
Thus $(\rho_0,u_0,\pi_0)$ is a global smooth stationary solution of
\eqref{eq:ipm} with $u_0=\cT\rho_0$.

The flow and pressure Hessian are
\begin{equation}\label{geo:eq:seedformulas}
\begin{aligned}
X_0(t,y)&=y+tu_0(y),\\
\nabla_x^2\pi_0(x)&=-h_0\cos\theta_0\,n_0\otimes n_0\,
f_{\eta_0}'(k\cdot x).
\end{aligned}
\end{equation}
In the chart starting from the identity, we have
\[
k\cdot X_0(t,y)=k\cdot y,\qquad
X_0(t,\cdot)^{-1}(x)=x-tu_0(x).
\]
Since
\[
\nabla_xu_0(y)=h_0m(n_0)\otimes n_0\,f_{\eta_0}'(k\cdot y),
\qquad (\nabla_xu_0(y))^2=0,
\]
it follows that
\[
D_yX_0=I+t\nabla_xu_0(y),\qquad
(D_yX_0)^{-1}=I-t\nabla_xu_0(y),\qquad \det D_yX_0=1.
\]
Applying $|m(n_0)|=\sin\theta_0$ and $S_0=4\theta_0^2$ yields
\[
\sup_{0\leq t\leq S_0}
\bigl(\norm{D_yX_0}_\infty+\norm{(D_yX_0)^{-1}}_\infty\bigr)
\leq2(1+S_0h_0\sin\theta_0)\leq C\theta_0^{-1}.
\]
Choose $R_0=c\eta_0/|k|$ with a fixed $0<c\leq1/16$. For $|y|\leq4R_0$, we have
\[
|k\cdot X_0(t,y)|=|k\cdot y|\leq4c\eta_0\leq\eta_0/4,
\qquad f_{\eta_0}'(k\cdot X_0(t,y))=1.
\]
Hence the gradient triple $(\nabla_x\rho_0,\ \nabla_xu_0,\ \nabla_x^2\pi_0)$ is given by \eqref{geo:eq:frozentriple} with $j=0$,
and \eqref{geo:eq:IHgeometry} holds with zero error.
Since $X_0(t,0)=0$, we have
\[
\sup_{\substack{0\leq t\leq S_0\\|y|\leq4R_0}}|X_0(t,y)|
\leq C\theta_0^{-1}R_0\leq Cd_0\theta_0^4,
\]
and hence these images remain in a fixed physical chart for $N_0$ sufficiently large.
Moreover, \eqref{geo:eq:seedformulas} yields
\[
\nabla_x^2\pi_0\leq Ch_0\eta_0 I=Cd_0 I\leq\tfrac14 I,
\]
which verifies \eqref{geo:eq:IHH} with $B_0=1/4$.

For the Gevrey estimates, Lemma~\ref{amp:lem:profile} yields
\[
\norm{D^n(f_{\eta_0}'(k\cdot y))}_{L^\infty}
\leq C\Lambda_0^n(n!)^2,\qquad
\Lambda_0=C|k|\eta_0^{-1},\qquad n\geq0.
\]
Applying \eqref{geo:eq:seed}--\eqref{geo:eq:seedformulas} and phase invariance yield
\[
\sup_{0\leq t\leq S_0}\norm{D^n w}_{L^\infty}
\leq Ch_0\Lambda_0^n(n!)^2
\]
for any $w$ in \eqref{geo:eq:Ddata} at $j=0$, including $\rho_0$.
By \eqref{geo:eq:normconversion}, one constant $D_0\geq2$ therefore
satisfies \eqref{geo:eq:Ddata} for all $n\geq0$, with
\[
D_0\leq Ch_0\Lambda_0^8\leq Cd_0^{-8}\theta_0^{-44},\qquad
R_0^{-1}\leq Cd_0^{-1}\theta_0^{-5}.
\]
Since $d_0$ is independent of $N_0$, increasing the exponent
and taking $N_0$ sufficiently large yields
\begin{equation}\label{geo:eq:seedP}
D_0+R_0^{-1}\leq\theta_0^{-C}.
\end{equation}
This initializes the induction. Note \eqref{geo:eq:nestedtubes} is void at $j=0$.

\subsection{The inductive step}\label{sec:inductionstep}
Let $(\rho_j,u_j,\pi_j)$ be an odd smooth solution of \eqref{eq:ipm} on $[0,S_j]$
satisfying \eqref{geo:eq:IHgeometry}, \eqref{geo:eq:IHH}, and
\eqref{geo:eq:Ddata}, with $0<R_j\leq1$. 
The frequency conditions below are verified in Subsection~\ref{geo:sec:budgets}.
Next we show that the inductive hypotheses can be verified.
\begin{proposition}\label{prop:step}
There is a constant $\bar{\mu}_j\geq2$ such that
\begin{equation}\label{eq:Pjconstruction}
\bar{\mu}_j\leq\left[D_jR_j^{-1}\delta_j^{-1}
(1+h_{j+1})(1+\theta_j^{-1})\right]^C.
\end{equation}
For a sufficiently large $K$, suppose
$\lambda_j\geq\max(\underline{\lambda},\bar{\mu}_j^Q)$ and
\begin{equation}\label{geo:eq:Eallocation}
\lambda_j^{-1/2}\leq\frac{\eps_*}{8}
h_{j+1}\theta_{j+1}^6.
\end{equation}
Then there is an odd solution of \eqref{eq:ipm} on $[0,S_{j+1}]$ satisfying
\eqref{geo:eq:IHgeometry}, \eqref{geo:eq:nestedtubes}, and
\eqref{geo:eq:Ddata} at index $j+1$, for suitable $R_{j+1}>0$ and
$D_{j+1}\leq(\lambda_j\bar{\mu}_j)^C$. The error in the coeffient tensors at $x=0$ is at most
$\eps_*h_{j+1}\theta_{j+1}^6/4$ on $[t_{j+1},S_{j+1}]$.
The initial increment and $R_{j+1}^{-1}$ satisfy
\eqref{geo:eq:initialestimate} and \eqref{geo:eq:radiuscost}. Moreover,
\begin{equation}\label{geo:eq:Hrecursion}
\nabla_x^2\pi_{j+1}(t,x)\leq\nabla_x^2\pi_j(t,x)+
\left(2\mathrm e h_{j+1}\mathrm e^{-\Gamma_j}
+C\delta_jh_{j+1}+\lambda_j^{-1/2}\right)I
\end{equation}
for all $(t,x)\in[0,S_{j+1}]\times\T^2$. The constants and exponents are independent of $j$.
\end{proposition}

\begin{proof}
\noindent\emph{Phase and cutoff.}
It follows from \eqref{geo:eq:widthratio} that $0<t_{j+1}<S_{j+1}<S_j<1$.
Choose an even Gevrey-two cutoff with
\[
0\leq\chi_j\leq1,\qquad \chi_j|_{B_{R_j}}=1,\qquad
\supp\chi_j\subset B_{2R_j},
\]
and regard $B_{3R_j}$ as the phase neighborhood. Define
\begin{equation}\label{geo:eq:terminaldata}
\begin{aligned}
\phi_{j,0}(y)&=\csc(\theta_j+\theta_j^2)n_{j+1}\cdot X_j(t_{j+1},y),\\
\xi_j(t,y)&=\csc(\theta_j+\theta_j^2)(D_yX_j(t,y))^{-T}(D_yX_j(t_{j+1},y))^Tn_{j+1},\\
\mu_j(t_{j+1},y)&=h_{j+1}\sin(\theta_j+\theta_j^2).
\end{aligned}
\end{equation}
Denote $\phi_j\circ X_j=\phi_{j,0}$, so that $\phi_j$ is transported and
$(\nabla_x\phi_j)\circ X_j=\xi_j$. With
$g_j=(\nabla_x\rho_j)\circ X_j$, we solve
\[
\partial_t\mu_j+m(\xi_j)\cdot g_j\,\mu_j=0
\quad\text{on }[0,S_{j+1}].
\]
Then $\mu_j>0$, and $p_j=\xi_j/|\xi_j|$, $A_j=\mu_j|\xi_j|$ satisfy
\[
p_j(t_{j+1},y)=n_{j+1},\qquad A_j(t_{j+1},y)=h_{j+1}.
\]
The deformation estimates in \eqref{geo:eq:Ddata} imply
\[
D_j^{-2}\csc(\theta_j+\theta_j^2)
\leq|\xi_j(t,y)|\leq D_j^2\csc(\theta_j+\theta_j^2).
\]
Oddness of $X_j$ implies $\phi_{j,0}$ is odd and $\xi_j,g_j,\mu_j$ are even.

\medskip\noindent\emph{Amplitude estimates.}
Apply Proposition~\ref{amp:prop:robust} after translating time by $t_j$,
with $(h,\theta,\vartheta,\sigma)=(h_j,\theta_j,\theta_j^2,(-1)^j)$.
By \eqref{geo:eq:IHgeometry}, the smallness condition
\eqref{amp:eq:smallness} holds with $\eps=\eps_*\theta_j^6$.
Thus \eqref{geo:eq:times} and \eqref{geo:eq:terminaldata} yield
\begin{equation}\label{geo:eq:activationestimates}
A_j(t_j,y)\leq2h_{j+1}\mathrm e^{-\Gamma_j},\qquad
A_j\leq2h_{j+1},\qquad (p_j)_2\geq\tfrac12\sin\theta_j
\end{equation}
on $[t_j,t_{j+1}]$. Since
\[
S_{j+1}-t_{j+1}=\frac4{h_{j+1}\theta_{j+1}^2},\qquad
\frac{h_j}{h_{j+1}}\left(\frac{\theta_j}{\theta_{j+1}}\right)^2
=K^{-1}\theta_j^{12},
\]
Lemma~\ref{geo:lem:retention} implies, for sufficiently large $K$,
\[
A_j\leq2h_{j+1},\qquad (p_j)_2\geq\tfrac12
\quad\text{on }[t_{j+1},S_{j+1}].
\]
Applying \eqref{geo:eq:IHH} and \eqref{a:eq:historyamplitude} yields
\begin{equation}\label{geo:eq:historybound} 
\sup_{0\leq t\leq t_j}A_j(t,y)
\leq2h_{j+1}\mathrm e^{-\Gamma_j+B_jt_j}
\leq2\mathrm e h_{j+1}\mathrm e^{-\Gamma_j}.
\end{equation}
Consequently, $\mu_j=A_j/|\xi_j|$ satisfies
\[
\begin{aligned}
0<\mu_j(t,y)&\leq
2\mathrm e D_j^2h_{j+1}\sin(\theta_j+\theta_j^2),
\quad 0\leq t\leq S_{j+1},\\
\mu_j(0,y)&\leq
2\mathrm e D_j^2h_{j+1}\sin(\theta_j+\theta_j^2)
\mathrm e^{-\Gamma_j}.
\end{aligned}
\]

\medskip\noindent\emph{The data constant.}
It follows from \eqref{a:eq:logcarrier} that
\[
g_j=(D_yX_j)^{-T}\nabla_y\rho_j(0),\qquad
\nabla_y\log\mu_j(t)=\int_t^{t_{j+1}}
\nabla_y\bigl(m(\xi_j)\cdot g_j\bigr)(s)\,ds.
\]
Denote $\widehat\xi_j=\sin(\theta_j+\theta_j^2)\xi_j$. Clearly, we see
$D_j^{-2}\leq|\widehat\xi_j|\leq D_j^2$ and
$m(\xi_j)=m(\widehat\xi_j)$. Applying the product and reciprocal estimates in
Subsection~\ref{a:sec:carriers} to \eqref{geo:eq:Ddata} yields
\[
\norm{D^n w}_{W^{8,\infty}}\leq D_j^{C(n+1)}(n!)^2,\qquad n\geq0,
w\in\{g_j,\widehat\xi_j,|\widehat\xi_j|^{-2},
m(\widehat\xi_j),\nabla_y\log\mu_j\}.
\]
For the remaining phase data, \eqref{geo:eq:terminaldata} and
$\csc(\theta_j+\theta_j^2)\leq C\theta_j^{-1}$ provide the same estimates with
an additional factor $C\theta_j^{-1}$ for $\nabla_y\phi_{j,0}$ and $\xi_j$;
also $|\xi_j|^{-2}=\sin^2(\theta_j+\theta_j^2)|\widehat\xi_j|^{-2}$.
Rescaling the cutoff and applying Lemma~\ref{amp:lem:profile} yields
\[
\norm{D^n\chi_j}_{H^8}\leq(CR_j^{-8})^{n+1}(n!)^2,
\norm{D^n f_{\delta_j}}_{H^8}\leq(C\delta_j^{-8})^{n+1}(n!)^2.
\]
The initial-density estimate follows from \eqref{geo:eq:Ddata}, and
\[
\exp\left(\int_0^{S_{j+1}}B_j\,dt\right)\leq\mathrm e^{1/2}.
\]
Thus a sufficiently large power of
$D_jR_j^{-1}\delta_j^{-1}(1+h_{j+1})(1+\theta_j^{-1})$ defines
$\bar{\mu}_j\geq D_j$ satisfying \eqref{eq:Pjconstruction} and
\eqref{a:eq:data}--\eqref{a:eq:positiveH} with
$(S,\Gamma,K_+)=(S_{j+1},\Gamma_j,B_j)$.

\medskip\noindent\emph{New solution.}
Theorem~\ref{thm:main} provides an odd solution
$(\rho_{j+1},u_{j+1},\pi_{j+1})$ of \eqref{eq:ipm} on $[0,S_{j+1}]$ satisfying
\eqref{a:eq:jeterror}. We have from \eqref{a:eq:initial},
\begin{equation}\label{geo:eq:initialestimate}
\norm{\rho_{j+1}(0)-\rho_j(0)}_{H^m}
\leq C_m\bar{\mu}_j^{C_m}
\left(\mathrm e^{-\Gamma_j}\lambda_j^{m-1}+\lambda_j^{-2}\right),
\qquad m\in\mathbb N_0.
\end{equation}
Since
\[
X_j(t,0)=X_{j+1}(t,0)=0,\qquad
\phi_{j,0}(0)=0,\qquad \chi_j(0)=f_{\delta_j}'(0)=1,
\]
Translated by $t_j$, Corollary~\ref{geo:cor:exactretention} yields
\begin{equation}\label{geo:eq:centernewsolution}
\begin{aligned}
\sup_{t_{j+1}\leq t\leq S_{j+1}}\Bigl(
&|\nabla_x\rho_{j+1}(t,0)-h_{j+1}n_{j+1}|
+\norm{\nabla_xu_{j+1}(t,0)-h_{j+1}m(n_{j+1})\otimes n_{j+1}}\\
&+\norm{\nabla_x^2\pi_{j+1}(t,0)+h_{j+1}\cos\theta_{j+1}\,
 n_{j+1}\otimes n_{j+1}}\Bigr)
\leq Ch_j\theta_j^{-2}+\lambda_j^{-1/2}.
\end{aligned}
\end{equation}
Since $h_{j+1}\theta_{j+1}^6=Kh_j\theta_j^{-2}$,
\eqref{geo:eq:Eallocation} and $K\geq8C/\eps_*$ imply
\[
Ch_j\theta_j^{-2}+\lambda_j^{-1/2}
\leq\left(\frac{C}{K}+\frac{\eps_*}{8}\right)
h_{j+1}\theta_{j+1}^6
\leq\frac{\eps_*}{4}h_{j+1}\theta_{j+1}^6.
\]

\medskip\noindent\emph{Global pressure.}
By \eqref{a:eq:principaljets}, the principal Hessian increment is
\[
H_{\lambda_j}(t,y)=-A_j\chi_j(p_j)_2p_j\otimes p_j\,
f_{\delta_j}'(\lambda_j\phi_{j,0}(y)),
\]
extended by zero outside $\supp\chi_j$. The amplitude estimates,
$\norm{f_{\delta_j}'}_\infty\leq1$, and
$f_{\delta_j}'\geq-C\delta_j$ implies
\[
H_{\lambda_j}\leq
\begin{cases}
2\mathrm e h_{j+1}\mathrm e^{-\Gamma_j}I,&0\leq t\leq t_j,\\
C\delta_jh_{j+1}I,&t_j\leq t\leq S_{j+1},
\end{cases}
\]
using $(p_j)_2\geq0$ on the second interval. Combining this with
\eqref{a:eq:jeterror} and composing with $X_j^{-1}$ proves
\eqref{geo:eq:Hrecursion}.

\medskip\noindent\emph{The new Gevrey data.}
Estimates~\eqref{a:eq:maphigh} and \eqref{a:eq:fieldhigh},
followed by \eqref{geo:eq:normconversion}, imply the estimates in
\eqref{geo:eq:Ddata} for the deformation matrices and physical gradients
composed with $X_{j+1}$. At $t=0$, $X_{j+1}=\id$, and the periodic
Poincar\'e inequality yields
\[
\norm{D^n(\rho_{j+1}(0))}_{H^8}
\leq C\norm{D^n(\nabla_y\rho_{j+1}(0))}_{H^7},
\qquad n\geq0,
\]
since $\rho_{j+1}(0)$ has zero mean. Thus \eqref{geo:eq:Ddata} holds
with a constant $D_{j+1}\geq2$ satisfying
\begin{equation}\label{geo:eq:Dnext}
D_{j+1}\leq(\lambda_j\bar{\mu}_j)^C.
\end{equation}

\medskip\noindent\emph{The next material neighborhood.}
Denote $Y_j=X_j^{-1}\circ X_{j+1}$. The order-zero estimate in
\eqref{a:eq:maphigh} and $\bar{\mu}_j\geq D_j$ imply
\[
Y_j(t,0)=0,\qquad
D_yY_j=((D_yX_j)^{-1}\circ Y_j)D_yX_{j+1},\qquad
\sup_t\norm{D_yY_j}_\infty\leq\bar{\mu}_j^C.
\]
Estimate~\eqref{a:eq:fieldhigh} with $n=1$ yields
\[
\sup_t\Bigl(\norm{D_y((\nabla_x\rho_{j+1})\circ X_{j+1})}_\infty
+\norm{D_y((\nabla_xu_{j+1})\circ X_{j+1})}_\infty
+\norm{D_y((\nabla_x^2\pi_{j+1})\circ X_{j+1})}_\infty\Bigr)
\leq(\lambda_j\bar{\mu}_j)^C.
\]
For a fixed exponent $C$ large enough for these estimates and a sufficiently small
 $c>0$, choose
\begin{equation}\label{geo:eq:rchoice}
R_{j+1}=c\min\left\{R_j\bar{\mu}_j^{-C},\,
\eps_*h_{j+1}\theta_{j+1}^6(\lambda_j\bar{\mu}_j)^{-C},\,
\frac1{1+D_{j+1}}\right\}.
\end{equation}
For $0\leq t\leq S_{j+1}$ and $|y|\leq4R_{j+1}$, we have
\[
|Y_j(t,y)|\leq\bar{\mu}_j^C|y|\leq4\bar{\mu}_j^CR_{j+1}\leq R_j,
\]
which verifies \eqref{geo:eq:nestedtubes}. Applying the mean-value estimate, the variation
of the new gradient triple from its central value is at most
\[
4R_{j+1}(\lambda_j\bar{\mu}_j)^C
\leq\frac{\eps_*}{4}h_{j+1}\theta_{j+1}^6.
\]
Together with \eqref{geo:eq:centernewsolution}, this verifies
\eqref{geo:eq:IHgeometry} with error at most
$\eps_*h_{j+1}\theta_{j+1}^6/2$. The images $X_{j+1}(t,B_{4R_{j+1}})$ remain in a fixed physical coordinate chart, since
\[
\sup_{t,\,|y|\leq4R_{j+1}}|X_{j+1}(t,y)|
\leq4D_{j+1}R_{j+1}\leq4c.
\]
Finally, it follows from \eqref{geo:eq:rchoice} 
\begin{equation}\label{geo:eq:radiuscost}
R_{j+1}^{-1}\leq C\left(
\bar{\mu}_j^CR_j^{-1}+
\frac{(\lambda_j\bar{\mu}_j)^C}{\eps_*h_{j+1}\theta_{j+1}^6}
+1+D_{j+1}\right)
\end{equation}
which is independent of time.
\end{proof}

We remark that adding the increment in \eqref{geo:eq:Hrecursion} to $B_j$ defines
$B_{j+1}$; we shall prove $B_{j+1}\leq1/2$ in Subsection~\ref{geo:sec:budgets}.

\subsection{Uniform choice of parameters}\label{geo:sec:budgets}
We verify the hypotheses of Proposition~\ref{prop:step} simultaneously
with the construction. The constants below are fixed before the angle $\theta_0$.

\subsubsection{The global pressure bound}
Following \eqref{geo:eq:scalars}, we have
\begin{equation}\label{geo:eq:hformula}
\theta_j=\theta_0^{2^j},\qquad
h_j=h_0K^j(\theta_0/\theta_j)^{14}.
\end{equation}
Fix $C$ from \eqref{geo:eq:Hrecursion} and choose $d>0$ satisying
Lemma~\ref{amp:lem:profile} with $2Cd\leq1/12$. Then
\eqref{geo:eq:frequencies} yields
\[
\sum_{j\geq0}C\delta_jh_{j+1}
=Cd\sum_{j\geq0}\frac{2^{-j}}{h_{j+1}}\leq2Cd\leq\frac1{12}.
\]
For sufficiently small $\theta_0$, \eqref{geo:eq:hformula} and
\eqref{geo:eq:widthratio} imply
\[
\begin{aligned}
\log h_{j+1}
&\leq2^j(28|\log\theta_0|+\log K)\leq\frac{2^j}{2\theta_0},\\
\Gamma_j&\geq\theta_j^{-1}\geq2^j\theta_0^{-1},\qquad
\theta_j^{-1/2}\geq2^j\theta_0^{-1/2}.
\end{aligned}
\]
Since $2^j\geq j+1$, decreasing $\theta_0$ further yields
\begin{equation}\label{geo:eq:smallseries}
\begin{aligned}
\sum_{j\geq0}h_{j+1}\mathrm e^{-\Gamma_j}
&\leq\frac{\mathrm e^{-1/(2\theta_0)}}{1-\mathrm e^{-1/(2\theta_0)}}
\leq\frac1{24\mathrm e},\\
\sum_{j\geq0}\lambda_j^{-1/2}
&\leq\frac{\mathrm e^{-\theta_0^{-1/2}/2}}{1-\mathrm e^{-\theta_0^{-1/2}/2}}
\leq\frac1{12}.
\end{aligned}
\end{equation}
Define
\[
B_0=\frac14,\qquad
B_{j+1}=B_j+2\mathrm e h_{j+1}\mathrm e^{-\Gamma_j}
+C\delta_jh_{j+1}+\lambda_j^{-1/2}.
\]
It follows 
\begin{equation}\label{geo:eq:uniformH}
\frac14\leq B_j\leq\frac14+\frac3{12}=\frac12,
\qquad j\geq0.
\end{equation}
Thus \eqref{geo:eq:Hrecursion} implies \eqref{geo:eq:IHH} at every step of the construction.

\subsubsection{The frequency threshold}
We apply the fixed-power product in the proof of Proposition~\ref{prop:step}
to define $\bar{\mu}_j$, so that $D_j,R_j^{-1}\leq\bar{\mu}_j$.
Constants may depend on $\eps_*,K,d,d_0$, but not on $j$ or $\theta_0$.
Once step $j$ is complete,
$h_{j+1}\theta_{j+1}^6=Kh_j\theta_j^{-2}\geq1$ and
\eqref{geo:eq:Dnext}, \eqref{geo:eq:radiuscost}, \eqref{geo:eq:hformula},
and \eqref{geo:eq:frequencies} imply
\[
\begin{aligned}
D_{j+1}+R_{j+1}^{-1}&\leq C(\lambda_j\bar{\mu}_j)^C,\\
\log h_{j+2}+\log\delta_{j+1}^{-1}
&\leq C\bigl(j+1+|\log\theta_{j+1}|\bigr).
\end{aligned}
\]
It follows
\begin{equation}\label{geo:eq:Precursion}
\log\bar{\mu}_{j+1}\leq C_{\rm r}\left(\log\bar{\mu}_j+\theta_j^{-1/2}
+j+1+|\log\theta_{j+1}|\right),
\end{equation}
with some $C_{\rm r}\geq1$.

Denote $a_j=\theta_j^{-1/2}=\log\lambda_j$, and hence $a_{j+1}=a_j^2$.
The estimate \eqref{geo:eq:seedP} and
\[
h_1=K\theta_0^{-18},\qquad
\delta_0^{-1}=d^{-1}K^2\theta_0^{-36}
\]
imply $\log\bar{\mu}_0\leq C_{\rm s}\log a_0$, with some $C_{\rm s}\geq1$.
Choose $C_*=8C_{\rm r}$ and take $a_0$ sufficiently large that
\[
\begin{split}
&a_0\geq QC_*,\qquad
\mathrm e^{a_0}\geq\underline{\lambda},\\
&1+4\log a_0\leq a_0,\qquad QC_{\rm s}\log a_0\leq a_0,\\
&\mathrm e^{-a_0/2}\leq\frac{\eps_*K}{8}a_0^{12}.
\end{split}
\]
Then we have
\[
j+1+|\log\theta_{j+1}|=j+1+4\log a_j
\leq2^j(1+4\log a_0)\leq2^ja_0\leq a_j.
\]
It follows from \eqref{geo:eq:Precursion} that
\[
\log\bar{\mu}_1\leq3C_{\rm r}a_0\leq C_*a_0,
\]
and, if $j\geq1$ and $\log\bar{\mu}_j\leq C_*a_{j-1}$, we have
\[
\log\bar{\mu}_{j+1}
\leq C_{\rm r}(C_*a_{j-1}+2a_j)
\leq\left(\frac{C_{\rm r}C_*}{a_0}+2C_{\rm r}\right)a_j
\leq\frac{C_*}{2}a_j.
\]
Moreover, $\Gamma_j\geq a_j^2$ implies
\[
\frac{\log\bar{\mu}_j}{a_j}\leq\frac{C_*}{a_{j-1}}\to0,
\qquad \frac{a_j}{\Gamma_j}\leq\frac1{a_j}\to0.
\]
Thus the sequence constructed below satisfies
\begin{equation}\label{geo:eq:Pwindow}
\log\bar{\mu}_j\leq C_*\theta_{j-1}^{-1/2}, \quad j\geq1,\qquad
\log\bar{\mu}_j=o(\theta_j^{-1/2}),\qquad
\theta_j^{-1/2}=o(\Gamma_j).
\end{equation}

The choice of $a_0$ implies $Q\log\bar{\mu}_0\leq a_0$, while for
$j\geq1$,
\[
Q\log\bar{\mu}_j\leq QC_*a_{j-1}\leq a_{j-1}^2=a_j.
\]
Hence, at each stage of the construction, we have
\begin{equation}\label{geo:eq:allthresholds}
\theta_j^{-1/2}\geq Q\log\bar{\mu}_j\quad\text{and}\quad
\lambda_j\geq\underline{\lambda}, \qquad j\geq0.
\end{equation}
Since $h_{j+1}\theta_{j+1}^6=Kh_j\theta_j^{-2}\geq Ka_0^{12}$, it follows
\[
\lambda_j^{-1/2}\leq\mathrm e^{-a_0/2}
\leq\frac{\eps_*K}{8}a_0^{12}
\leq\frac{\eps_*}{8}h_{j+1}\theta_{j+1}^6,
\]
which verifies \eqref{geo:eq:Eallocation}. These conditions above and the
preceding pressure estimates hold for sufficiently small $\theta_0$.
We fix one such angle $\theta_0=\arctan(N_0^{-1})$ before the induction.

\begin{proposition}\label{prop:sequence}
For one sufficiently small angle $\theta_0$, there is a sequence
of odd periodic solutions on $[0,S_j]$ satisfying
\eqref{geo:eq:IHgeometry}, \eqref{geo:eq:IHH}, and \eqref{geo:eq:Ddata},
and also \eqref{geo:eq:nestedtubes} for $j\geq1$. Their data constants
obey \eqref{geo:eq:Pwindow}, and their initial increments satisfy
\eqref{geo:eq:initialestimate}.
\end{proposition}
\begin{proof}
The stationary solution satisfies the induction hypotheses and
$Q\log\bar{\mu}_0\leq a_0$. At each stage,
\eqref{geo:eq:allthresholds} and \eqref{geo:eq:Eallocation} allow us to apply
Proposition~\ref{prop:step} to construct the next solution and its initial
increment. Equations \eqref{geo:eq:uniformH} and \eqref{geo:eq:Precursion}
propagate the pressure estimate and the estimate
$\log\bar{\mu}_{j+1}\leq C_*a_j$, while Proposition~\ref{prop:step}
propagates the geometric, Gevrey estimates and nesting hypotheses. Induction yields
the sequence, and the two ratios above prove the error assertions in
\eqref{geo:eq:Pwindow}.
\end{proof}

\subsubsection{Convergence of the initial data and gradient growth}
Fix $m\in\mathbb N_0$. Applying \eqref{geo:eq:Pwindow} and \eqref{geo:eq:widthratio}, we obtain $\log\bar{\mu}_j\leq C_*a_j^{1/2}$ and
$\Gamma_j\geq a_j^2$ for $j\geq1$. Hence \eqref{geo:eq:initialestimate} yields
\[
\norm{\rho_{j+1}(0)-\rho_j(0)}_{H^m}
\leq\mathrm e^{-\Gamma_j/2}+\mathrm e^{-a_j}\leq2\mathrm e^{-a_j}, \quad j\geq j_m.
\]
As $a_j\geq2^ja_0$, we obtain
\begin{equation}\label{geo:eq:summableinitial}
\sum_{j\geq0}\norm{\rho_{j+1}(0)-\rho_j(0)}_{H^m}<\infty
\qquad\text{for every fixed }m\in\mathbb N_0.
\end{equation}

By \eqref{geo:eq:widthratio}, we have
\begin{equation}\label{geo:eq:finiteT}
\begin{split}
&\frac{h_j\theta_j^2}{h_{j+1}\theta_{j+1}^2}
=K^{-1}\theta_j^{12}\leq\frac14,\\
&\sum_{j\geq0}(t_{j+1}-t_j)
\leq2\sum_{j\geq0}\frac1{h_j\theta_j^2}
\leq\frac8{3h_0\theta_0^2}<\infty.
\end{split}
\end{equation}
Denote $T_\infty=\lim_jt_j$. The estimate above also yields
\[
0<T_\infty-t_j\leq2\sum_{k\geq j}\frac1{h_k\theta_k^2}
\leq\frac8{3h_j\theta_j^2}<\frac4{h_j\theta_j^2}=S_j-t_j.
\]
Since $(h_j\theta_j^2)^{-1}\to0$, we have
$t_j,S_j\to T_\infty$ and $0<T_\infty<S_j\leq S_0<1$.

Finally, \eqref{geo:eq:hformula} and $\sin\theta_j\geq\theta_j/2$ imply
\[
h_j\sin\theta_j\geq\tfrac12h_0\theta_0^{14}K^j\theta_j^{-13}\to\infty.
\]
Since $\eps_*\theta_j^6\leq\tfrac12\min(1,\sin\theta_j)$,
\eqref{geo:eq:IHgeometry} at the origin implies
\begin{equation}\label{geo:eq:gradientgrowth}
\begin{aligned}
|\nabla_x\rho_j(t_j,0)|&\geq h_j(1-\eps_*\theta_j^6)
\geq\tfrac12h_j\to\infty,\\
\norm{\nabla_xu_j(t_j,0)}&\geq h_j(\sin\theta_j-\eps_*\theta_j^6)
\geq\tfrac12h_j\sin\theta_j\to\infty.
\end{aligned}
\end{equation}

\subsection{The limiting argument}\label{sec:limit}
\begin{proof}[Proof of Theorem~\ref{thm:blowup}]
Following \eqref{geo:eq:summableinitial}, the odd initial densities converge in $H^m$ to
\begin{equation}\label{eq:limitingdatum}
\rho_{\rm in}=\lim_{j\to\infty}\rho_j(0)\in C^\infty(\T^2), \forall m>0.
\end{equation}
Let $(\rho,u,\pi)$ be the corresponding maximal smooth solution, with
lifespan $T_*=T_*(\rho_{\rm in})>0$ given by Lemma~\ref{lem:local}.

\medskip\noindent\emph{Convergence on fixed intervals.}
Fix $T<\min(T_*,T_\infty)$ and denote
$L_T=C(1+\sup_{[0,T]}\norm\rho_{H^4})$.
On a bootstrap interval where $\norm{\rho_n-\rho}_{H^3}\leq1$,
\eqref{eq:differenceestimate} and Gronwall's inequality imply
\begin{equation}\label{eq:limitfixed}
\sup_{0\leq t\leq T}\norm{\rho_n(t)-\rho(t)}_{H^3}
\leq\mathrm e^{L_TT}\norm{\rho_n(0)-\rho_{\rm in}}_{H^3}\to0.
\end{equation}
For sufficiently large $n$, the right-hand side is less than $1/2$,
hence the bootstrap closes on $[0,T]$.

\medskip\noindent\emph{Estimates on earlier intervals.}
Denote
\[
\varepsilon_j=6\mathrm e h_{j+1}\mathrm e^{-\Gamma_j}
+\lambda_j^{-1/2},\qquad \sum_j\varepsilon_j<\infty
\]
by \eqref{geo:eq:smallseries}. Each principal tensor in
\eqref{a:eq:principaljets} has norm at most $A_j$, since
$|m(p_j)|,\chi_j,|f_{\delta_j}'|\leq1$. Thus
\eqref{geo:eq:historybound} and \eqref{a:eq:jeterror} imply
\begin{equation}\label{eq:exacthistory}
\begin{aligned}
\sup_{0\leq t\leq t_j}\Bigl(
&\norm{\nabla_x(\rho_{j+1}-\rho_j)(t)}_\infty
+\norm{\nabla_x(u_{j+1}-u_j)(t)}_\infty\\
&+\norm{\nabla_x^2(\pi_{j+1}-\pi_j)(t)}_\infty\Bigr)
\leq\varepsilon_j.
\end{aligned}
\end{equation}
For fixed $J$, telescoping on $[0,t_J]$ yields
\[
\begin{aligned}
&\quad\sup_{n\geq J}\sup_{0\leq t\leq t_J}
\bigl(\norm{\nabla_x\rho_n(t)}_\infty+\norm{\nabla_xu_n(t)}_\infty\bigr)\\
&\leq\sup_{0\leq t\leq t_J}
\bigl(\norm{\nabla_x\rho_J(t)}_\infty+\norm{\nabla_xu_J(t)}_\infty\bigr)
+\sum_{k=J}^\infty\varepsilon_k=: \mathcal M_J<\infty.
\end{aligned}
\]
By \eqref{geo:eq:summableinitial} and \eqref{eq:tametransport}, we have for any
integer $m\geq3$,
\begin{equation}\label{eq:stagehistorysobolev}
\sup_{n\geq J}\sup_{0\leq t\leq t_J}\norm{\rho_n(t)}_{H^m}
\leq\left(\sup_n\norm{\rho_n(0)}_{H^m}\right)
\mathrm e^{C_m\mathcal M_Jt_J}=:K_{J,m}<\infty.
\end{equation}

\medskip\noindent\emph{The lifespan.}
If $T_*<T_\infty$, choose $J$ with $T_*<t_J$. Passing to the limit in
\eqref{eq:stagehistorysobolev} on each $[0,T]$ with $T<T_*$, applying
\eqref{eq:limitfixed}, yields
\[
\sup_{t<T_*}\norm{\rho(t)}_{H^3}\leq K_{J,3}<\infty,
\]
which contradicts maximality from Lemma~\ref{lem:local}.
Hence
\begin{equation}\label{eq:lifespanlower}
T_*\geq T_\infty.
\end{equation}

If $T_*>T_\infty$, then $\rho$ is smooth on $[0,T_\infty]$, and
\eqref{eq:differenceestimate} yields on $[0,t_j]$
\begin{equation}\label{eq:limitcomparison}
\frac{d}{dt}\norm{\rho_j-\rho}_{H^3}
\leq C\left(\sup_{0\leq s\leq T_\infty}\norm{\rho(s)}_{H^4}\right)
\norm{\rho_j-\rho}_{H^3}+C\norm{\rho_j-\rho}_{H^3}^2.
\end{equation}
With $L=C(1+\sup_{[0,T_\infty]}\norm\rho_{H^4})$, the same bootstrap yields
\begin{equation}\label{eq:limitstrong}
\sup_{0\leq t\leq t_j}\norm{\rho_j(t)-\rho(t)}_{H^3}
\leq\mathrm e^{LT_\infty}\norm{\rho_j(0)-\rho_{\rm in}}_{H^3}\to0.
\end{equation}
Applying $H^3(\T^2)\hookrightarrow C^1(\T^2)$ and
\eqref{geo:eq:gradientgrowth} yields
\[
\tfrac12h_j\leq|\nabla_x\rho_j(t_j,0)|
\leq\sup_{0\leq t\leq T_\infty}\norm{\nabla_x\rho(t)}_\infty+o(1),
\]
contradicting $h_j\to\infty$. Therefore
\begin{equation}\label{eq:finitebreakdown}
0<T_*(\rho_{\rm in})=T_\infty<1.
\end{equation}

\medskip\noindent\emph{Smooth convergence and blowup at the origin.}
Fix $T<T_\infty$ and $J$ with $T\leq t_J$. For integers $s>m>3$,
\eqref{eq:limitfixed}, \eqref{eq:stagehistorysobolev} and interpolation yield
\[
\sup_{[0,T]}\norm{\rho_n-\rho}_{H^m}
\leq C_{J,s,T}
\left(\sup_{[0,T]}\norm{\rho_n-\rho}_{H^3}\right)^{(s-m)/(s-3)}\to0.
\]
Together with \eqref{eq:limitfixed}, this yields
\begin{equation}\label{eq:limitsmoothfixed}
\begin{split}
\rho_n\to\rho\quad\text{in }C([0,T];H^m),
\qquad T<T_\infty,\ m\in\mathbb N_0,\\
u_n=\cT\rho_n\to u=\cT\rho\quad\text{in }C([0,T];H^m),
\qquad T<T_\infty,\ m\in\mathbb N_0
\end{split}
\end{equation}
Applying \eqref{eq:exacthistory} and passing to the limit yields
\begin{equation}\label{eq:limithistorytail}
\begin{aligned}
\sup_{0\leq t\leq t_J}\Bigl(
&\norm{\nabla_x(\rho-\rho_J)(t)}_\infty
+\norm{\nabla_x(u-u_J)(t)}_\infty\Bigr)
\leq\sum_{k=J}^\infty\varepsilon_k.
\end{aligned}
\end{equation}

On $[t_j,t_{j+1}]$, \eqref{geo:eq:IHgeometry} and
$(m(n_j)\otimes n_j)_{11}=\sin^2\theta_j\cos\theta_j$ imply
\[
\begin{aligned}
\partial_{x_2}\rho_j(t,0)&\geq h_j(\cos\theta_j-\eps_*\theta_j^6),\\
\partial_{x_1}(u_j)_1(t,0)
&\geq h_j(\sin^2\theta_j\cos\theta_j-\eps_*\theta_j^6).
\end{aligned}
\]
At the origin, we have $\phi_{j,0}=0$ and $\chi_j=f_{\delta_j}'(0)=1$.
In view of \eqref{geo:eq:activationestimates}, we have
\[
(g_{\lambda_j})_2=A_j(p_j)_2\geq0,\qquad
(M_{\lambda_j})_{11}=A_j(p_j)_1^2(p_j)_2\geq0.
\]
Applying \eqref{a:eq:jeterror} and \eqref{eq:limithistorytail} with
$J=j+1$ yields for $t\in[t_j,t_{j+1}]$,
\begin{equation}\label{eq:origincomponentestimates}
\begin{aligned}
\partial_{x_2}\rho(t,0)
&\geq h_j(\cos\theta_j-\eps_*\theta_j^6)
-\lambda_j^{-1/2}-\sum_{k=j+1}^\infty\varepsilon_k,\\
\partial_{x_1}u_1(t,0)
&\geq h_j(\sin^2\theta_j\cos\theta_j-\eps_*\theta_j^6)
-\lambda_j^{-1/2}-\sum_{k=j+1}^\infty\varepsilon_k.
\end{aligned}
\end{equation}
For an absolute $c>0$, we have
\[
\cos\theta_j-\eps_*\theta_j^6\geq c,\qquad
\sin^2\theta_j\cos\theta_j-\eps_*\theta_j^6\geq c\theta_j^2,
\qquad
h_j\theta_j^2=h_0\theta_0^{14}K^j\theta_j^{-12}\to\infty.
\]
Since $t_j\to T_\infty$ and the error tails in
\eqref{eq:origincomponentestimates} vanish, we have
\begin{equation}\label{eq:originblowup}
\partial_{x_2}\rho(t,0)\to+\infty,\qquad
\partial_{x_1}u_1(t,0)\to+\infty
\quad\text{as }t\to T_\infty.
\end{equation}

For the endpoint energy identities, \eqref{a:eq:mapzero} implies
\[
\sup_{0\leq t\leq T_\infty}\norm{\rho_{j+1}(t)-\rho_j(t)}_\infty
\leq\norm{\rho_{j+1}(0)-\rho_j(0)}_\infty
+C\bar{\mu}_j^C\lambda_j^{-1},
\]
where we used $\sup_j\norm{\nabla\rho_j(0)}_\infty<\infty$ from
\eqref{geo:eq:summableinitial}. Thanks to \eqref{geo:eq:Pwindow},
$\bar{\mu}_j^C\lambda_j^{-1}\leq\lambda_j^{-1/2}$ for all sufficiently
large $j$. Thus $\rho_j$ converges
uniformly on $[0,T_\infty]\times\T^2$ to $\rho$.
Define $u(T_\infty)=\cT\rho(T_\infty)$. The $L^2$ multiplier bound
and \eqref{eq:energy} then imply
\[
\norm{\rho(t)}_2=\norm{\rho_{\rm in}}_2,\qquad
\norm{u(t)}_2\leq\norm{\rho_{\rm in}}_2,
\qquad 0\leq t\leq T_*.
\]
\end{proof}

The embedding $H^3\hookrightarrow C^1$ and \eqref{eq:originblowup} imply
\[
\lim_{t\to T_*}\norm{\rho(t)}_{H^3}=\infty.
\]
Moreover, \eqref{eq:tametransport} yields
\[
\norm{\rho(t)}_{H^3}\leq\norm{\rho_{\rm in}}_{H^3}
\exp\!\left(C\int_0^t\left(\norm{\nabla_x\rho(s)}_\infty+
\norm{\nabla_xu(s)}_\infty\right)\,ds\right),
\]
and hence
\[
\int_0^{T_*}\left(\norm{\nabla_x\rho(t)}_\infty+
\norm{\nabla_xu(t)}_\infty\right)\,dt=\infty.
\]

\section*{Acknowledgement}
Publicly available ChatGPT was used to understand the finite-time blowup construction of Open AI \cite{OpenAIEuler} for the Euler equations, 
and to assist with some of the computations in the manuscript.

\end{document}